\documentclass[11pt]{article} 
\usepackage{a4,amssymb, amsmath, amsthm, dsfont}
\usepackage{mathrsfs,amsfonts}
\usepackage{graphics,color}
\usepackage{epsfig}
\usepackage{subfig}
\usepackage{bbm}
\usepackage{color}
\usepackage{mathtools}
\usepackage{nicefrac}
\usepackage{float}

\newtheorem{thm}{Theorem}[section]
\newtheorem{cor}[thm]{Corollary}
\newtheorem{lem}[thm]{Lemma}

\newtheorem{defn}[thm]{Definition}
\newtheorem{rem}[thm]{Remark}

\numberwithin{equation}{section}

\def\be{\begin{equation}}
	\def\ee{\end{equation}}
\def\bse{\begin{subequations}}
	\def\ese{\end{subequations}}
\def\bge{\begin{eqnarray}}
	\def\bgee{\begin{eqnarray*}}
		\def\ege{\end{eqnarray}}
	\def\egee{\end{eqnarray*}}

\renewcommand{\theequation}{\arabic{section}.\arabic{equation}}

\begin{document}
	
\title{Wellposedness of a Nonlinear Parabolic-Dispersive Coupled System Modelling MEMS}

\author{
Heiko Gimperlein\thanks{Leopold-Franzens-Universit\"{a}t Innsbruck, Engineering Mathematics, Technikerstra\ss e 13, 6020 Innsbruck, Austria} \thanks{Department of Mathematical, Physical and Computer Sciences,
University of Parma, 43124, Parma, Italy} \and  Runan He\thanks{Institut f\"{u}r Mathematik, Martin-Luther-Universit\"{a}t Halle-Wittenberg, 06120 Halle (Saale), Germany} \and  Andrew A.~Lacey\thanks{Maxwell Institute for Mathematical Sciences and Department of Mathematics, Heriot-Watt University, Edinburgh, EH14 4AS, United Kingdom}}
\date{}
\maketitle

\providecommand{\keywords}[1]{{\noindent \textit{Key words:}} #1}
\providecommand{\msc}[1]{{\textit{MSC classes:}} #1}

\abstract{\noindent In this paper we study the local wellposedness of the solution to a non-linear parabolic-dispersive coupled system which models a Micro-Electro-Mechanical System (MEMS). A simple electrostatically actuated MEMS capacitor device has two parallel plates separated by a gas-filled thin gap.  The nonlinear parabolic-dispersive coupled system modelling the device consists of a quasilinear parabolic equation for the gas pressure and a semilinear plate equation for gap width. We show the local-in-time existence of strict solutions for the system, by combining a local-in-time existence result for the dispersive equation, H\"{o}lder continuous dependence of its solution on that of the parabolic equation, and then local-in-time existence for a resulting abstract parabolic problem. Semigroup approaches are vital for both main parts of the problem.\\
}

\\
\msc{35M33 (primary), 35G61, 35D30, 74F10 (secondary)}
\vskip 0.8cm

\section{Introduction}

In this article, we study finite-time existence, uniqueness and regularity of the solution to the following nonlinear parabolic-dispersive coupled system, which models an idealized electrostatically actuated MEMS device, accounting for elasticity of the plate:
\begin{subequations}\label{cp2}
\begin{equation}\label{cp2-1-1}
\frac{\partial\left(wu\right)}{\partial t}=\nabla\cdot\left(w^3u\nabla u\right)\quad x\in\Omega,\ t\geq 0;
\end{equation}
\begin{equation}\label{cp2-1-2}
\frac{\partial^2w}{\partial t^2}=\Delta w-\Delta^2w-\frac{\beta_F}{w^2}+\beta_p(u-1),\quad x\in\Omega,\ t\geq 0;
\end{equation}
\begin{equation}\label{cp2-2}
u(x,0)=u_0(x),\ w(x,0)=w_0(x),\ \frac{\partial w}{\partial t}(x,0)=v_0(x),\quad x\in {\Omega};
\end{equation}
\begin{equation}\label{cp2-3}
u(x,t)=\theta_1, \ w(x,t)=\theta_2, \ \Delta w(x,t)=0,\quad x\in\partial\Omega,\ t\geq 0.
\end{equation}
\end{subequations}
The unknown functions $u(x,t)$ and $w(x,t)$ correspond, respectively, to gas pressure and gap width, $\Omega\subset\mathds{R}^n$ is a bounded and open region with smooth boundary $\partial\Omega$, $n=1,\ 2$; $\beta_F$, $\beta_p$, $\theta_1$, $\theta_2>0$ are given constants; $u_0=u_0(x)$, $v_0=v_0(x)$ and $w_0=w_0(x)$ are given functions. We shall prove the following wellposedness result which applies for short time:

\begin{thm}\label{coupled system}
Let $\alpha \in (0,\frac{1}{2})$, $u_0 \in H^{2+\alpha}(\Omega)$, $v_0 \in H^2(\Omega)$ and $w_0 \in H^4(\Omega)$, compatible with the boundary conditions and such that $u_0$, $w_0>0$. The initial-boundary value problem  \eqref{cp2} admits a unique strict solution $(u, w)$ {on a time interval $[0, T)$}  and
\[u\in C^{\alpha+1}\left([0, {T}); L^2(\Omega)\right)\cap C^\alpha\left([0, {T}); H^2(\Omega)\right),\]
\[w\in C^{2}\left([0, {T}); L^2(\Omega)\right)\cap C^{1}([0, {T}); H^2(\Omega))\cap C\left([0, {T}); H^4(\Omega)\right).\]
\end{thm}
\begin{rem}\label{4thRem}
	a) Global-in-time solutions are not expected for general initial data, because quenching singularities with $\displaystyle\inf_{x\in\Omega} w(x,t)\to 0$ may develop as $t\to {T}$, for a finite time ${T}<\infty$ (see \cite{gimperlein2022quenching}). \\
	b) Even for smooth data the Sobolev regularity of $w$ is limited because the right hand side in \eqref{cp2-1-2} does not vanish at the boundary $\partial\Omega$: Indeed, using Fourier series one can explicitly solve the linear dispersive equation
	\be\label{Fourier-LHP}
	\frac{\partial^2 w}{\partial t^2}+\frac{\partial^4 w}{\partial x^4}=f(x,t),\quad (x, t)\in (0,1)\times[0, \infty);
	\ee
	with homogeneous initial conditions $w(x, 0) = \frac{\partial w}{\partial t}(x, 0)=0$ ($x \in (0,1)$) and homogenous boundary conditions $w(x, t)=0$,
	$\frac{\partial^2 w}{\partial x^2}(x, t)=0$ ($x\in \{0,1\}$, $t \in (0,\infty)$). For $f=1$ one finds that the solution $w(t)\in H^{\frac{9}{2}-\epsilon}(0, 1)$
	for every $\epsilon>0$ and $t\in[0,\infty)$, but $w(t)\not \in H^{\frac{9}{2}}(0, 1)$. We therefore do not expect higher integer-order Sobolev regularity for $w$ in Theorem \ref{coupled system}.
\end{rem}
\begin{figure}[t]
	\begin{center}
		{\includegraphics[width=0.5\textwidth]{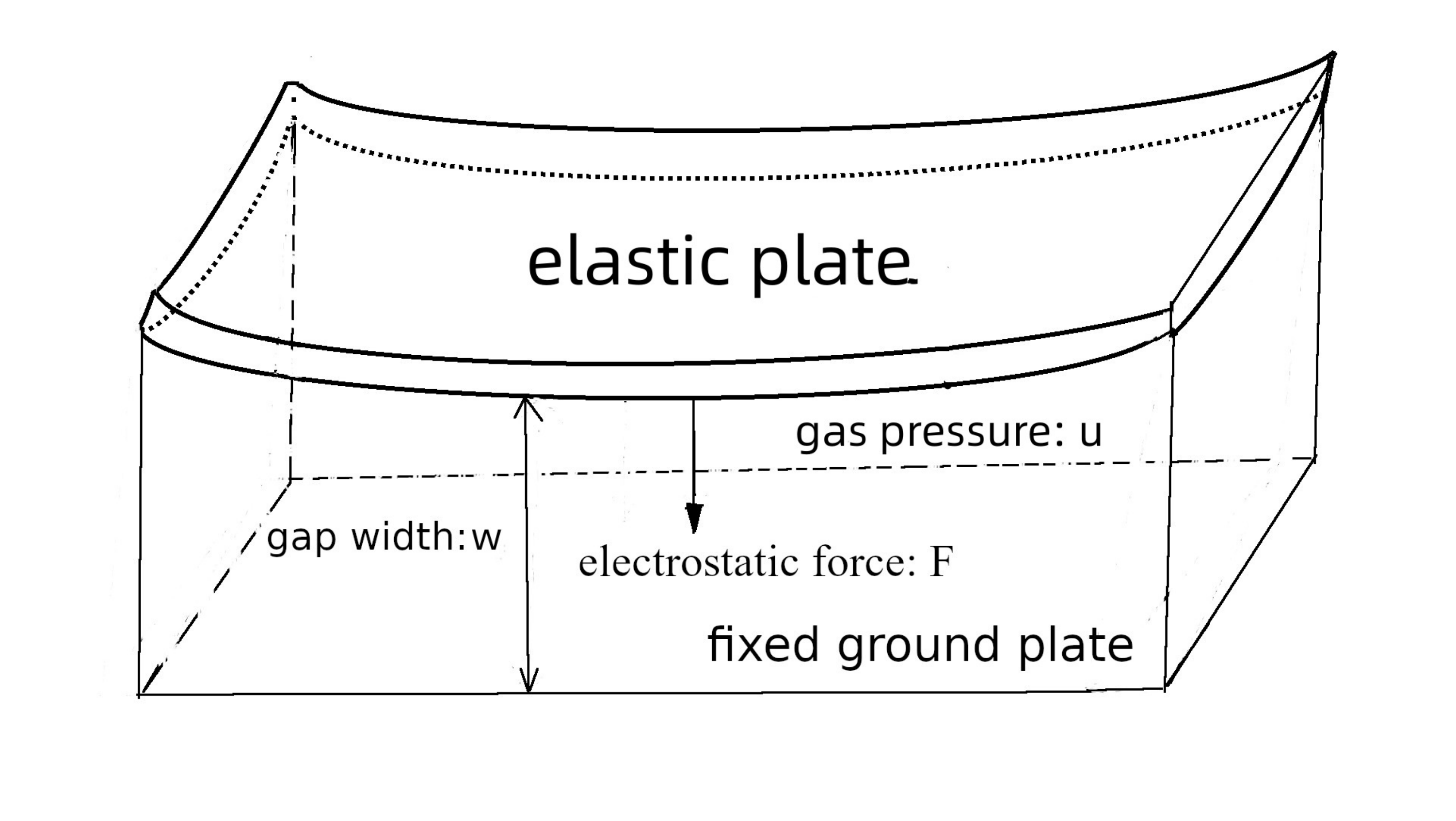}}
		\vspace{-0.6cm}\caption{Idealized electrostatically actuated MEMS capacitor.}
		\label{fig0}
	\end{center}
\vspace{-0.4cm}
\end{figure}
Our proof of Theorem \ref{coupled system} relies on the techniques for quasilinear parabolic equations developed, for example, by Amann, Arendt, Lunardi and Sinestrari \cite{AH, AW, LA0, LA1, LS1, LS2,  SE}. Semigroup methods of this kind have become a powerful tool for MEMS-related models defined by a single equation or by an elliptic-parabolic coupled system, see the recent survey \cite{LW}. We here combine such parabolic techniques for the quasilinear Reynolds' equation \eqref{cp2-1-1} with semigroup techniques for the semilinear fourth-order equation \eqref{cp2-1-2}.\\

{In the recent work \cite{gimperlein2023wellposedness} the authors showed the wellposedness of a related, simpler model in which \eqref{cp2-1-1} is replaced by a linear, elliptic equation for the gas pressure $u$. Physically, the simpler model in \cite{gimperlein2023wellposedness} replaces the dynamics of $u$ by a quasi-static approximation which can apply in
 limiting cases; here we consider the more accurate, fuller model to
give a good representation of the physical system outlined in Figure~\ref{fig0}.} {Technically, the analysis of both coupled systems relies on a delicate combination of the techniques available for the constituent equations. In the recent work \cite{gimperlein2023wellposedness} the well-known analysis of linear, elliptic equations allowed us to reduce the simpler model to a perturbed semilinear dispersive equation for the gap width $w$, which is studied using strongly continuous semigroup techniques  for such equations. For the realistic model \eqref{cp2-1-1}-\eqref{cp2-3} considered here,  similarly complete information is not available for the quasilinear, degenerately parabolic equation \eqref{cp2-1-1}. Nevertheless, by refining the analysis of \eqref{cp2-1-2}, we here reduce the coupled system to an abstract quasilinear, degenerately parabolic equation for the gas pressure $u$, for which we are able to show wellposedness.  Analytic semigroup techniques allow us to study this quasilinear, degenerately parabolic equation. }\\

{The model \eqref{cp2} gives the behaviour of a basic electrically actuated MEMS (Micro-Electro-Mechanical Systems) capacitor (see, e.g.,  \cite{JS}). This device contains two conducting plates which are close and parallel to each other when the device is uncharged and at equilibrium. We take, more generally, a fixed potential difference to be applied; this potential difference acts across the plates and the MEMS device forms a capacitor. The two plates lie inside a sealed box also containing a gas, with pressure substantially below atmospheric but not a perfect vacuum. The gas gives a small resistance to the motion of the upper,  plate, which is taken to be flexible but pinned around its edges. The other, lower, plate is assumed perfectly rigid and flat. See Figure \ref{fig0}. Breakdown of the device can occur through a pull-in instability, when the two plates touch, the physical phenomenon described by quenching.
	
	Eqn.~\eqref{cp2-1-1} is a compressible form of the standard Reynolds' equation for the pressure, $u$, in the gap between the plates, where the local gap width is $w$ (see, for example, {\cite{OO}}), and the gas is assumed to behave ideally and isothermally, so that its density can be taken to be proportional to pressure; {this  contrasts with the incompressible, liquid-like representation in \cite{gimperlein2023wellposedness}.}

	The upper electrode of the capacitor moves as a thin elastic plate, so that it behaves according to a dynamic plate equation balancing  the inertial term on the left-hand side of {\eqref{cp2-1-2}} with: 
\begin{itemize}
\item tension terms applied across the plate leading to the usual Laplacian (the first term on the right); 
\item a biharmonic term modelling linear elasticity (the second term on the right) {\cite{HKO}}; 
\item an electrostatic force attracting the upper plate towards the lower (the third term on the right) -- strength of this force per unit area is given by the local electric field strength times the surface charge density, the latter itself being proportional to the former, while this, the field strength, is inversely proportional to the gap width $w$; 
\item net upward gas pressure acting on the plate -- pressure in the gap acting up and constant ambient pressure acting down  (the final term on the right). 
\end{itemize}
For more details see {\cite{BP}, \cite{JS}}.
	
	The terms in the equations have been scaled to obtain unit coefficients in {\eqref{cp2-1-1}}. Without loss of generality, we have taken, for simplicity,  various coefficients of terms in {\eqref{cp2-1-2}} also to be one. This does not affect our analysis of the problem.}

In another, forthcoming paper, we study the limiting case where the movement of the upper plate is dominated by its tension, so that elastic effects are negligible. This approximation leads to a wave equation for the gap width $w$, instead of the dispersive equation \eqref{cp2-1-2}.

We review various previous models for electrostatic MEMS devices, some taking the form of a single equation others a coupled system. We also review literature which studies these models, both numerically and analytically, to obtain qualitative behaviour.

As the elastic plate in a MEMS device is fabricated at a micro-scale, the electrostatic force becomes relatively large so that it is the key force causing the bending of the plate when the device operates. The electrostatic force is inversely proportional to the square of the gap width between the two plates(see  \cite{EGG}, \cite{BP}, Sec.~3.4 of \cite{JS}), so that the distributed transverse load is the electrostatic force per area is $F_d$, 
\begin{equation}\label{ElecForce}
	F_d=-\frac{\beta_F}{w^2},\quad\text{where}\ \beta_F\ \text{is an electrostatic coefficient}.
\end{equation}
Hence, the static deflection of charged elastic plates in electrostatic actuators can be represented by a nonlinear elliptic equation
\begin{equation}\label{BendingEqn3}
	-\Delta w+\beta_e\Delta^2 w=-\frac{\beta_F}{w^2}.
\end{equation}
Lin et al. \cite{LY} study the existence, construction, approximation, and behaviour of classical and singular solutions to equation \eqref{BendingEqn3}. Other such problems can be found in references  \cite{EG}, \cite{VM}. 

From Chapter 12 in book \cite{EGG}, there is a value $\beta^*\in(0, \infty)$  such that for $0<\beta_F<\beta^*$ there exists at least one weak solution to \eqref{BendingEqn3}, while no solution exists for $\beta_F>\beta^*$.

To more fully model the behaviour of the plate, we consider the momentum of the plate as it is deformed, the elastic nature of plate, damping forces, and the electrostatic force between two plates and get an equation of motion
\begin{equation}\label{MotEqn}
	\epsilon^2\frac{\partial^2w}{\partial t^2}+\frac{\partial w}{\partial t}-\Delta w+\beta_e\Delta^2 w=-\frac{\beta_F}{w^2}.
\end{equation}
Here the gap width  $w=w(x, t)$ depends on time $t$ and point $x$ on the surface of movable plate, $\frac{\partial w}{\partial t}$ is a damping term,
\[\epsilon^2 =\frac{\text{inertial\ coefficient}}{\text{damping\ coefficient}},\]
$\beta_e$ accounts for the relative importance of tension and flexural rigidity in the elastic plate,  $\beta_F$ is proportional to the square of the applied voltage.

Considering the equation defined on a bounded domain of $\mathbb{R}^n$, $1\leq n\leq 3$, Guo, \cite{Guo}, finds that when a voltage -- represented here mathematically by $\beta_F$ -- is applied, the elastic plate deflects towards the ground plate, and quenching may occur when $\beta_F$ exceeds the critical value $\beta^*$ for the time-independent problem \eqref{BendingEqn3}. 
Guo \cite{Guo} shows that there exists a $\beta_{F_{1}}\in(0,\ \beta^*]$ such that for $0\leq \beta_F<\beta_{F_{1}}$, the solution of an initial boundary value problem for \eqref{MotEqn} globally exists. Under some further technical hypotheses, in this case the solution exponentially converges to a regular steady state. For $\beta_F > \beta^*$, the solution quenches at finite time.


Recent publications only study the compressible version \eqref{cp2-1-2} of the standard Reynolds' equation numerically, and this can be seen in the works \cite{BY, Bl, St, SRBUCP}. In particular, Bao et al.~\cite{BYSF} study the squeeze film damping with small amplitude deflections and linearize the nonlinear Reynolds' equation (i.e.~\eqref{cp2-1-2}) around the equilibrium position. The resulting equation is regarded as a form of the heat equation and it is possible to find  analytical solutions for this.

See  the survey article \cite{LW} for a discussion of a wider class of models arising in the description of MEMS. {We are not aware of any rigorous results for MEMS models which take into account both the dynamics of the gas and the elastodynamics of the plate.}\\


The plan of the paper is as follows: In Section \ref{Sec:prerequisite}, we introduce notation, the relevant function spaces and some of their basic properties. We also introduce the mild solution and strict solution for the general evolution equation and their existence results, and show some Lipschitz continuity estimates. In Section \ref{4th-order problem}, we use a solution strategy for the system \eqref{cp2-1-1}, \eqref{cp2-1-2} based on decoupling the equations for the {gap-width} $w$ and the pressure $u$. We first consider the semilinear fourth-order equation \eqref{cp2-1-2} for the deflection $w$ with an arbitrarily given pressure $u$ and use semigroup techniques for \eqref{cp2-1-2} to show that the local wellposedness of \eqref{cp2-1-2}. While the regularity theory  of dispersive equations has been of much recent interest, we here require detailed properties of the solution operator $u\longmapsto w(u)$ in order to analyse the nonlinear Reynolds' equation \eqref{cp2-1-1} with abstract coefficients involving $w(u)$. For example, we prove appropriate H\"{o}lder continuity of the solution operator $u\longmapsto w(u)$ in Section \ref{SecSolnOp}. In Section \ref{section4}, we investigate the local wellposedness of \eqref{cp2-1-1} for $u$ with abstract coefficients involving $w(u)$ by using techniques for quasilinear parabolic equations.


\subsection{Outline}

{Note that system \eqref{cp2} can be written, as long as $w>0$ (no quenching occurs), as a coupled system in the form}
\bse\label{cp3}
\be\label{cp3-1}
\frac{\partial u}{\partial t}=\frac{1}{w}\nabla\cdot\left(w^3u\nabla u\right)-\frac{v}{w}u,\quad x\in\Omega,\ t\geq 0;
\ee
\be\label{cp3-2}
\frac{\partial v}{\partial t}=\Delta w-\Delta^2w-\frac{\beta_F}{w^2}+\beta_p(u-1),\quad x\in\Omega,\ t\geq 0;
\ee
\be\label{cp3-3}
\frac{\partial w}{\partial t}=v,\quad x\in\Omega,\ t\geq 0;
\ee
\ese
with the initial values $u(x,0)=u_0(x)$, $v(x,0)=v_0(x)$, $w(x,0)=w_0(x)$, $x\in {\Omega}$ and boundary values $u(x,t)=\theta_1$,  $w(x,t)=\theta_2$,  $\Delta w(x,t)=0$, $x\in\partial\Omega$, $t\geq 0$, where the initial values are compatible with the boundary conditions, i.e. $u_0(x)=\theta_1$, $w_0(x)=\theta_2$ and $\Delta w_0(x)=0$ for all $x\in\partial\Omega$, moreover, $u_0\in H^{2+\sigma}(\Omega)$  with $\sigma\in (0, 1)$, $v_0\in H^2(\Omega)\cap H_0^1(\Omega)$ and $w_0\in H^4(\Omega)$, then look for a unique strict solution $(u, v, w)$ of the  coupled system \eqref{cp3} for short time.
Section \ref{4th-order problem}  shows that there exists a unique solution $(v, w)$ of the sub-system \eqref{cp3-2}, \eqref{cp3-3} for arbitrarily given but appropriately regular $u$, initial values $v(x,0)=v_0(x)$, $w(x,0)=w_0(x)$, $x\in {\Omega}$ and boundary values $w(x,t)=\theta_2$,  $\Delta w(x,t)=0$, $x\in\partial\Omega,\ t\geq 0$, then Section \ref{SecSolnOp} establishes relevant properties of solution operators $u\longmapsto v=v(u)$, $u\longmapsto w=w(u)$ for short time $T$ such as: 
\begin{thm}\label{t1}
The solution operator
\[W:\ C\left(\left[0,  T\right], B_{H^2}(u_0,  r)\right)\longrightarrow C\left(\left[0,  T\right], B_{L^2}(v_0,  r)\times B_{H^2}(w_0,  r)\right)\] \[u\longmapsto W(u)=\left(v, w\right)=\left(v(u), w(u)\right)\]
is Lipschitz continuous with respect to $u$, i.e.
\begin{align}
\left\|W({u}_1)-W({u}_2)\right\|_{C\left([0, T]; L^{2}(\Omega)\times H^{2}(\Omega)\right)}\leq L_{W}\|{u}_1-{u}_2\|_{C\left([0, T]; H^2(\Omega)\right)},
\end{align}
where $r>0$ is sufficiently small, $L_{W}>0$ is a Lipschitz constant,
\[B_{H^2}(U,  r)=\left\{f\in H^2(\Omega):\ f|_{\partial\Omega}=U|_{\partial\Omega},\ \|f-U\|_{H^2(\Omega)}\leq r\right\},\] 
\[ B_{L^2}(V,  r)=\left\{f\in L^2(\Omega):\ \|f-V\|_{L^2(\Omega)}\leq r\right\}.\]
\end{thm}
\begin{cor}\label{c2}
For $u\in C\left([0, T]; B_{H^2}(u_0,  r)\right)$ and a small radius $r>0$,  the Fr\'{e}chet derivative $W'(u)$ of $W(u)$, given by
\[{W}{'}(u): C\left([0,T]; H^2(\Omega)\cap H_0^1(\Omega)\right)\longrightarrow C\left([0,T]; L^2(\Omega)\times \left\{H^2(\Omega)\cap H_0^1(\Omega)\right\}\right),\]
\[q\longmapsto W'(u)q=\left(v'(u)q, w'(u)q\right)\]
is Lipschitz continuous with respect to $u$, i.e. for $\left\|q\right\|_{C\left([0,T]; H^2(\Omega)\right)}\leq 1$,
\be
\left\|{W}'(u_1)q-{W}'(u_2)q\right\|_{C\left([0, T]; L^{2}(\Omega)\times H^{2}(\Omega)\right)}\leq L_{F}\left\|{u}_1-{u}_2\right\|_ {C\left([0, T]; H^2(\Omega)\right)}.
\ee
Here $L_F$ is a Lipschitz constant.
\end{cor}
\begin{cor}\label{c3}
If $r>0$ is small and $u\in C^\alpha\left([0, T]; B_{H^2}(u_0,r)\right)$, setting $\tilde{u}_0=u_0-\theta_1$,
then there exists a Lipschitz constant $L_M>0$, such that
\begin{align}
\sup_{0\leq t<t+h\leq T}\left\|[{W}'(u)q](t+h)-[{W}'(u)q](t)\right\|_{L^{2}(\Omega)\times H^{2}(\Omega)}\leq&h^\alpha L_{M}\left\|q\right\|_{C\left([0,T]; H^2(\Omega)\right)}\notag\\
+&h^\alpha TL_{M}\left\|q\right\|_{C^\alpha\left([0, T]; H^2(\Omega)\right)}\notag
\end{align}
holds for all $q\in C^\alpha\left([0, T]; B_{H^2}(\tilde{u}_0,r)\right)$.
\end{cor}
Section \ref{section4} shows  an existence result for the coupled system \eqref{cp2}. The strategy of proof is to reformulate the system \eqref{cp2} as the quasilinear parabolic equation with abstract coefficients involving $v(u)$ and $w(u)$
\bse\label{1QPE}
\be\label{1QPE-1}
\frac{\partial u}{\partial t}=\frac{1}{w(u)}\nabla\cdot\left([w(u)]^3u\nabla u\right)-\frac{v(u)}{w(u)}u,\quad (x,t)\in\Omega\times(0, T),
\ee
\be\label{1QPE-2}
u(x,0)=u_0(x),\quad x\in\Omega,\quad u(x,t)=\theta_1,\quad (x,t)\in\partial\Omega\times[0, T],
\ee
\ese
and then show the solution of \eqref{1QPE} exists as long as $\left(v(u), w(u)\right)\in C\left([0, T]; B_{L^2}(v_0, r)\times B_{H^2}(w_0, r)\right)$ for small $r>0$ and $T>0$ by using a contraction mapping argument.

We set $\tilde{u}=u-\theta_1$, where $\tilde{u}(t): \Omega\longrightarrow \mathds{R} $, $x\longmapsto [\tilde{u}(t)](x)=\tilde{u}(x,t) $, $\forall\ t\in[0, T]$,
\[F(\tilde{u})=\frac{1}{w(\tilde{u}+\theta_1)}\nabla\cdot\left([w(\tilde{u}+\theta_1)]^3(\tilde{u}+\theta_1)\nabla \tilde{u}\right)-\frac{v(\tilde{u}+\theta_1)}{w(\tilde{u}+\theta_1)}(\tilde{u}+\theta_1),\]
and start the argument with the definition of linearization $\mathcal{P}^*$ of $F(\tilde{u})$ around $\tilde{u}_0$,
\[\mathcal{P}^*:\ D\left(\mathcal{P}^*\right)\subseteq H^2(\Omega)\cap H_0^1(\Omega)\longrightarrow L^2(\Omega),\] \[\mathcal{P}^*\psi=\frac{1}{w_0}\nabla\cdot\left\{w_0^3u_0\nabla\psi+\left(w_0^3\nabla u_0\right)\psi\right\}-\frac{v_0}{w_0}\psi,\]
then we show that the operator $\mathcal{P}^*$ generates an analytic semigroup $\left\{e^{t\mathcal{P}^*}:\ t\geq 0\right\}$ and rewrite \eqref{1QPE} in the form of
\be\label{2QPE}
\tilde{u}'(t)=\mathcal{P}^*\tilde{u}(t)+[F(\tilde{u})](t)-\mathcal{P}^*\tilde{u}(t),\quad t\in(0, T),\quad \tilde{u}(0)=\tilde{u}_0.
\ee
In order to prove the existence result for the nonlinear problem \eqref{2QPE}, we shall need the following H\"{o}lder result which is deduced from Theorem \ref{t1}, Corollary \ref{c2} and Corollary \ref{c3}.
\begin{lem}\label{1}
If $\tilde{u},\ q\in C^\alpha\left([0, T]; B_{H^2}\left(\tilde{u}_0, r\right)\right)$, then there exist postive constants $L_A$ and $L_B$, such that for all $ 0\leq t<t+h\leq T$,
\be
\left\|\left[F(\tilde{u})\right](t+h)-\left[F(\tilde{u})\right](t)\right\|_{L^2(\Omega)}\leq \left\{[\tilde{u}+\theta_1]_{C^\alpha\left(\left[0, T\right]; H^2(\Omega)\right)}+L_U\right\}L_Ah^\alpha,
\ee
\begin{align}
&\left\|\left[F'(\tilde{u})q\right](t+h)-\left[F'(\tilde{u})q\right](t)-\mathcal{P}^*\left[q(t+h)-q(t)\right]\right\|_{L^2(\Omega)}\notag\\
\leq& h^\alpha T^\alpha L_{B}\left\|q\right\|_{C^\alpha([0, T]; H^2(\Omega))}+h^\alpha T^\alpha L_{B}\left\|\tilde{u}+\theta_1\right\|_{C^\alpha([0, T]; H^2(\Omega))}\left\|q\right\|_{C^\alpha([0, T]; H^2(\Omega))}\notag\\
+&h^\alpha L_B\left\|q\right\|_{C([0, T]; H^2(\Omega))}+h^\alpha L_B\left\|\tilde{u}+\theta_1\right\|_{C^\alpha([0, T]; H^2(\Omega))}\left\|q\right\|_{C([0, T]; H^2(\Omega))}.
\end{align}
\end{lem}
We are going to show that the existence of a unique strict solution of \eqref{2QPE} by proving there exists $T_{\max}>0$, such that the nonlinear map $\Gamma$ defined by
\be\label{0}
\Gamma(\tilde{u}(t))=e^{t\mathcal{P}^*}\tilde{u}_0+\displaystyle\int_0^te^{(t-s)\mathcal{P}^*}\left\{[F(\tilde{u})](s)-\mathcal{P}^*\tilde{u}(s)\right\}ds,\quad t\in [0, T],
\ee
is a contractive map and has a unique fixed point in $C^\alpha\left([0, T]; H^2(\Omega)\cap H_0^1(\Omega)\right)$ for $T\in[0, T_{\max})$ and $\alpha\in(0, 1)$. To prove the assertion, we define
\[Y=\left\{\tilde{u}\in C^\alpha\left([0, T]; H^2(\Omega)\cap H_0^1(\Omega)\right):\ \tilde{u}(0)=\tilde{u}_0,\ \left\|\tilde{u}(\cdot)-\tilde{u}_0\right\|_{C^\alpha([0, T]; H^2(\Omega))}\leq r\right\},\]
with small $r>0$, by using the H\"{o}lder results in Lemma \ref{1}, we deduce that, there is $T_{\max}>0$, such that, for $T\in(0, T_{\max})$,  small $r>0$ and $\alpha\in(0, 1)$, $\Gamma$ is a contractive map which maps $Y$ to itself, i.e. $\tilde{u}_1$, $\tilde{u}_2\in Y$,
\[\left\|\Gamma\tilde{u}_1-\Gamma\tilde{u}_2\right\|_{C^\alpha\left([0, T]; H^2(\Omega)\right)}\leq\frac{1}{2}\left\|\tilde{u}_1-\tilde{u}_2\right\|_{C^\alpha\left([0, T]; H^2(\Omega)\right)},\quad\Gamma(Y)\subseteq Y.\]
By the Banach fixed point theorem, we conclude the existence of a unique fixed point in $Y$ which is a unique strict solution of \eqref{2QPE} belonging to \[C^\alpha\left([0, T]; H^2(\Omega)\cap H_0^1(\Omega)\right)\cap C^{\alpha+1}\left([0, T]; L^2(\Omega)\right),\]
by the regularity results of the evolution equation of parabolic type from \cite{AH, LS1, SE}.

Combining the existence and regularity results from Section \ref{4th-order problem}  with the existence of a unique strict solution of \eqref{2QPE}, we conclude the proof of Theorem \ref{coupled system}.


\section{Preliminaries} \label{Sec:prerequisite}


In this section,  we first formulate  some auxiliary results which will be useful in the proof of the main theorem, with the proofs of Lemma \ref{estimates},  Lemma \ref{Lip-G-Lem} and Lemma \ref{Lip-nonlinearity} being found in Appendix \ref{AppA}. We then state a general existence result for  evolution equations and the regularity in time without proof. 


\subsection{Notations}\label{SecNotations} 

Recall that $\beta_F$, $\beta_p$, $\theta_1$ and $\theta_2$ given positive constants. Let $\Omega$ be an open and bounded subset of $\mathds{R}^n$ with smooth boundary $\partial\Omega$, $n=1,\ 2$. Denote by $C=C(\Omega)$ a positive constant which may vary from line to line below but only depends on $\Omega$. 
\begin{defn}
Denote by $X$  a Banach space, with norm $\|\cdot\|_{X}$, $k\in\mathbb{N}$ and $T\in(0, \infty)$. $\mathcal{B}(X)$ denotes the space of bounded linear operators on $X$. In the following, we shall be particularly interested in $X=L^2(\Omega)$, $L^\infty(\Omega)$,  $H^k(\Omega)$, etc. The space $\mathcal{B}([0,T];X)$ consists of all  measurable, almost everywhere bounded functions $u: [0,T]\longrightarrow X$, $ t\longmapsto u(t) $, with norm $\|u\|_{\mathcal{B}([0,T];X)}=\sup_{t\in[0, T]}\|u(t)\|_{X}$. If $X$ is a function space as above, we write $u(t): \Omega\longrightarrow \mathds{R}$ with $x\longmapsto [u(t)](x)=u(x,t)$. The closed subspace of continuous functions is denoted by $C([0,T];X)$, and
\[C^k([0,T];X)=\left\{u:[0,T]\rightarrow X:  \tfrac{d^ju}{dt^j}\in C([0,T];X), j\in[0, k]\right\},  \|u\|_{C^k([0,T];X)}=\hspace*{-0.2cm}\sup_{t\in[0, T]}\sum_{j=0}^{k}\left\|\tfrac{d^ju(t)}{dt^j}\right\|_{X}.\]
The definition extends to non-integer order $k+\alpha$, $\alpha\in(0, 1)$, by setting 
\[C^\alpha([0,T];X)=\left\{u:[0,T]\rightarrow X: [u]_{C^\alpha([0, T];X)}=\sup_{0\leq t<t+h\leq T}\tfrac{\|u(t+h)-u(t)\|_X}{|h|^\alpha}<\infty\right\},\]
\[\|u\|_{C^{\alpha}([0,T];X)}=\|u\|_{C([0,T];X)}+[u]_{C^{\alpha}([0, T];X)}.\]
\[C^{\alpha+k}([0,T];X)=\left\{u\in C^k([0,T];X): \quad \tfrac{d^ku}{dt^k}\in C^{\alpha}([0,T];X)\right\},\]
\[\|u\|_{C^{\alpha+k}([0,T];X)}=\|u\|_{C^k([0,T];X)}+\left[\tfrac{d^ku}{dt^k}\right]_{C^{\alpha}([0, T];X)}.\]
\noindent Note that  $C\left([0, T]; B_{L^2}\left(V, r\right)\right)=\Big\{v\in C\left([0, T]; L^2\left(\Omega\right)\right):  \displaystyle\sup_{t\in[0, T]}\left\|v(t)-V\right\|_{H^2(\Omega)}\leq r\Big\}$, $V\in L^2(\Omega)$,  $C\left([0, T]; B_{H^2}\left(U, r\right)\right)=\Big\{u\in C\left([0, T]; H^2\left(\Omega\right)\right):\ u|_{\partial\Omega}=U|_{\partial\Omega},\ \displaystyle\sup_{t\in[0, T]}\left\|u(t)-U\right\|_{H^2(\Omega)}\leq r\Big\}$, $C^\alpha\left([0, T]; B_{H^2}\left(U, r\right)\right)=\Big\{u\in C^\alpha\left([0, T]; H^2\left(\Omega\right)\right):\ u|_{\partial\Omega}=U|_{\partial\Omega},\ \displaystyle\sup_{t\in[0, T]}\left\|u(t)-U\right\|_{H^2(\Omega)}\leq r\Big\}$\\ with $U\in H^2(\Omega)$. 

If $\mathcal{P}: D(\mathcal{P}) \subset X \to X$ is an unbounded linear operator which generates an analytic semigroup $e^{\mathcal{P}t}$, we define intermediate space $D_{\mathcal{P}}(\alpha, \infty)$ as follows:
\[D_{\mathcal{P}}(\alpha, \infty)=\left\{v\in X:\ \|v\|_{\alpha}=\sup_{t>0}\|t^{1-\alpha}\mathcal{P}e^{\mathcal{P}t}v\|_{X}<\infty\right\}.\]
It is a Banach space with respect to the norm
$\|v\|_{D_{\mathcal{P}}(\alpha, \infty)}=\|v\|_{X}+\|v\|_{\alpha}$.
Its closed subspace
$D_{\mathcal{P}}(\alpha)=\left\{v\in X:\ \lim_{t\rightarrow 0}t^{1-\alpha}\mathcal{P}e^{\mathcal{P}t}v=0\right\}$
inherits the norm of $D_{\mathcal{P}}(\alpha, \infty)$. 
\end{defn}
Our main results on the wellposedness of the semilinear dispersive equation \eqref{cp2-1-2} will be shown by constructing a Picard iteration in the complete metric space  $\mathcal{Z}(T)$, given by
\begin{align}
\mathcal{Z}(T):&=\bigg\{(\tilde{v},\tilde{w})\in C\left([0,T]; L^2(\Omega)\times H^2(\Omega)\right):\ \left(\tilde{v}(0), \tilde{w}(0)\right)=\left(\tilde{v}_0, \tilde{w}_0\right),\quad  \tilde{w}|_{\partial\Omega}=\tilde{w}_0|_{\partial\Omega},\notag\\
&\sup_{t\in[0, T]}\|(\tilde{v}(t)-\tilde{v}_0,\tilde{w}(t)-\tilde{w}_0)\|_{L^2(\Omega)\times H^{2}(\Omega)}\leq r\bigg\},\quad \text{where}\quad \tilde{v}_0=v_0,\quad \tilde{w}_0=w_0-\theta_2\label{ini-NBD}.
\end{align}
\subsection{Useful Estimates}
The estimates in this subsection are proven in Appendix \ref{AppA}.
\begin{lem}\label{estimates}
There exists a constant $C=C(\Omega)>0$, such that for all
\be\label{r-def-I}
r\in\left(0,\ \frac{\kappa}{2C}\right),
\ee
$w\in C\left([0, T]; B_{H^2}\left(w_0,r\right)\right)$ has the lower bound such as
\be\label{w-lower-bound}
w(t)\geq\frac{\kappa}{2},\quad \forall\ t\in[0, T].
\ee
Moreover, for all $w_1$, $w_2\in C\left([0, T]; B_{H^2}(w_0, r)\right)$, there exist  positive constants $C_k$, $k=1,\ 2,\ 3$, depending on $\Omega$, $\kappa$ and $\left\|w_0\right\|_{H^2(\Omega)}$, such that
\be\label{C-a}
\sup_{t\in[0, T]}\left\|\frac{1}{[w_1(t)]^{k}}\right\|_{H^2(\Omega)}\leq {C_1^k},\quad k=1,\ 2,\ 3,
\ee
\be\label{C-d}
\sup_{t\in[0, T]}\left\|\frac{1}{[w_1(t)]^{k}}-\frac{1}{[w_2(t)]^{k}}\right\|_{H^2(\Omega)}\leq C_k\sup_{t\in[0, T]}\left\|w_1(t)-w_2(t)\right\|_{H^2(\Omega)},\ k=2,\ 3.
\ee
\end{lem}
\begin{lem}\label{Lip-G-Lem}
The nonlinear operator $G$, defined  by
\[G:\ C\left([0, T]; B_{H^2}\left(\tilde{w}_0, r\right)\right)\longrightarrow C([0, T]; H^2(\Omega)),\quad \tilde{w}\longmapsto  G(\tilde{w})\]
\[[G(\tilde{w})](t)=G(\tilde{w}(t))=-\frac{\beta_F}{[\tilde{w}(t)+\theta_2]^2}+\beta_p(\theta_1-1),\]
has the following properties:
\be \label{Holdercontinuous}
\sup_{0\leq t<t+h\leq T}\left\|[G(\tilde{w})](t+h)-[G(\tilde{w})](t)\right\|_{H^{2}(\Omega)}\leq L_G\sup_{0\leq t<t+h\leq T}\left\|\tilde{w}(t+h)-\tilde{w}(t)\right\|_{H^{2}(\Omega)},
\ee
\be\label{Lip-G}
\sup_{t\in[0, T]}\left\|[G(\tilde{w}_1)](t)-[G(\tilde{w}_2)](t)\right\|_{H^{2}(\Omega)}\leq L_G\sup_{t\in[0, T]}\left\|\tilde{w}_1(t)-\tilde{w}_2(t)\right\|_{H^{2}(\Omega)},
\ee
\be\label{Lip-G-1}
\sup_{t\in[0, T]}\left\|[G(\tilde{w}_1)](t)-G(\tilde{w}_0)\right\|_{H^{2}(\Omega)}\leq L_G r.
\ee
Here $L_G=L_G\left(\Omega,\ \kappa,\ \|w_0\|_{H^{2}(\Omega)},\ \beta_F\right)$ is a constant.\\
Furthermore, the Fr\'{e}chet derivative $G'(\tilde{w})$ of $G(\tilde{w})$ on  $\tilde{w}\in C\left([0, T]; B_{H^2}\left(\tilde{w}_0, r\right)\right)$, defined by
\[G'\left(\tilde{w}\right):\ C\left([0, T]; H^2(\Omega)\cap H_0^1(\Omega)\right)\longrightarrow C\left([0, T]; H^2(\Omega)\cap H^1_0(\Omega)\right),\quad  q\longmapsto  G'\left(\tilde{w}\right)q,\]
\[\left[G'\left(\tilde{w}\right)q\right](t)=\left[G'\left(\tilde{w}(t)\right)\right]q(t)=\frac{2\beta_F}{\left(\tilde{w}(t)+\theta_2\right)^3}q(t),\]
satisfies
\be\label{Lip-G-2}
\sup_{t\in[0, T]}\left\|\left[G'\left(\tilde{w}\right)q\right](t)\right\|_{H^2(\Omega)}\leq L_G\sup_{t\in[0, T]}\left\|q(t)\right\|_{H^2(\Omega)},
\ee
and  $G'(\tilde{w}(t)): \left\{H^2(\Omega)\cap H_0^1(\Omega)\right\}\longrightarrow \left\{H^2(\Omega)\cap H_0^1(\Omega)\right\}$ satisfies
\be\label{uniformly-continuous-Fre-G}
\lim_{h\rightarrow0}\sup_{\begin{smallmatrix}0\leq t\leq t+ h\leq T\\ 0\leq\tau\leq 1\end{smallmatrix}}\left\|G'\left(\tilde{w}(t)+\tau\left[\tilde{w}(t+h)-\tilde{w}(t)\right]\right)-G'\left(\tilde{w}(t)\right)\right\|_{\mathcal{B}\left(H^2(\Omega)\right)}=0.
\ee
\end{lem}
\begin{lem}\label{Lip-nonlinearity}
Assume the operators $u\longmapsto v(u)$ and $u\longmapsto w(u)$, respectively given by
\[u\longmapsto v(u):\quad  C\left([0, T]; B_{H^2}\left(u_0, r\right)\right)\longrightarrow C\left([0, T]; B_{L^2}\left(v_0, r\right)\right),\] \[u\longmapsto w(u):\quad  C\left([0, T]; B_{H^2}\left(u_0, r\right)\right)\longrightarrow C\left([0, T]; B_{H^2}\left(v_0, r\right)\right),\]
satisfy, for all $u_1,\ u_2\in C\left([0, T]; B_{H^2}\left(u_0, r\right)\right)$,
\begin{equation*}
\sup_{t\in[0, T]}\|[v(u_1)](t)-[v(u_2)](t)\|_{L^2(\Omega)}\leq L_W \sup_{t\in[0, T]}\|{u}_1(t)-{u}_2(t)\|_{H^2(\Omega)},
\end{equation*}
\begin{equation*}
\sup_{t\in[0, T]}\|[w(u_1)](t)-[w(u_2)](t)\|_{H^2(\Omega)}\leq L_W \sup_{t\in[0, T]}\|{u}_1(t)-{u}_2(t)\|_{H^2(\Omega)}.
\end{equation*}
Then $f(u)$, defined by
\begin{equation*}
u\longmapsto f(u): C\left([0, T]; B_{H^2}\left(u_0, r\right)\right)\longrightarrow C([0, T]; L^2(\Omega)),\quad f(u)=\frac{1}{w(u)}\nabla\cdot\left([w(u)]^3u\nabla u\right)-\frac{v(u)}{w(u)}u,
\end{equation*}
is Lipschitz continuous in $u$,
\be\label{nonlinearity-Lip-0}
\sup_{t\in[0, T]}\|[f(u_1)](t)-[f(u_2)](t)\|_{L^2(\Omega)}\leq L_e \sup_{t\in[0, T]}\|{u}_1(t)-{u}_2(t)\|_{H^2(\Omega)}.
\ee
Here $L_W$ and $L_e$ are  Lipschitz constants, and $L_e$ depends on $L_W$, $\Omega$, $\kappa$, $\left\|u_0\right\|_{H^2(\Omega)}$, $\left\|v_0\right\|_{L^2(\Omega)}$, $\left\|w_0\right\|_{H^2(\Omega)}$.
\end{lem}

\subsection{Properties of Evolution Equations}
We recall standard notions and results for abstract evolution equations.
\begin{defn}
Let $\mathfrak{X}$ be a Banach space, $\mathcal{A}: D(\mathcal{A}) \subset \mathfrak{X} \to \mathfrak{X}$ a linear, unbounded operator which  generates a strongly continuous semigroup ($C_0$-semigroup) $\{T(t): t\geq 0\}$. Further, let $T\in(0, \infty)$, $\mathcal{G}\in C([0, T]; \mathfrak{X})$ and $\Phi_0\in \mathfrak{X}$. A function $\Phi$ is called a mild solution of the  inhomogeneous evolution equation
\be\label{IEE}
\Phi'(t)=\mathcal{A}\Phi(t)+\mathcal{G}(t), \quad t\in [0, T],\quad \Phi(0)=\Phi_0,
\ee
if $\Phi\in C([0, T]; \mathfrak{X})$ is given by the integral formulation
\be\label{linear solu}
\Phi(t)=T(t)\Phi_0+\int_0^tT(t-s)\mathcal{G}(s)ds,\quad t\in[0, T].
\ee
A function $\Phi$ is said to be a strict solution of \eqref{IEE}, if $\Phi\in C([0, T]; D(\mathcal{A}))\cap C^1([0, T]; \mathfrak{X})$ is given by the integral formulation \eqref{linear solu} and satisfies \eqref{IEE}.
\end{defn}

\begin{lem}\label{IEE-S}
Let the linear operator $\mathcal{A}$ defined on a Banach space $\mathfrak{X}$ generate the $C_0$-semigroup $\{T(t): t\geq 0\}$, $T\in(0, \infty)$, and $\Phi_0\in D(\mathcal{A})$. If $\mathcal{G}\in C([0, T]; \mathfrak{X})$ and $\Phi$ is a solution of the inhomogeneous evolution equation \eqref{IEE}, then $\Phi$ is given by the integral formulation \eqref{linear solu}.

Assume either that $\mathcal{G}\in C([0, T]; D(\mathcal{A}))$ or that $\mathcal{G}\in C^1([0, T]; \mathfrak{X})$. Then the mild solution $\Phi$ defined by \eqref{linear solu} uniquely solves the inhomogeneous evolution equation \eqref{IEE} on $[0, T]$, and
\[\Phi\in C([0, T]; D(\mathcal{A}))\cap C^1([0, T]; \mathfrak{X}).\]
\end{lem}
The proof of Lemma \ref{IEE-S} is given  in \cite{schnaubelt}, Theorem 6.9.

\begin{lem}\label{time-derivative-continuity}
Let $\mathfrak{X}$ be a Banach space and $\Phi\in C([0, T]; \mathfrak{X})$ be differentiable from the right with right derivative $\Psi\in C\left([0, T]; \mathfrak{X}\right)$. Then $\Phi\in C^1\left([0, T]; \mathfrak{X}\right)$ and $\Phi'=\Psi$.
\end{lem}
The proof of Lemma \ref{time-derivative-continuity} is given in \cite{schnaubelt}, Lemma 8.9.


\section{Wellposedness of the Dispersive Equation}\label{4th-order problem}



\subsection*{Refined Analysis of the  Dispersive Equation}

Denote by $H_*^2(\Omega):=H^2(\Omega)\cap H_0^1(\Omega)$, $H_*^4(\Omega):=\left\{\chi\in H^{4}(\Omega): \ \chi|_{\partial\Omega}=\Delta \chi|_{\partial\Omega}=0\right\}$.
Take $T\in(0, \infty)$ to be specified below. We first introduce a state $\mathbf{a}=\left(a_1,\ a_2\right)$ and a state space
\be\label{state-space}
\mathfrak{X}=L^2(\Omega)\times H_*^2(\Omega)
\ee
 with its norm $\|\cdot\|_{\mathfrak{X}}=\left\|\cdot\right\|_{L^2(\Omega)\times H^2(\Omega)}$ and its scalar product
\[\langle \mathbf{a}, \mathbf{b}\rangle_{\mathfrak{X}}=\displaystyle\int_\Omega a_1\cdot{b_1}+\nabla a_2\cdot\nabla{b_2}+\Delta a_2\cdot\Delta{b_2}\ {d}x,\quad \mathbf{a}=\left(a_1, a_2\right)\in\mathfrak{X}, \ \mathbf{b}=\left(b_1, b_2\right)\in\mathfrak{X}.\]
We then define a linear operator $\mathbb{A}$ by
\begin{align}
 D(\mathbb{A}):=&\left\{\phi\in H_*^2(\Omega): \exists\ f\in L^2(\Omega),\ \forall\ \psi\in H_*^2(\Omega),\ \text{s.t.}\ \int_\Omega \nabla\phi\cdot\nabla{\psi}+\Delta\phi\cdot\Delta{\psi}dx=\int_\Omega f\cdot{\psi}dx \right\},\notag\\
\mathbb{A}\phi:=&-f,\ \text{where} \ f \ \text{is given by} \ D(\mathbb{A}), \quad \left\|\phi\right\|_{D(\mathbb{A})}:=\left\|\phi\right\|_{L^2(\Omega)}+\left\|\mathbb{A}\phi\right\|_{L^2(\Omega)}.\label{domain2}
\end{align}
It is easy to see that  $\Delta \phi|_{\partial\Omega}=0$ for all $\phi\in D(\mathbb{A})$, and from elliptic regularity theory, it follows that
\begin{equation}\label{domain3}
	D(\mathbb{A})=\left\{\chi\in H^{4}(\Omega): \ \chi|_{\partial\Omega}=\Delta \chi|_{\partial\Omega}=0\right\}=H_*^4(\Omega),\quad\left\|\chi\right\|_{D(\mathbb{A})}\simeq\|\chi\|_{H^4(\Omega)}.
\end{equation}
We further define the linear operator $\mathcal{A}$ with its domain $D(\mathcal{A})$ and its graph norm $\|\cdot\|_{D(\mathcal{A})}$ by
\bse\label{A-1}
\be\label{A-1-1}
\mathcal{A}=\begin{pmatrix}0\ &\mathbb{A}\\ 1\ &0\end{pmatrix}, \quad D(\mathcal{A})= H_*^2(\Omega)\times H_*^4(\Omega),
\ee
\be\label{A-1-2}
\|\mathbf{a}\|_{D(\mathcal{A})}:=\|\mathbf{a}\|_{\mathfrak{X}}+\|\mathcal{A}\mathbf{a}\|_{\mathfrak{X}}\simeq\|a_1\|_{H^2(\Omega)}+\|a_2\|_{H^4(\Omega)}, \quad\mathbf{a}=\left(a_1, a_2\right)\in D(\mathcal{A}).
\ee
\ese
We now consider the initial-boundary problem of  semilinear fourth-order equation \eqref{cp2-1-2} on the unknown function $w$ with an arbitrarily given but fixed $u\in C\left([0, T]; B_{H^2}\left(u_0, r\right)\right)$,   initial values
\be\label{4th-semilinear-hyperbolic-equation-2}
w(x,0)=w_0(x),\quad\frac{\partial w}{\partial t}(x,0)=v_0(x), \quad x\in\Omega,
\ee
and pinned boundary conditions
\be\label{4th-semilinear-hyperbolic-equation-3}
w(x,t)=\theta_2, \quad \Delta w(x,t)=0, \quad (x,t)\in\partial\Omega\times[0, T].
\ee
We set $\tilde{w}(x,t)=w(x,t)-\theta_2$, where $\tilde{w}(t): \Omega\longrightarrow \mathds{R}$, $x\longmapsto [\tilde{w}(t)](x)=\tilde{w}(x,t)$. Note that the operator $\mathbb{A}$, defined in  \eqref{domain2}, is a realisation of the differential expression $\Delta-\Delta^2$ from equation \eqref{cp2-1-2} for the pinned boundary conditions $\tilde{w}(x,t)=0$,  $\Delta\tilde{w}(x,t)=0$,  $x\in\partial\Omega$, $t\in[0, T]$. Using the definition of $\mathbb{A}$, we rewrite \eqref{cp2-1-2} with \eqref{4th-semilinear-hyperbolic-equation-2} and \eqref{4th-semilinear-hyperbolic-equation-3} as the equation \eqref{4th-LWE-1} for the unknown function $\tilde{w}$:
\be\label{4th-LWE-1}
\tilde{w}^{''}(t)=\mathbb{A}\tilde{w}(t)-\frac{\beta_F}{(\tilde{w}(t)+\theta_2)^2}+\beta_p(\tilde{u}(t)+\theta_1-1),\ t\in[0, T],\
\tilde{w}(0)=\tilde{w}_0,\ \tilde{w}'(0)=\tilde{v}_0,
\ee
where $\tilde{w}'$ and $\tilde{w}^{''}$ respectively denote the first and second derivative of the unknown function $\tilde{w}$ with respect to $t\in [0, T]$,  $\tilde{u}=u-\theta_1$  is given in  $C\left([0, T]; B_{H^2}\left(\tilde{u}_0, r\right)\right)$ with  $\tilde{u}_0=u_0-\theta_1$, $\tilde{v}_0=v_0$,  $\tilde{w}_0=w_0-\theta_1$ and $\tilde{u}_0(x)=\tilde{w}_0(x)=0$ for all  $x\in \partial\Omega$.
We further introduce a new time-dependent state $\Phi(t)=\left(\varphi_1(t),\ \varphi_2(t)\right)$, $t\in[0, T]$, and set
\be\label{G-2-2}
[G(\varphi_2)](t)=-\frac{\beta_F}{(\varphi_2(t)+\theta_2)^2}+\beta_p(\theta_1-1),\quad  \mathcal{G}=\left[\mathcal{G}\left(\Phi\right)\right](t)=\left(\left[G\left({\varphi}_2\right)\right](t)+\beta_p\tilde{u}(t), 0\right),
\ee
\vspace*{-0.5cm}
\be\label{ini-Phi}
\Phi_0=\left(\tilde{v}_0, \tilde{w}_0\right)\in D(\mathcal{A}).
\ee
We are going to prove that the  semilinear fourth-order equation \eqref{4th-LWE-1} has a unique strict solution by showing Lemma \ref{equivalence-1}, Lemma \ref{generator}, Theorem \ref{4th-solu-thm}, Corollary \ref{Lip-mild-solu}, Corollary \ref{Holdercontinuity} and Theorem \ref{4th-mild-solution-cor} in Appendix \ref{AppA1}. This would conclude the wellposedness of the dispersive equation \eqref{cp2-1-2}.
\begin{lem}\label{equivalence-1}
Let given $\tilde{u}\in C\left([0, T]; B_{H^2}\left(\tilde{u}_0, r\right)\right)\cap C^1\left([0, T]; L^2(\Omega)\right)$, $\mathcal{A}$ and $\mathcal{G}$ be defined by \eqref{A-1} and \eqref{G-2-2} respectively.  The semilinear fourth-order equation \eqref{4th-LWE-1} has a unique solution
\[\tilde{w}\in C^2\left([0, T]; L^2(\Omega)\right)\cap C^1\left([0, T]; H_*^2(\Omega)\right)\cap C\left([0, T]; H_*^4(\Omega)\right)\]
if and only if the semilinear  evolution equation
\be\label{4th-IEE}
\Phi'(t)=\mathcal{A}\Phi(t)+\left[\mathcal{G}\left(\Phi\right)\right](t), \ t\in [0, T],\ \Phi(0)=\Phi_0,
\ee
has a unique solution
\[\Phi\in C([0, T]; D(\mathcal{A}))\cap C^1([0, T]; \mathfrak{X}).\]
If this is the case, we have $\Phi=\left(\tilde{w}', \tilde{w}\right)$.
\end{lem}

\begin{lem}\label{generator}
Let $\Omega$ be an open and bounded subset of $\mathds{R}^n$ with smooth boundary $\partial\Omega$, $n=1,\ 2$. Then the linear operator $\mathcal{A}$, defined by \eqref{A-1}, generates a strongly continuous semigroup ($C_0$-semigroup) $\left\{T(t)\in\mathcal{B}\left(\mathfrak{X}\right): t\in[0, \infty)\right\}$.
\end{lem}

\begin{thm}\label{4th-solu-thm}
For $r\in\left(0, \frac{\kappa}{2C}\right)$,  there exist $T_0>0$,  such that for $T\in(0, T_0)$ and given function $\tilde{u}\in C\left([0, T]; B_{H^2}\left(\tilde{u}_0, r\right)\right)$, the semilinear evolution equation \eqref{4th-SWE-1} on $(\tilde{v},\tilde{w})$,
\be\label{4th-SWE-1}
\begin{pmatrix}\tilde{v}'(t) \\ \tilde{w}'(t) \end{pmatrix}=\mathcal{A}\begin{pmatrix}\tilde{v}(t) \\ \tilde{w}(t) \end{pmatrix}+\begin{pmatrix}[G(\tilde{w})](t)+\beta_p\tilde{u}(t), \\ 0 \end{pmatrix}, \ t\in[0, T],\ \begin{pmatrix}\tilde{v}(0) \\ \tilde{w}(0) \end{pmatrix}=\begin{pmatrix}\tilde{v}_0 \\ \tilde{w}_0\end{pmatrix},
\ee
{has} a unique mild solution $\left(\tilde{v}, \tilde{w}\right)\in  C\left([0, T]; L^2(\Omega)\times H_*^2(\Omega)\right)$ defined by
\be\label{mild-solu-form}
\begin{pmatrix}\tilde{v}(t)\\ \tilde{w}(t)\end{pmatrix}=T(t)\begin{pmatrix}\tilde{v}_0\\ \tilde{w}_0\end{pmatrix}+
\displaystyle\int_0^t\left\{T(t-s)\begin{pmatrix}[G(\tilde{w})](s)+\beta_p\tilde{u}(s)\\ 0\end{pmatrix}\right\}ds.
\ee
\end{thm}

\begin{cor}\label{Lip-mild-solu}
Let $T\in(0, T_0)$ and $\tilde{u}\in C\left([0, T]; B_{H^2}\left(\tilde{u}_0,  r\right)\right)\cap C^1\left([0, T]; L^2(\Omega)\right)$. Then the mild solution of the semilinear evolution equation \eqref{4th-SWE-1},
$\left(\tilde{v}, \tilde{w}\right): [0, T]\longrightarrow L^2(\Omega)\times H_*^2(\Omega)$,
defined by the integral form \eqref{mild-solu-form}, is locally Lipschitz continuous with respect to $t\in[0, T]$, i.e. $\forall\ h\in(0, T]$,
\be\label{Lip-mild-solu-inquality}
\sup_{0\leq t<t+h\leq T}\left\|\begin{pmatrix}\tilde{v}(t+h)-\tilde{v}(t)\\ \tilde{w}(t+h)-\tilde{w}(t)\end{pmatrix}\right\|_{L^2(\Omega)\times H^2(\Omega)}\leq L_Vh.
\ee
Here $L_V$ is a Lipschitz constant depending on $\beta_F$, $\beta_p$, $T_0$, $\kappa$, $\Omega$, $\left\|\tilde{u}_0\right\|_{H^{2}(\Omega)}$, $\|\tilde{w}_0\|_{H^2(\Omega)}$,\\ $\left\|\left(\tilde{v}_0, \tilde{w}_0\right)\right\|_{D(\mathcal{A})}$, $M_0=\displaystyle\sup_{t\in[0, \infty)}\left\|T(t)\right\|_{\mathcal{B}(L^2(\Omega)\times H^2(\Omega))}$, $\left\|\tilde{u}\right\|_{C^1\left([0, T_0); L^2(\Omega)\right)}$.
\end{cor}

\begin{cor}\label{Holdercontinuity}
If $\alpha\in(0, 1)$, $T\in(0, T_0)$ and given
$\tilde{u}\in C^\alpha\left([0, T]; B_{H^2}(\tilde{u}_0, r)\right)$,
then the mild solution of the semilinear evolution equation \eqref{4th-SWE-1},
$\left(\tilde{v}, \tilde{w}\right): [0, T]\longrightarrow L^2(\Omega)\times H_*^2(\Omega)$,
defined by the integral formulation \eqref{mild-solu-form}, is locally H\"{o}lder continuous with exponent $\alpha$ with respect to $t\in[0, T]$, i.e. 
\be\label{Holdercontinuityformular}
\sup_{0\leq t<t+h\leq T}\left\|\begin{pmatrix}\tilde{v}(t+h)-\tilde{v}(t)\\ \tilde{w}(t+h)-\tilde{w}(t)\end{pmatrix}\right\|_{L^2(\Omega)\times H^2(\Omega)}\leq L_Uh^\alpha,\quad \forall\ h\in(0, T].
\ee
Here $L_U$ is a Lipschitz constant depending on $\alpha$,  $T_0$, $\kappa$, $\Omega$, $\beta_p$, $\beta_F$, $\left\|\tilde{u}_0\right\|_{H^2(\Omega)}$, $\left\|\tilde{w}_0\right\|_{H^{2}(\Omega)}$,\\ $\left\|\left(\tilde{v}_0, \tilde{w}_0\right)\right\|_{D(\mathcal{A})}$ and  $M_0=\displaystyle\sup_{t\in[0, \infty)}\left\|T(t)\right\|_{\mathcal{B}\left(L^2(\Omega)\times H^2(\Omega)\right)}$.
\end{cor}

\begin{thm}\label{4th-mild-solution-cor}
For given $\tilde{u}\in C\left([0, T]; B_{H^2}\left(\tilde{u}_0,  r\right)\right)\cap C^1\left([0, T]; L^2(\Omega)\right)$ and $T\in(0, T_0)$, the mild solution $\left(\tilde{v}, \tilde{w}\right)$ of the semilinear evolution equation \eqref{4th-SWE-1}, defined by the integral form \eqref{mild-solu-form}, is the strict solution of   equation \eqref{4th-SWE-1} and
\[\left(\tilde{v}, \tilde{w}\right)\in C^1\left([0, T]; L^2(\Omega)\times H_*^2(\Omega)\right)\cap C\left([0, T]; H_*^2(\Omega)\times H_*^4(\Omega)\right).\]
\end{thm}

\section{Solution Operators}\label{SecSolnOp}
\begin{thm}\label{4th-cpl}
Let $T_0$ be given by Theorem \ref{4th-solu-thm} and $T\in (0, T_0)$. Then a solution operator, given by
\[ W_1:\ C\left([0, T]; B_{H^2}\left(\tilde{u}_0, r\right)\right)\longrightarrow \mathcal{Z}(T),\quad \tilde{u}\longmapsto  W_1(\tilde{u})=\left(\tilde{v}, \tilde{w}\right)=\left(\tilde{v}(\tilde{u}), \tilde{w}(\tilde{u})\right),\quad r\in\left(0, \frac{\kappa}{2C}\right)\]
with
\[[W_1(\tilde{u})](t)=\displaystyle T(t)\begin{pmatrix}\tilde{v}_0\\  \tilde{w}_0\end{pmatrix}+
\int_0^t\left\{T(t-s)
\begin{pmatrix}
[G(\tilde{w})](s)+\beta_p\tilde{u}(s)\\
0
\end{pmatrix}
 \right\}ds,\quad t\in[0, T].\]
has Lipschitz continuity, i.e.
\be\label{Lip-wrt-u}
\displaystyle\sup_{t\in[0, T]}\left\|[W_1(\tilde{u}_1)](t)-[W_1(\tilde{u}_2)](t)\right\|_{L^{2}(\Omega)\times H^{2}(\Omega)}\leq L_{W}\sup_{t\in[0, T]}\|\tilde{u}_1(t)-\tilde{u}_2(t)\|_{H^2(\Omega)}.
\ee
Here $L_{W}$ is a Lipschitz constant depending on $T_0$, $M_0=\displaystyle\sup_{t\in[0,\infty)}\left\|T(t)\right\|_{\mathcal{B}\left(L^2(\Omega)\times H^2(\Omega)\right)}$, $\kappa$, $\|w_0\|_{H^{2}(\Omega)}$, $\Omega$, and the coefficients $\beta_p$ and $\beta_F$. Furthermore, let
\[ W_2: C\left([0, T]; B_{H^2}\left(\tilde{u}_0,  r\right)\right) \longrightarrow  C([0,T]; L^2(\Omega)),\quad \tilde{u}\longmapsto \frac{\tilde{v}}{\tilde{w}+\theta_2}.\]
Then $W_2(\tilde{u})$ also depends Lipschitz-continuously on $\tilde{u}\in C\left([0, T]; B_{H^2}\left(\tilde{u}_0,  r\right)\right)$, i.e.
\be\label{Lip-frac-W}
\displaystyle\sup_{t\in[0, T]}\left\|[W_2(\tilde{u}_1)](t)-[W_2(\tilde{u}_2)](t)\right\|_{L^{2}(\Omega)}\leq L_{W_2}\displaystyle\sup_{t\in[0, T]}\|\tilde{u}_1(t)-\tilde{u}_2(t)\|_{H^2(\Omega)},
\ee
where $L_{W_2}$ is a Lipschitz constant depending on above $L_{W}$ and $\left\|\tilde{v}_0\right\|_{L^2(\Omega)}$.
\end{thm}
\begin{proof}
For $T\in(0, T_0)$, $\tilde{u}_1$, $\tilde{u}_2\in C\left([0, T]; B_{H^2}\left(\tilde{u}_0,  r\right)\right)$,  $W_1(\tilde{u}_1)=\left(\tilde{v}_1, \tilde{w}_1\right)$, $W_1(\tilde{u}_2)=\left(\tilde{v}_2, \tilde{w}_2\right)$ belong to $\mathcal{Z}(T)$, then it folows that 
\[[G(\tilde{w}_1)](t)-[G(\tilde{w}_2)](t)+\beta_p\tilde{u}_1(t)-\beta_p\tilde{u}_2(t)\in H^2(\Omega),\quad\forall\ t\in[0, T],\]
and one can have the following estimates
\begin{align}
&\|[W_1(\tilde{u}_1)](t)-[W_1(\tilde{u}_2)](t)\|_{L^2(\Omega)\times H^{2}(\Omega)}\notag\\
=&\left\|\int_0^t\displaystyle T(t-s)\begin{pmatrix}[G(\tilde{u}_1,\tilde{w}_1)](s)-[G(\tilde{u}_2,\tilde{w}_2)](s)\\ 0\end{pmatrix}ds\right\|_{L^2(\Omega)\times H^{2}(\Omega)}\notag\\
\leq&M_0\int_0^t\left\|[G(\tilde{w}_1)](s)-[G(\tilde{w}_2)](s)+\beta_p\left[\tilde{u}_1(s)-\tilde{u}_2(s)\right]\right\|_{L^{2}(\Omega)}ds\notag\\
\leq&M_0\int_0^t\left\|[G(\tilde{w}_1)](s)-[G(\tilde{w}_2)](s)+\beta_p\left[\tilde{u}_1(s)-\tilde{u}_2(s)\right]\right\|_{H^{2}(\Omega)}ds,\label{W estimate-11}
\end{align}
where $M_0$ is a operator norm of $\left\{T(t)\in\mathcal{B}\left(L^2(\Omega)\times H_*^2(\Omega)\right):\ t\in[0, \infty)\right\}$. We notice that,
\begin{align}
\|\tilde{w}_1(t)-\tilde{w}_2(t)\|_{H^{2}(\Omega)}&\leq\left\|\begin{pmatrix}\tilde{v}_1(t)-\tilde{v}_2(t)\\ \tilde{w}_1(t)-\tilde{w}_2(t)\end{pmatrix}\right\|_{L^{2}(\Omega)\times H^{2}(\Omega)}\notag\\
&=\|[W_1(\tilde{u}_1)](t)-[W_1(\tilde{u}_2)](t)\|_{L^{2}(\Omega)\times H^{2}(\Omega)},\quad \forall\ t\in [0, T].\label{w estimates2-11}
\end{align}
Hence combining \eqref{w estimates2-11} with the Lipschitz continuity estimate \eqref{Lip-G} of $G$ from Lemma \ref{Lip-G-Lem} gives
\begin{align}
&\left\|[G(\tilde{w}_1)](s)-[G(\tilde{w}_2)](s)+\beta_p\left[\tilde{u}_1(s)-\tilde{u}_2(s)\right]\right\|_{H^{2}(\Omega)}\notag\\
\leq& L_G \|[W_1(\tilde{u}_1)](s)-[W_1(\tilde{u}_2)](s)\|_{L^{2}(\Omega)\times H^{2}(\Omega)}+\beta_p\left\|\tilde{u}_1(s)-\tilde{u}_2(s)\right\|_{H^2(\Omega)}\notag,\ \forall\ 0\leq s\leq t\leq T.
\end{align}
Hence
\begin{align}
&\|[W_1(\tilde{u}_1)](t)-[W_1(\tilde{u}_2)](t)\|_{L^2(\Omega)\times H^{2}(\Omega)}\notag\\
\leq&T_0M_0\beta_p\sup_{t\in[0,T]}\left\|\tilde{u}_1(t)-\tilde{u}_2(t)\right\|_{H^2(\Omega)}+M_0L_G\displaystyle\int_0^t\|[W_1(\tilde{u}_1)](s)-[W_1(\tilde{u}_2)](s)\|_{L^{2}(\Omega)\times H^{2}(\Omega)}ds\notag.
\end{align}
Gronwall's inequality implies
\[\sup_{t\in[0,T]}\|[W_1(\tilde{u}_1)](t)-[W_1(\tilde{u}_2)](t)\|_{L^{2}(\Omega)\times H^{2}(\Omega)}\leq\displaystyle T_0M_0\beta_pe^{M_0L_GT_0}\sup_{t\in[0,T]}\|\tilde{u}_1(t)-\tilde{u}_2(t)\|_{H^2(\Omega)}.\]
Consequently we conclude \eqref{Lip-wrt-u} by setting
\[L_{W}=T_0M_0\beta_pe^{M_0L_GT_0}.\]
Since $L_G$ is the Lipschitz constant depending on $\kappa$, $\|w_0\|_{H^{2}(\Omega)}$, $\Omega$, the coefficient $\beta_F$,  thus $L_W$ depends on $T_0$, $M_0$, $\kappa$, $\|w_0\|_{H^{2}(\Omega)}$,  $\Omega$, and the coefficients $\beta_p$ and $\beta_F$,  that is
\[L_W=L_W\left(T_0,\ M_0,\ \kappa,\ \|w_0\|_{H^{2}(\Omega)},\ \Omega,\  \beta_p,\ \beta_F\right).\]
From the conclusion \eqref{w-lower-bound} from Lemma \ref{estimates}, we conclude that there exists a constant $C=C(\Omega)$, such that for all $r\in\left(0, \frac{\kappa}{2C}\right)$,
\[\displaystyle{\tilde{w}_1(t)+\theta_2}\geq \frac{\kappa}{2},\quad {\tilde{w}_2(t)+\theta_2}\geq \frac{\kappa}{2}\]
holds for all $t\in[0, T]$. Then for above constant $C=C(\Omega)$ and all $t\in [0, T]$, we obtain
\begin{align}
\left\|\frac{\tilde{v}_1(t)-\tilde{v}_2(t)}{\tilde{w}_1(t)+\theta_2}\right\|_{L^2(\Omega)}&=\left[\int_\Omega\left|\frac{\tilde{v}_1(t)-\tilde{v}_2(t)}{\tilde{w}_1(t)+\theta_2}\right|^2dx\right]^{\frac{1}{2}}\notag\\
&\leq \frac{2}{\kappa}\left\|\tilde{v}_1(t)-\tilde{v}_2(t)\right\|_{L^2(\Omega)},\notag
\end{align}
and therefore
\begin{align}
\left\|\frac{\tilde{v}_2(t)}{\tilde{w}_1(t)+\theta_2}-\frac{\tilde{v}_2(t)}{\tilde{w}_2(t)+\theta_2}\right\|^2_{L^2(\Omega)}&=\int_\Omega\left|\displaystyle\tilde{v}_2(t)\right|^2\left|\frac{1}{\tilde{w}_1(t)+\theta_2}-\frac{1}{\tilde{w}_2(t)+\theta_2}\right|^2dx\notag\\
&=\int_\Omega\left|\displaystyle\tilde{v}_2(t)\right|^2\frac{\left|\tilde{w}_1(t)-\tilde{w}_2(t)\right|^2}{\left|\tilde{w}_1(t)+\theta_2\right|^2\left|\tilde{w}_2(t)+\theta_2\right|^2}dx\notag\\
&\leq\frac{2^4}{\kappa^4}\int_\Omega\left|\displaystyle\tilde{v}_2(t)\right|^2\left|\tilde{w}_1(t)-\tilde{w}_2(t)\right|^2dx\notag\\
&\leq\frac{2^4}{\kappa^4}\left\|\tilde{w}_1(t)-\tilde{w}_2(t)\right\|_{L^\infty(\Omega)}^2\int_\Omega\left|\displaystyle\tilde{v}_2(t)\right|^2dx\notag\\
&\leq\frac{2^4C^2}{\kappa^4}\left\|\tilde{w}_1(t)-\tilde{w}_2(t)\right\|_{H^2(\Omega)}^2\int_\Omega\left|\displaystyle\tilde{v}_2(t)\right|^2dx\notag.
\end{align}
This shows
\begin{align}
&\left\|[W_2(\tilde{u}_1)](t)-[W_2(\tilde{u}_2)](t)\right\|_{L^2(\Omega))}\notag\\
\leq&\left\|\frac{\tilde{v}_1(t)-\tilde{v}_2(t)}{\tilde{w}_1(t)+\theta_2}\right\|_{L^2(\Omega)}+\left\|\frac{\tilde{v}_2(t)}{\tilde{w}_1(t)+\theta_2}-\frac{\tilde{v}_2(t)}{\tilde{w}_2(t)+\theta_2}\right\|_{L^2(\Omega)}\notag\\
\leq&\frac{2}{\kappa}\left\|\tilde{v}_1(t)-\tilde{v}_2(t)\right\|_{L^2(\Omega)}+\frac{4C}{\kappa^2}\left\|\tilde{w}_1(t)-\tilde{w}_2(t)\right\|_{H^2(\Omega)}\left\|\displaystyle\tilde{v}_2(t)\right\|_{L^2(\Omega)}\notag\\
\leq&\frac{2}{\kappa}\left\|\tilde{v}_1(t)-\tilde{v}_2(t)\right\|_{L^2(\Omega)}+\frac{4C}{\kappa^2}\left\|\tilde{w}_1(t)-\tilde{w}_2(t)\right\|_{H^2(\Omega)}\left(\left\|\displaystyle\tilde{v}_0\right\|_{L^2(\Omega)}+r\right)\notag\\
\leq&\frac{2}{\kappa}\left\|\tilde{v}_1(t)-\tilde{v}_2(t)\right\|_{L^2(\Omega)}+\frac{4C}{\kappa^2}\left\|\tilde{w}_1(t)-\tilde{w}_2(t)\right\|_{H^2(\Omega)}\left(\left\|\displaystyle\tilde{v}_0\right\|_{L^2(\Omega)}+\frac{\kappa}{2C}\right)\label{final estimate-12}.
\end{align}
We set
\[L_{W_2}=L_{W}\cdot\max\left\{\frac{2}{\kappa},\ \frac{4C}{\kappa^2}\left(\left\|\displaystyle\tilde{v}_0\right\|_{L^2(\Omega)}+\frac{\kappa}{2C}\right)\right\},\]
and $L_{W_2}$ depends on above $L_{W}$ and  $\left\|\tilde{v}_0\right\|_{L^2(\Omega)}$,  that is
\[L_{W_2}=L_{W_2}\left(L_{W_1},\ \left\|\tilde{v}_0\right\|_{L^2(\Omega)}\right).\]
Using  estimates \eqref{Lip-wrt-u} and \eqref{final estimate-12}, we conclude \eqref{Lip-frac-W}.
\end{proof}

\subsection*{Fr\'{e}chet derivative}
For $T\in(0, T_0)$.  Recall that  $u=\tilde{u}+\theta_1$, $v=\tilde{v}$ and $w=\tilde{w}+\theta_2$. According to the integral form \eqref{mild-solu-form},  the solution operator given by
\bse\label{def-W}
\be\label{def-W-01}
\begin{split}
 W: C\left([0, T]; B_{H^2}(u_0,  r)\right)\longrightarrow C\left([0, T]; B_{L^2}(v_0,  r)\times B_{H^2}(w_0,  r)\right),\\
   u\longmapsto W(u)=\left(v, w\right)=\left(v(u), w(u)\right),\qquad\qquad\qquad\quad
   \end{split}
\ee
with
\be\label{def-W-02}
\left[W(u)\right](t)=\begin{pmatrix}0\\ \theta_2\end{pmatrix}+ T(t)\begin{pmatrix}v_0\\ w_0-\theta_2\end{pmatrix}+
\int_0^t\left\{T(t-s)\begin{pmatrix}[G(w-\theta_2)](s)+\beta_p(u(s)-\theta_1)\\ 0 \end{pmatrix}\right\}ds,
\ee
\ese
similarly satisfies the following Lipschitz continuity
\begin{align}
\displaystyle\sup_{t\in[0, T]}\left\|[W({u}_1)](t)-[W({u}_2)](t)\right\|_{L^{2}(\Omega)\times H^{2}(\Omega)}\leq L_{W}\sup_{t\in[0, T]}\|{u}_1(t)-{u}_2(t)\|_{H^2(\Omega)},\label{Lip-W-I}
\end{align}
where $L_{W}$ is a Lipschitz constant depending on $T_0,\ M_0,\ \kappa,\ \|w_0\|_{H^{2}(\Omega)},\  \Omega,\  \beta_p$ and $\beta_F$.

Let $\lambda\in\mathds{R}$ be small such that, for any $q\in C\left([0, T]; H_*^2(\Omega)\right)$, $u+\lambda q\in C\left([0, T]; B_{H^2}(u_0, r)\right)$,  then according to the definition of the Fr\'{e}chet derivative ${W}{'}(u)$ of ${W}(u)$ on $u$,
\be\label{Fre-W}
{W}{'}(u)q=\displaystyle\lim_{\lambda\rightarrow 0}\frac{1}{\lambda}\left[{W}({u}+\lambda q)-{W}({u})\right],
\ee
${W}{'}(u)$ is a map defined by
\bse\label{Fre-W-1}
\be
{W}{'}(u): C\left([0,T]; H_*^2(\Omega)\right)\longrightarrow C\left([0,T]; L^2(\Omega)\times H_*^2(\Omega)\right),
\ee
with
\be
q\longmapsto W'(u)q=\left(v'(u)q, w'(u)q\right).
\ee
\ese
\eqref{Lip-W-I} implies that the Fr\'{e}chet derivative $W{'}(u)$ of $W(u)$ with respect to $u$ exists and
\bse\label{bound-Fre-W}
\be\label{bound-Fre-W-01}
\sup_{t\in[0, T]}\left\|\left[W'(u)q\right](t)\right\|_{L^{2}(\Omega)\times H^{2}(\Omega)}\leq L_W\sup_{t\in[0, T]}\left\|q(t)\right\|_{H^2(\Omega)},
\ee
\be\label{bound-Fre-W-02}
\sup_{t\in[0, T]}\left\|[W'(u)q](t)\right\|_{L^{2}(\Omega)\times H^{2}(\Omega)}\leq L_W,\quad  \forall\ q\in C\left([0, T]; B_{H^2}(0,1)\right).
\ee
\ese
\begin{cor}\label{Lip-Frechet-W}
For any given $q\in C\left([0, T]; B_{H^2}(0,1)\right)$ and $u_1$, $u_2\in C\left([0, T]; B_{H^2}(u_0, r)\right)$ with $T\in(0, T_0)$,  the Fr\'{e}chet derivative $W'(u)$ of $W(u)$ satisfies
\be\label{Lip-Frechet-W-I}
\displaystyle\sup_{t\in[0, T]}\left\|[{W}'(u_1)q](t)-[{W}'(u_2)q](t)\right\|_{L^{2}(\Omega)\times H^{2}(\Omega)}\leq L_{F}\displaystyle\sup_{t\in[0, T]}\left\|{u}_1(t)-{u}_2(t)\right\|_ {H^2(\Omega)}.
\ee
Here $L_F$ is a Lipschitz constant depending on $T_0$, $\Omega$, $\beta_p$, $\beta_F$, $M_0$, $\kappa$, $\|w_0\|_{H^{2}(\Omega)}$ and $\|v_0\|_{L^2(\Omega)}$.
\end{cor}
\begin{proof}
From Lemma \ref{generator}, the linear operator $\mathcal{A}$ generates a $C_0$-semigroup
\[\left\{T(t)=\begin{pmatrix}T_{11}(t) & T_{12}(t)\\ T_{21}(t) & T_{22}(t) \end{pmatrix}\in \mathcal{B}\left(L^2(\Omega)\times H_*^2(\Omega)\right):\quad t\in[0,\infty)\right\}.\]
Because $[G(w-\theta_2)](s)+\beta_p(u(s)-\theta_1)=-\beta_F[w(s)]^{-2}+\beta_p(u(s)-1)$, the integral form \eqref{def-W-02} then implies the second component $w$ of $W(u)=\left(v, w\right)=\left(v(u), w(u)\right)$ is given by
\begin{align}
w(t)=&\left[w(u)\right](t)\notag\\
=&\theta_2+T_{21}(t)v_0+T_{22}(t)\left(w_0-\theta_2\right)+\displaystyle\int_0^tT_{21}(t-s)\left(\beta_p(u(s)-1)-\frac{\beta_F}{[w(u)]^2(s)}\right)ds.\notag
\end{align}
Following this definition and the definitions \eqref{Fre-W} and \eqref{Fre-W-1} of the Fr\'{e}chet derivative $W'(u)=\left(v'(u), w'(u)\right)$, the Fr\'{e}chet derivative $w'(u)$ of $w(u)$ on $u$, which is also the second component of the Fr\'{e}chet derivative $W'(u)$, is written as follows:
\bse\label{Fre-w}
\be\label{Fre-w1}
w'(u): C\left([0, T]; H_*^2(\Omega)\right)\longrightarrow C\left([0, T]; H_*^2(\Omega)\right),
\ee
where
\be\label{Fre-w2}
[w'(u)q](t)=\displaystyle\int_0^t T_{21}(t-s)\left\{\beta_p q(s)+2\beta_F\frac{[w'(u)q](s)}{[w(u)]^3(s)}\right\}ds.
\ee
\ese
We next show that there exists a Lipschitz constant $L_{F_2}$ depending on $T_0$,  $M_0$, $\kappa$, $\|w_0\|_{H^2(\Omega)}$, $\Omega$, $\beta_p$ and $\beta_F$, such that
\be\label{Lip-Fre-w}
\displaystyle\sup_{t\in[0, T]}\|[w'(u_1)q](t)-[w'(u_2)q](t)\|_{H^2(\Omega)}\leq L_{F_2}\sup_{t\in[0, T]}\|{u}_1(t)-{u}_2(t)\|_{H^2(\Omega)}.
\ee
Letting $u_1,\ u_2\in C\left([0, T]; B_{H^2}(u_0,  r)\right)$, the definiton \eqref{def-W} of the solution operator $W$ implies that  $(v_1, w_1)=(v(u_1), w(u_1))$, $(v_2, w_2)=(v(u_2), w(u_2))\in C\left([0, T]; B_{L^2}(v_0,  r)\times B_{H^2}(w_0,  r)\right)$, then one obtains $w_1(t)$, $w_2(t)\in H^2(\Omega)$.
Because of the definitions \eqref{Fre-W} and \eqref{Fre-W-1} of Fr\'{e}chet derivative $W^{'}(u)$,
\[W'(u_1)q=\left(v'(u_1)q, w'(u_1)q\right),\quad  W'(u_2)q=\left(v'(u_2)q, w'(u_2)q\right) \in C\left([0, T]; L^2(\Omega)\times H_*^2(\Omega)\right).\]
Hence with
\[w_I(t):=[w'(u_1)q](t)\in H_*^2(\Omega),\quad  w_J(t):=[w'(u_2)q](t)\in H_*^2(\Omega),\quad \forall\ t\in [0, T],\]
we find
\[\frac{[w'(u_1)q](t)}{[w(u_1)]^3(t)}-\frac{[w'(u_2)q](t)}{[w(u_2)]^3(t)}=\frac{w_I(t)}{[w_1(t)]^3}-\frac{w_J(t)}{[w_2(t)]^3}\in H_*^2(\Omega),\quad \forall\ t\in [0, T].\]
According to inequality \eqref{bound-Fre-W},
\[\sup_{t\in[0, T]}\left\|w_I(t)\right\|_{H^2(\Omega)}=\sup_{t\in[0, T]}\left\|[w'(u_1)q](t)\right\|_{H^2(\Omega)}\leq L_W.\]
The Lipschitz continuity estimate \eqref{Lip-W-I} implies that
\be\label{Lip-w-II}
\left\|w_1(t)-w_2(t)\right\|_{H^2(\Omega)}=\left\|[w(u_1)](t)-[w(u_2)](t)\right\|_{H^2(\Omega)}\leq L_W\left\|u_1(t)-u_2(t)\right\|_{H^2(\Omega)}.
\ee
The algebraic properties of $H^2(\Omega)$, i.e. Lemma \ref{alg}, estimates \eqref{C-a} and \eqref{C-d} of Lemma \ref{estimates}, and above \eqref{Lip-w-II} imply
\begin{align}
\left\|\frac{w_I(t)}{[w_1(t)]^3}-\frac{w_J(t)}{[w_2(t)]^3}\right\|_{H^2(\Omega)}\leq&\left\|w_I(t)\right\|_{H^2(\Omega)}\cdot \left\|\frac{1}{[w_1(t)]^3}-\frac{1}{[w_2(t)]^3}\right\|_{H^2(\Omega)}\notag\\
+&\left\|w_I(t)-w_J(t)\right\|_{H^2(\Omega)}\cdot\left\|\frac{1}{[w_2(t)]^3}\right\|_{H^2(\Omega)}\notag\\
\leq&L_W C_3\left\|w_1(t)-w_2(t)\right\|_{H^2(\Omega)}+C_1^3\left\|w_I(t)-w_J(t)\right\|_{H^2(\Omega)}\notag\\
\leq&L^2_WC_3\left\|u_1(t)-u_2(t)\right\|_{H^2(\Omega)}+C_1^3\left\|w_I(t)-w_J(t)\right\|_{H^2(\Omega)}\label{w-cubic}.
\end{align}
Combining \eqref{w-cubic} with the form \eqref{Fre-w} of the Fr\'{e}chet derivative $w'(u)q$ of $w(u)$ on $u$ gives
\begin{align}
\displaystyle\|w_I(t)-w_J(t)\|_{H^2(\Omega)}=&\left\|2\beta_F\int_0^tT_{21}(t-s)\left(\frac{w_I(s)}{[w_1(s)]^3}-\frac{w_J(s)}{[w_2(s)]^3}\right)ds\right\|_{H^2(\Omega)}\notag\\
\leq&2\beta_F \int_0^t \sup_{0\leq s\leq t}\left\|T_{21}(t-s)\right\|_{\mathcal{B}\left(L^2(\Omega), H^2(\Omega)\right)}\left\|\frac{w_I(s)}{[w_1(s)]^3}-\frac{w_J(s)}{[w_2(s)]^3}\right\|_{L^2(\Omega)}ds\notag\\
\leq&2\beta_F  M_0\int_0^t\left\|\frac{w_I(s)}{[w_1(s)]^3}-\frac{w_J(s)}{[w_2(s)]^3}\right\|_{H^2(\Omega)}ds\notag\\
\leq&2\beta_F M_0L^2_{W}C_3T_0\|{u}_1(t)-{u}_2(t)\|_{H^2(\Omega)}\notag\\
+&2\beta_FM_0C_1^3\int_0^t\left\|w_I(s)-w_J(s)\right\|_{H^2(\Omega)}ds\notag.
\end{align}
Consequently, according to Gronwall's inequality,
\[\displaystyle\|w_I(t)-w_J(t)\|_{H^2(\Omega)}\leq 2\beta_F M_0L^2_{W}C_3T_0e^{2\beta_FM_0C_1^3 T_0}\|{u}_1(t)-{u}_2(t)\|_{H^2(\Omega)}.\]
Thus the estimate \eqref{Lip-Fre-w} holds by setting
\[L_{F_2}=2\beta_F M_0L^2_{W}C_3T_0e^{2\beta_FM_0C_1^3 T_0}.\]
Similarly,  there exists a Lipschitz constant $L_{F_1}=L_{F_1}\left(L_{F_2},\ \|{v}_0\|_{L^2(\Omega)}\right)>0$, such that the Frech\'{e}t derivative $v{'}(u)$ of the first component $v(u)$ of $W(u)$ on $u$, given by
\be\label{Fre-v}
[v{'}(u)q](t)=\displaystyle\int_0^t T_{11}(t-s)\left\{\beta_p q(s)+2\beta_F\frac{[w{'}(u)q](s)}{[w(u)]^3(s)}\right\}ds,
\ee
is a map $t\longrightarrow[v{'}(u)](t)$ from $[0, T]$ to  $C([0, T];\mathcal{B}\left(H_*^2(\Omega), L^2(\Omega))\right)$ and satisfies
\be\label{Lip-Fre-v}
\displaystyle\sup_{t\in[0, T]}\|[v'(u_1)q](t)-[v'(u_2)q](t)\|_{L^2(\Omega)}\leq L_{F_1}\sup_{t\in[0, T]}\|{u}_1(t)-{u}_2(t)\|_{H^2(\Omega)}.
\ee
Let $L_F=\max\left\{L_{F_1},L_{F_2}\right\}$, and the assertion \eqref{Lip-Frechet-W-I} follows from \eqref{Lip-Fre-w} and \eqref{Lip-Fre-v}.
\end{proof}

\begin{cor}\label{Holder-Frechet-W-I-Cor}
Take $T\in (0, T_0)$, $r\in\left(0, \frac{\kappa}{2C}\right)$, $u_0\in H^2(\Omega)$ compatible with boundary condition, $u_0\mid_{\partial\Omega}=\theta_1$, and $\tilde{u}_0=u_0-\theta_1\in H_*^2(\Omega)$. If
$u\in C^\alpha\left([0, T]; B_{H^2}(u_0,r)\right)$,
then there exists a Lipschitz constant $L_M$ depending on
$\alpha$, $M_0$, $T_0$, $\Omega$, $\beta_F$, $\beta_p$, $\kappa$, $\|v_0\|_{L^2(\Omega)}$ and $\|w_0\|_{H^2(\Omega)}$,
such that
\begin{align}
\sup_{0\leq t<t+h\leq T}\left\|[{W}'(u)q](t+h)-[{W}'(u)q](t)\right\|_{L^{2}(\Omega)\times H^{2}(\Omega)}\leq&h^\alpha L_{M}\sup_{t\in [0, T]}\left\|q(t)\right\|_{H^2(\Omega)}\notag\\
+&h^\alpha TL_{M}\left\|q\right\|_{C^\alpha\left([0, T]; H^2(\Omega)\right)}\label{Holder-Frechet-W-I}
\end{align}
holds for all $q\in C^\alpha([0, T]; B_{H^2}(\tilde{u}_0, r))$.
\end{cor}
\begin{proof}
Let $T\in(0, T_0)$ and $r\in\left(0, \frac{\kappa}{2C}\right)$. If the given function 
$u$ belongs to $C^\alpha\left([0, T]; B_{H^2}(u_0,r)\right)$, also $u\in C\left([0, T]; B_{H^2}(u_0,r)\right)$,  according to  Theorem \ref{4th-solu-thm} and Corollary \ref{Holdercontinuity}, it follows  that  the semilinear fourth-order equation \eqref{cp2-1-2} has a unique mild solution $w(u)\in C^\alpha\left([0, T]; B_{H^2}(w_0, r)\right)$  and  $w(u)$ can be written by
\[\left[w(u)\right](t)=\theta_2+T_{21}(t)v_0+T_{22}(t)\left(w_0-\theta_2\right)+\displaystyle\int_0^tT_{21}(t-s)\left(\beta_p(u-1)-\frac{\beta_F}{[w(u)]^2(s)}\right)ds.\]
The definitions \eqref{Fre-W} and \eqref{Fre-W-1} of the Fr\'{e}chet derivative $W'(u)$ imply that the Fr\'{e}chet derivative $w'(u)$ satisfies
\be\label{TFre-w2}
[w'(u)q](t)=\displaystyle\int_0^t T_{21}(t-s)\left\{\beta_p q(s)+2\beta_F\frac{[w'(u)q](s)}{[w(u)]^3(s)}\right\}ds.
\ee
We are going to show that there is a constant $L_{M_1}>0$, such that
\be\label{H-I}
\left\|[w'(u)q](t+h)-[w'(u)q](t)\right\|_{H^2(\Omega)}\leq L_{M_1}h^\alpha\left\{\sup_{t\in[0, T]}\left\|q(t)\right\|_{H^2(\Omega)}+T\left\|q\right\|_{C^\alpha([0, T]; H^2(\Omega))}\right\}
\ee
holds for $\forall\ 0\leq t<t+h\leq T,\ h\in(0, T]$. Because
\begin{align}
[w'(u)q](t+h)-[w'(u)q](t)=&\displaystyle\int_0^h T_{21}(t+h-s)\left\{\beta_p q(s)+2\beta_F\frac{[w'(u)q](s)}{[w(u)]^3(s)}\right\}ds\notag\\
+&\displaystyle\int_0^t T_{21}(t-s)\beta_p[q(s+h)-q(s)]ds\notag\\
+&\displaystyle\int_0^t T_{21}(t-s)2\beta_F\left\{\frac{[w'(u)q](s+h)}{[w(u)]^3(s+h)}-\frac{[w'(u)q](s)}{[w(u)]^3(s)}\right\}ds\label{H00},
\end{align}
$q(t)\in H_*^2(\Omega)$, the definition \eqref{def-W} of $W(u)$ and the definition \eqref{Fre-W} of $W'(u)$ give
\[\beta_p q(t)+2\beta_F\frac{[w'(u)q](t)}{[w(u)]^3(t)}\in H_*^2(\Omega).\]
Because \eqref{C-a} of Lemma \ref{estimates} and \eqref{bound-Fre-W}, we have
\begin{align}
\sup_{t\in [0, T]}\left\|\frac{[w'(u)q](t)}{[w(u)]^3(t)}\right\|_{H^2(\Omega)}\leq&\sup_{t\in[0, T]}\left\|\frac{1}{[w(u)]^3(t)}\right\|_{H^2(\Omega)}\sup_{t\in [0, T]}\left\|[w'(u)q](t)\right\|_{H^2(\Omega)}\notag\\
\leq&C_1^3L_W\sup_{t\in [0, T]}\left\|q(t)\right\|_{H^2(\Omega)}.
\end{align}
Therefore
\begin{align}
&\left\|\displaystyle\int_0^h T_{21}(t+h-s)\left\{\beta_p q(s)+2\beta_F\frac{[w'(u)q](s)}{[w(u)]^3(s)}\right\}ds\right\|_{H^2(\Omega)}\notag\\
\leq&hM_0\left\{\beta_p\sup_{t\in [0, T]}\|q(t)\|_{H^2(\Omega)}+\sup_{t\in [0, T]}\left\|\frac{[w'(u)q](t)}{[w(u)]^3(t)}\right\|_{H^2(\Omega)}\right\}\notag\\
\leq&hM_0\left(\beta_p+C_1^3L_W\right)\sup_{t\in [0, T]}\left\|q(t)\right\|_{H^2(\Omega)}\label{H01}.
\end{align}
As $q\in C^\alpha([0, T]; B_{H^2}(\tilde{u}_0, r))$,
\begin{align}
\left\|\displaystyle\int_0^t T_{21}(t-s)\beta_p[q(s+h)-q(s)]ds\right\|_{H^2(\Omega)}\leq& TM_0\beta_p\sup_{0\leq t<t+h\leq T}\left\|q(t+h)-q(t)\right\|_{H^2(\Omega)}\notag\\
\leq&h^\alpha TM_0\beta_p\|q\|_{C^\alpha([0, T]; H^2(\Omega))}\label{H02},
\end{align}
and $w(u)\in C^\alpha\left([0, T]; B_{H^2}(w_0, r)\right)$ is a mild solution of the semilinear fourth-order equation \eqref{cp2-1-2}, \[\left\|[w(u)](t)-w_0\right\|_{H^2(\Omega)}\leq r,\ \left\|[w(u)](t+h)-w_0\right\|_{H^2(\Omega)}\leq r, \ \forall\ t,\ t+h\in [0, T],\]
then estimate \eqref{C-d} of Lemma \ref{estimates} and the estimate \eqref{Holdercontinuityformular} of Corollary \ref{Holdercontinuity}  imply
\be\label{frac-3-D}
\left\|\frac{1}{[w(u)]^3(t+h)}-\frac{1}{[w(u)]^3(t)}\right\|_{H^2(\Omega)}\leq C_3 \left\|[w(u)](t+h)-[w(u)](t)\right\|_{H^2(\Omega)}\leq C_3L_Uh^\alpha.
\ee
Therefore, estimate \eqref{C-a} of Lemma \ref{estimates}, inequalities \eqref{bound-Fre-W} and \eqref{frac-3-D} imply
\begin{align}
&\left\|\frac{[w'(u)q](t+h)}{[w(u)]^3(t+h)}-\frac{[w'(u)q](t)}{[w(u)]^3(t)}\right\|_{H^2(\Omega)}\notag\\
\leq&\left\|\frac{[w'(u)q](t+h)}{[w(u)]^3(t+h)}-\frac{[w'(u)q](t+h)}{[w(u)]^3(t)}\right\|_{H^2(\Omega)}
+\left\|\frac{[w'(u)q](t+h)}{[w(u)]^3(t)}-\frac{[w'(u)q](t)}{[w(u)]^3(t)}\right\|_{H^2(\Omega)}\notag\\
\leq&\left\|[w'(u)q](t+h)\right\|_{H^2(\Omega)}\left\|\frac{1}{[w(u)]^3(t+h)}-\frac{1}{[w(u)]^3(t)}\right\|_{H^2(\Omega)}\notag\\
+&\left\|[w'(u)q](t+h)-[w'(u)q](t)\right\|_{H^2(\Omega)}\left\|\frac{1}{[w(u)]^3(t)}\right\|_{H^2(\Omega)}\notag\\
\leq& L_W\sup_{t\in [0, T]}\left\|q(t)\right\|_{H^2(\Omega)}C_3L_Uh^\alpha+\left\|[w'(u)q](t+h)-[w'(u)q](t)\right\|_{H^2(\Omega)}C_1^3\notag,
\end{align}
and hence
\begin{align}
&\left\|\displaystyle\int_0^t T_{21}(t-s)2\beta_F\left\{\frac{[w'(u)q](s+h)}{[w(u)]^3(s+h)}-\frac{[w'(u)q](s)}{[w(u_0)]^3(s)}\right\}ds\right\|_{H^2(\Omega)}\notag\\
\leq&2\beta_FM_0T_0L_W\sup_{t\in [0, T]}\left\|q(t)\right\|_{H^2(\Omega)}C_3L_Uh^\alpha\notag\\
+&2\beta_FM_0C_1^3\displaystyle\int_0^t \left\|[w'(u)q](s+h)-[w'(u)q](s)\right\|_{H^2(\Omega)}ds.\label{H03}
\end{align}
Consequently, \eqref{H00}, \eqref{H01}, \eqref{H02} and \eqref{H03} imply,
\begin{align}
\left\|[w'(u)q](t+h)-[w'(u)q](t)\right\|_{H^2(\Omega)}
\leq& h^\alpha T_0^{1-\alpha}M_0\left[\beta_p+C_1^3L_W\right]\sup_{t\in [0, T]}\left\|q(t)\right\|_{H^2(\Omega)}\notag\\
+&h^\alpha TM_0\beta_p\|q\|_{C^\alpha\left([0, T]; H^2(\Omega)\right)}\notag\\
+&h^\alpha2\beta_FM_0T_0L_WC_3L_U\sup_{t\in [0, T]}\left\|q(t)\right\|_{H^2(\Omega)}\notag\\
+&2\beta_FM_0C_1^3\displaystyle\int_0^t \left\|[w'(u)q](s+h)-[w'(u)q](s)\right\|_{H^2(\Omega)}ds\notag.
\end{align}
Set
\[R_1=M_0T_0^{1-\alpha}\left[\beta_p+C_1^3L_W\right],\quad R_2=2\beta_FM_0T_0L_WC_3L_U ,\quad R_3=M_0\beta_p,\quad R_4=2\beta_FM_0C_1^3.\]
Gronwall's inequality implies $\forall\ 0\leq t<t+h\leq T$,
\begin{align}
\left\|[w'(u)q](t+h)-[w'(u)q](t)\right\|_{H^2(\Omega)}\leq &h^\alpha e^{R_4T_0}(R_1+R_2)\sup_{t\in [0, T]}\left\|q(t)\right\|_{H^2(\Omega)}\notag\\
+&h^\alpha T e^{R_4T_0} R_3\|q\|_{C^\alpha\left([0, T]; H^2(\Omega)\right)}.\notag
\end{align}
Equation \eqref{H-I} holds by setting
\[L_{M_1}= (R_1+R_2+R_3)e^{R_4T_0}\]
and $L_{M_1}$ depends on
$\alpha,\ M_0,\ T_0,\ \Omega,\ \|w_0\|_{H^2(\Omega)},\ \kappa,\ \beta_F,\ \beta_p.$

Similarly, there exists a Lipschitz constant
\[L_{M_2}=L_{M_2}\left(L_{M_1},\ \|{v}_0\|_{L^2(\Omega)}\right)>0,\]
such that the Frech\'{e}t derivative $v{'}(u)$ of the first component $v(u)$ of $W(u)$, defined by
\[[v{'}(u)q](t)=\displaystyle\int_0^t T_{11}(t-s)\left\{\beta_p q(s)+2\beta_F\frac{[w{'}(u)q](s)}{[w(u)]^3(s)}\right\}ds,\]
satisfies
\be\label{H05}
\|[v'(u)q](t+h)-[v'(u)q](t)\|_{L^2(\Omega)}\leq L_{M_2}h^\alpha\left[\sup_{t\in[0, T]}\left\|q(t)\right\|_{H^2(\Omega)}+T\left\|q\right\|_{C^\alpha\left([0, T]; H^2(\Omega)\right)}\right].
\ee
Let $L_M=\max\left\{L_{M_1},L_{M_2}\right\}$, and the assertion \eqref{Holder-Frechet-W-I} follows from \eqref{H-I} and \eqref{H05}.
\end{proof}
\section{Wellposedness of the Coupled System}\label{section4}

\subsection*{Abstract Formulation of the Coupled System}
Let $T\in(0, T_0)$ be taken to be specified below. We are going to study the unique existence of the strict solution for the initial-boundary value problem for the coupled system which is written by the quasilinear parabolic equation with abstract coefficients involving $v(u)$ and $w(u)$:
\bse\label{QPE}
\be\label{QPE-1}
\frac{\partial u}{\partial t}=\frac{1}{w(u)}\nabla\cdot\left([w(u)]^3u\nabla u\right)-\frac{v(u)}{w(u)}u,\quad (x,t)\in\Omega\times(0, T),
\ee
\be\label{QPE-2}
u(x,0)=u_0(x),\quad x\in\Omega,\quad u(x,t)=\theta_1,\quad (x,t)\in\partial\Omega\times[0, T].
\ee
\ese
Here $u=u(x,t)$ is an unknown function, $v(u)=[v(u)](x,t)$ and $w(u)=[w(u)](x,t)$ are implicitly given as functions of $u$ by the integral formulation
\[\begin{pmatrix}v(t)\\ w(t)\end{pmatrix}=\begin{pmatrix}0\\ \theta_2\end{pmatrix}+ T(t)\begin{pmatrix}v_0\\ w_0-\theta_2\end{pmatrix}+\int_0^t\left\{T(t-s)\begin{pmatrix}-\beta_F[w(s)]^{-2}+\beta_p\left(u(s)-1\right)\\ 0 \end{pmatrix}\right\}ds,\]
Here, $\left\{T(t):\ t\geq 0\right\}$ is the strongly continuous semigroup ($C_0$-semigroup) from Lemma \ref{generator}.  Note that  $\tilde{u}_0=u_0-\theta_1$ and $\tilde{u}=u-\theta_1$. If $\tilde{u} \in C\left([0, T]; B_{H^2}(\tilde{u}_0, r)\right)$, then from the existence of the mild solution of the semilinear evolution equation \eqref{4th-SWE-1}, i.e.~Theorem \ref{4th-solu-thm}, the functions $v$ and $w$ depending on $u$ can be regarded as the solution operators satisfying the definition \eqref{def-W} of solution operator $W$ and
\vspace*{-0.3cm}
\[v:  C\left([0, T]; B_{H^2}(\tilde{u}_0, r)\right)\longrightarrow  C\left([0, T];B_{ L^2}(v_0, r)\right), \quad  \tilde{u}\longmapsto  v(\tilde{u}+\theta_1),\]
\[w:  C\left([0, T]; B_{H^2}(\tilde{u}_0, r)\right)\longrightarrow C\left([0, T]; B_{H^2}(w_0, r)\right),\quad \tilde{u}\longmapsto w(\tilde{u}+\theta_1).\]
Hence $F(\tilde{u})$ is given by
\bse\label{F-def}
\be\label{F-def-1}
\tilde{u}\longmapsto F(\tilde{u}):\quad C\left([0, T]; B_{H^2}(\tilde{u}_0, r)\right)\longrightarrow C\left([0, T]; L^2(\Omega)\right),
\ee
\be\label{F-def-2}
F(\tilde{u})=\frac{1}{w(\tilde{u}+\theta_1)}\nabla\cdot\left(\left[w(\tilde{u}+\theta_1)\right]^3(\tilde{u}+\theta_1)\nabla\tilde{u}\right)-\frac{v(\tilde{u}+\theta_1)}{w(\tilde{u}+\theta_1)}(\tilde{u}+\theta_1).
\ee
\ese
The linearization of $F(\tilde{u})$ is defined by
\be\label{linearization}
q\longmapsto F'(\tilde{u}_0)q:\quad C\left([0, T]; H_*^2(\Omega)\right)\longrightarrow C\left([0, T]; L^2(\Omega)\right).
\ee
Here,  $F'(\tilde{u}_0)q$ is the Fr\'{e}chet derivative of $F(\tilde{u})$ on $\tilde{u}$ at $\tilde{u}_0$, $F'(\tilde{u}_0)q$ at $t$ is given as:
\begin{align}
\left[F'(\tilde{u}_0)q\right](t)=&\frac{1}{[w(u_0)](t)}\nabla\cdot\left\{[w(u_0)]^3(t)u_0\nabla q(t)+[w(u_0)]^3(t)q(t) \nabla u_0\right\}\notag\\
+&\frac{1}{[w(u_0)](t)}\nabla\cdot\left\{3[w(u_0)]^2(t)[w'(u_0)q](t)u_0\nabla  u_0\right\}\notag\\
-&\frac{[w'(u_0)q](t)}{[w(u_0)]^2(t)}\nabla\cdot\left([w(u_0)]^3(t)u_0\nabla u_0\right)-\frac{[v(u_0)](t)}{[w(u_0)](t)}q(t)\notag\\
-&\frac{[w(u_0)](t)[v'(u_0)q](t)-[v(u_0)](t)[w'(u_0)q](t)}{[w(u_0)]^2(t)}u_0, \label{Fre-F}
\end{align}
where the functions $v(u_0)$ and $w(u_0)$ satisfy the definition \eqref{def-W} of the solution operator $W$ with $u=u_0$. Equivalently, $(\tilde{v}, \tilde{w})=(v(u_0), w(u_0)- \theta_2)$ is a unique mild solution of the semilinear evolution equation \eqref{4th-SWE-1} with $\tilde{u}=u_0-\theta_1$, and $([v(u_0)](0),[w(u_0)](0))=(v_0, w_0)$. Define
\begin{align}
\mathcal{P}^*q(t)=&\frac{1}{w_0}\nabla\cdot\left\{w_0^3u_0\nabla q(t)+w_0^3q(t)\nabla u_0\right\}+\frac{1}{w_0}\nabla\cdot\left\{3w_0^2[w'(u_0)q](0)u_0\nabla u_0\right\}\notag\\
-&\frac{[w'(u_0)q](0)}{w_0^2}\nabla\cdot\left(w_0^3u_0\nabla u_0\right)-\frac{v_0}{w_0}q(t)-\frac{w_0[v'(u_0)q](0)-v_0[w'(u_0)q](0)}{w_0^2}u_0. \label{linearopt-0}
\end{align}
Note that the Fr\'{e}chet derivative $w'(u)$ at $u=u_0$ and $t=0$ is given by
\[[w'(u_0)q](0)=\displaystyle\lim_{h\rightarrow 0}\frac{1}{h}\left\{[w(u_0+hq)](0)-[w(u_0)](0)\right\}.\]
Here $h\in\mathds{R}$ is small such that, for any $q\in C\left([0, T]; H_*^2(\Omega)\right)$, $u_0+hq\in C\left([0, T]; B_{H^2}(u_0,  r)\right)$. Because $(\tilde{v}_h, \tilde{w}_h):=(v(u_0+hq), w(u_0+hq)-\theta_2)$ is a unique mild solution of the semilinear evolution equation \eqref{4th-SWE-1} with $\tilde{u}=u_0+hq-\theta_1$, then $([v(u_0+hq)](0),  [w(u_0+hq)](0))=(v_0, w_0)$. Since $([v(u_0)](0),[w(u_0)](0))=(v_0, w_0)$, then $[w'(u_0)q](0)=0$ and $[v'(u_0)q](0)=0$, hereby \eqref{linearopt-0} becomes
\[\mathcal{P}^*q(t)=\frac{1}{w_0}\nabla\cdot\left\{w_0^3u_0\nabla q(t)+w_0^3q(t)\nabla u_0\right\}-\frac{v_0}{w_0}q(t).\]
$\mathcal{P}^*$ is a linear operator defined by
\be\label{linearopt}
\mathcal{P}^*:\quad D\left(\mathcal{P}^*\right)\subseteq H_*^2(\Omega)\longrightarrow L^2(\Omega),\quad \mathcal{P}^*\psi=\frac{1}{w_0}\nabla\cdot\left\{w_0^3u_0\nabla\psi+\left(w_0^3\nabla u_0\right)\psi\right\}-\frac{v_0}{w_0}\psi.
\ee
We remind the reader that $\mathcal{P}^*$ is the Dirichlet realization of the differential expression in \eqref{QPE-1}. Using the definition of $\mathcal{P}^*$, we rewrite \eqref{QPE} as the equation \eqref{NPE} on unknown function $\tilde{u}$:
\be\label{NPE}
\tilde{u}'(t)=\mathcal{P}^*\tilde{u}(t)+[F(\tilde{u})](t)-\mathcal{P}^*\tilde{u}(t), \quad t\in[0, T],\quad \tilde{u}(0)=\tilde{u}_{0}\geq\epsilon_1>0.
\ee
Here $\epsilon_1$ is a given positive constant. We are going to show that the linearization operator $\mathcal{P}^*$ satisfies the elliptic estimate:
\begin{lem}\label{linearization-elliptic}
There exist positive constants $K$ and $K_{o}$  depending on $u_0$, $w_0$, such that $\forall\ q(t)\in D(\mathcal{P}^*)$ and $t\in[0, T]$,  $\mathcal{P}^*$ satisfies the elliptic estimate
\be\label{elliptic-est}
\left|\int_{\Omega}\frac{q(t)}{w_0}\nabla\cdot\left[w_0^3u_0\nabla q(t)\right]dx\right|\geq K\int_{\Omega}\left|\nabla q(t)\right|^2dx-\displaystyle K_{o}\int_{\Omega}|q(t)|^2dx.
\ee
\end{lem}
\begin{proof}
For $q(t)\in D(\mathcal{P}^*)$ and $t\in[0, T]$,
\[\mathcal{P}^*_hq(t)=\frac{1}{w_0}\nabla\cdot\left\{w_0^3u_0\nabla q(t)\right\},\quad \forall\ t\in[0, T],\]
denotes the highest order derivative term of $\mathcal{P}^*q(t)$, then by the divergence theorem, we obtain
\begin{align}
\displaystyle\int_{\Omega}\frac{q(t)}{w_0}\nabla\cdot\left\{w_0^3u_0\nabla q(t)\right\}dx=&\int_{\partial\Omega}\left\{w_0^2u_0q(t)\nabla q(t)\right\}\cdot\vec{n}dS-\int_{\Omega}\nabla\left[\frac{q(t)}{w_0}\right]\cdot\left\{w_0^3 u_0\nabla q(t)\right\}dx.\label{div1}
\end{align}
Since $q(t)\in D(\mathcal{P}^*)$, $D(\mathcal{P}^*)\subseteq H_*^2(\Omega)$ and $H_*^2(\Omega)=H^2(\Omega)\cap H_0^1(\Omega)$, then $q(t)\in H_0^1(\Omega)$,
hence
\[\displaystyle\int_{\partial\Omega}\left\{w_0^2u_0q(t)\nabla q(t)\right\}\cdot\vec{n}dS=0.\]
Because $u_0(x)\geq\epsilon_1>0$, $\epsilon_1$ is a given constant, and $\kappa=\displaystyle\inf_{x\in\overline{\Omega}}w_0(x)$,  \eqref{div1} becomes
\begin{align}
\left|\int_{\Omega}\frac{q(t)}{w_0}\nabla\cdot\left\{w_0^3u_0\nabla q(t)\right\}dx\right|
=&\left|\int_{\Omega}\nabla\left[\frac{q(t)}{w_0}\right]\cdot\left\{w_0^3 u_0\nabla q(t)\right\}dx\right|\notag\\
=&\left|\int_{\Omega}w_0 u_0\left\{w_0\nabla q(t)-q(t)\nabla w_0\right\}\cdot\nabla q(t)dx\right|\notag\\
\geq&\left|\int_{\Omega}u_0 w_0^2\left|\nabla q(t)\right|^2dx\right|-\left|\int_{\Omega}\displaystyle u_0w_0\nabla w_0\cdot[ q(t)\nabla q(t)]dx\right|\notag\\
\geq&\displaystyle{\epsilon_1\kappa^2}\int_{\Omega}\left|\nabla q(t)\right|^2dx-\left|\int_{\Omega}\displaystyle u_0w_0\nabla w_0\cdot[ q(t)\nabla q(t)]dx\right|\label{Fre-F-1}.
\end{align}
Notice that  $w_0\in H^4(\Omega)$ and $u_0\in H^2(\Omega)$ and write $C=C(\Omega)>0$ a constant, hence
\begin{align}
\left|\int_{\Omega}u_0w_0\nabla w_0\cdot[ q(t)\nabla q(t)]dx\right|\leq&\left\|u_0w_0\nabla w_0\right\|_{L^\infty(\Omega)}\left|\int_{\Omega}q(t)\nabla q(t)dx\right|\notag\\
\leq&C\left\|u_0w_0\nabla w_0\right\|_{H^2(\Omega)}\left|\int_{\Omega}q(t)\nabla q(t)dx\right|\notag\\
\leq&C\left\|u_0\right\|_{H^2(\Omega)}\left\| w_0\right\|_{H^2(\Omega)}\left\|\nabla w_0\right\|_{H^2(\Omega)}\left|\int_{\Omega}q(t)\nabla q(t)dx\right|\notag\\
\leq&C\left\|u_0\right\|_{H^2(\Omega)}\left\|w_0\right\|^2_{H^3(\Omega)}\left|\int_{\Omega}q(t)\nabla q(t)dx\right|\notag
\end{align}
With $K_2=C\left\|u_0\right\|_{H^2(\Omega)}\left\|w_0\right\|^2_{H^3(\Omega)}$ and Young's inequality,
\begin{align}
\left|\int_{\Omega}\frac{q(t)}{w_0}\nabla\cdot\left\{w_0^3u_0\nabla q(t)\right\}dx\right|\geq&(\epsilon_1\kappa^2-\varepsilon^2K_2)\int_{\Omega}\left|\nabla q(t)\right|^2dx-\displaystyle \frac{K_2}{4\varepsilon^2}\int_{\Omega}|q(t)|^2dx.\label{F}
\end{align}
The assertion \eqref{elliptic-est} follows for $\varepsilon$ sufficiently small.
\end{proof}
\begin{cor}\label{generatoroflinearization}
$\mathcal{P}^*$, defined by \eqref{linearopt}, is a sectorial operator and generates an analytic semigroup $\left\{e^{t\mathcal{P}^*}:\ t\geq0 \right\}$ on $H_*^2(\Omega)$.
\end{cor}
\begin{proof}
Using the Corollary 12.19 and Corollary 12.21 in \cite{GG}, we obtain that the operator $\mathcal{P}^*$ in Lemma \ref{linearization-elliptic} satisfies the elliptic estimate \eqref{elliptic-est}, as well as the following estimate for the resolvent set:
\be\label{sectorial-1}
\rho(\mathcal{P}^*)\supset S_{\Theta,\omega}=\left\{\lambda\in \mathds{C}: \lambda\neq\omega, |\arg(\lambda-\omega)|<\Theta, \omega\in\mathds{R}, \Theta\in \left(\frac{\pi}{2}, \pi\right) \right\}.
\ee
Proposition 1.22, Proposition 1.51 and Theorem 1.52 in \cite{OE} then imply the following estimates of its resolvent $(\lambda-\mathcal{P}^*)^{-1}$:
\be\label{sectorial-2}
\|(\lambda-\mathcal{P}^*)^{-1}\|_{\mathcal{B}\left(L^2(\Omega), H^2(\Omega)\right)}\leq \displaystyle \frac{M}{|\lambda-\omega|},
\ee
for $\omega\in \mathds{R}$, $M>0$ and $\lambda\in  S_{\Theta,\omega}$, and $\mathcal{P}^*$ is a sectorial operator which generates an analytic semigroup $\{e^{t\mathcal{P}^*}: t\geq 0\}$ on $H_*^2(\Omega)$. 
\end{proof}
\subsection*{Graph Norm of $\mathcal{P}^*$}
If the domain $D(\mathcal{P}^*)$ of $\mathcal{P}^*$ is endowed with the graph norm of $\mathcal{P}^*$, $\|g\|_{D(\mathcal{P}^*)}=\|g\|_{L^2(\Omega)}+\|\mathcal{P}^*g\|_{L^2(\Omega)}$, then there exists a constant $\gamma_0\geq 1$, such that
\be\label{graphnorm}
\displaystyle{\gamma_0}^{-1}\left(\|g\|_{L^2(\Omega)}+\|\mathcal{P}^*g\|_{L^2(\Omega)}\right)\leq \|g\|_{H^2(\Omega)}\leq \gamma_0\left(\|g\|_{L^2(\Omega)}+\|\mathcal{P}^*g\|_{L^2(\Omega)}\right).
\ee
In fact,  because $H_*^2(\Omega)\hookrightarrow L^2(\Omega)$, $\mathcal{P}^*\in\mathcal{B}\left(H_*^2(\Omega), L^2(\Omega)\right)$, there exists a constant $c_0>0$ such that
\[\|g\|_{L^2(\Omega)}+\|\mathcal{P}^*g\|_{L^2(\Omega)}\leq c_0\|g\|_{H^2(\Omega)},\quad \forall\ g\in H_*^2(\Omega),\ i.e.\ H_*^2(\Omega)\hookrightarrow D(\mathcal{P}^*).\]
Equations \eqref{sectorial-1} and \eqref{sectorial-2} imply that $\mathcal{P}^*$ is a closed operator, so that $D(\mathcal{P}^*)$ is a complete Banach space. We conclude $D\left(\mathcal{P}^*\right)= H_*^2(\Omega)$, which is assertion \eqref{graphnorm}.

Because $H_*^2(\Omega)=H^2(\Omega)\cap H_0^1(\Omega)$ is dense in $L^2(\Omega)$, we obtain $\mathcal{P}^*$ is densely defined in $L^2(\Omega)$ and $\overline{D\left(\mathcal{P}^*\right)}=L^2(\Omega)$.

If $t>0$ and $\psi\in L^2(\Omega)$ then $e^{t\mathcal{P}^*}\psi\in D\left(\left(\mathcal{P}{^*}\right)^k\right)$ for each $k\in\mathds{N}$. Moreover, there exist $\overline{M}_0$, $\overline{M}_1$, $\overline{M}_2>0$ (depending on $\Theta$ in \eqref{sectorial-1} and $M$ in \eqref{sectorial-2}), such that
\be\label{analyticsemigroupbound}
\left\|t^k\left(\mathcal{P}{^*}\right)^ke^{t\mathcal{P}^*}\right\|_{\mathcal{B}\left(L^2(\Omega)\right)}\leq \overline{M}_k,\quad s>0,\quad k=0,1,2,\quad t\in [0, T_0).
\ee

\begin{thm}\label{linear-parabolic-equation}
Let $\mathcal{P}^*:D(\mathcal{P}^*)\longrightarrow L^2(\Omega)$ be a sectorial operator and generate an analytic semigroup $e^{t\mathcal{P}^*}$, $ D\left(\mathcal{P}^*\right)\cong H_*^2(\Omega)$ and $\overline{D\left(\mathcal{P}^*\right)}=L^2(\Omega)$. If $T\in(0, T_0)$, $\alpha\in(0, 1)$ and
\[\tilde{u}_0\in D(\mathcal{P}^*),\quad  \mathcal{F}(0)+\mathcal{P}^*\tilde{u}_0\in \overline{D\left(\mathcal{P}^*\right)}, \quad \mathcal{F}\in C^\alpha\left([0, T]; L^2(\Omega)\right),\]
then
\be\label{LPE-mild-sol}
\varphi(t)=e^{t\mathcal{P}^*}\tilde{u}_0+\int_{0}^te^{(t-s)\mathcal{P}^*}\mathcal{F}(s)ds,
\ee
is the unique function belonging to $C^1([0, T]; L^2(\Omega))\cap C([0, T]; D(\mathcal{P}^*))$ which solves the problem
\be\label{4th-LPE}
\varphi'(t)=\mathcal{P}^* \varphi(t)+\mathcal{F}(t),\quad t\in[0, T],\quad \varphi(0)=\tilde{u}_0.
\ee
Moreover, the following maximal regularity property holds:
\[\mathcal{F}\in C^\alpha([0, T]; L^2(\Omega)),\quad \mathcal{P}^*\tilde{u}_0+\mathcal{F}(0)\in D_{\mathcal{P}^*}(\alpha, \infty)\Longrightarrow\]
\[\varphi\in C^{\alpha+1}([0, T]; L^2(\Omega))\cap C^\alpha([0, T]; H_*^2(\Omega)),\quad \varphi'(t)\in D_{\mathcal{P}^*}(\alpha, \infty),\quad \forall\ t\in[0, T],\]
and there exists a continuous and  increasing function $I:\mathds{R}_{+}\rightarrow \mathds{R}_{+}$ (depending on $\overline{M}_0$, $\overline{M}_1$, $\overline{M}_2$ and $\alpha$) such that
\be\label{linear-solu-est}
\|\varphi\|_{C^\alpha([0, T]; D(\mathcal{P}^*))}\leq I(T)\left[\|\tilde{u}_0\|_{L^2(\Omega)}+\|\mathcal{F}\|_{C^\alpha([0, T]; L^2(\Omega))}+\left\|\mathcal{P}^*\tilde{u}_0+\mathcal{F}(0)\right\|_{D_{\mathcal{P}^*}(\alpha, \infty)}\right].
\ee
\end{thm}
\begin{rem}
Theorem \ref{linear-parabolic-equation}, corresponding to Theorem 1.2 of Lunardi \cite{LS1}, is a maximal regularity result for linear autonomous evolution equations of parabolic type. Its proof follows the proof of Theorem 4.5 in Sinestrari \cite{SE}. We are going to use this result to prove the existence of a strict solution to the coupled system, which is Theorem \ref{cp-sys}. Before our proof,  we need Lemma \ref{RHSMax}. The detailed proof of Lemma \ref{RHSMax} can be found in Appendix \ref{AppA2}.
\end{rem}
\begin{lem}\label{RHSMax}
Let $F(\tilde{u})$ and $\mathcal{P}^*$ be defined by \eqref{F-def} and \eqref{linearopt} respectively, $T\in(0, T_0)$. Set $u_0=\tilde{u}_0+\theta_1$. If $\tilde{u}$, $q\in C^\alpha\left([0, T]; B_{H^2}\left(\tilde{u}_0,r\right)\right)$,
 then there exist postive constants $L_A=L_A\left(u_0, v_0, w_0, \Omega\right)$ and 
$L_B=L_B\left(u_0, v_0, w_0, \Omega, \alpha, T_0, L_U, L_W, L_M\right)$, such that $\forall\ 0\leq t<t+h\leq T$,
\be\label{Max-I}
\left\|\left[F(\tilde{u})\right](t+h)-\left[F(\tilde{u})\right](t)\right\|_{L^2(\Omega)}\leq \left\{[\tilde{u}+\theta_1]_{C^\alpha\left(\left[0, T\right]; H^2(\Omega)\right)}+L_U\right\}L_Ah^\alpha,
\ee
and
\begin{align}
&\left\|\left[F'(\tilde{u})q\right](t+h)-\left[F'(\tilde{u})q\right](t)-\mathcal{P}^*\left[q(t+h)-q(t)\right]\right\|_{L^2(\Omega)}\notag\\
\leq& h^\alpha T^\alpha L_{B}\left\|q\right\|_{C^\alpha([0, T]; H^2(\Omega))}+h^\alpha T^\alpha L_{B}\left\|\tilde{u}+\theta_1\right\|_{C^\alpha([0, T]; H^2(\Omega))}\left\|q\right\|_{C^\alpha([0, T]; H^2(\Omega))}\notag\\
+&h^\alpha L_B\sup_{t\in[0, T]}\left\|q(t)\right\|_{H^2(\Omega)}+h^\alpha L_B\left\|\tilde{u}+\theta_1\right\|_{C^\alpha([0, T]; H^2(\Omega))}\sup_{t\in[0, T]}\left\|q(t)\right\|_{H^2(\Omega)}.\label{Max-II}
\end{align}
Here the constants $L_U$, $L_W$ and $L_M$ are given by Corollary \ref{Holdercontinuity}, Theorem \ref{4th-cpl} and Corollary \ref{Holder-Frechet-W-I-Cor} respectively.
\end{lem}
\begin{thm}\label{cp-sys}
Assume the initial value $u_0\in\left\{\psi\in H^{2+\sigma}(\Omega):\ \psi(x)=\theta_1,\ x\in\partial\Omega\right\}$ is given for $\sigma\in\left(0,\ \frac{1}{2}\right)$ such that the compatibility condition
\[\frac{1}{w_0}\nabla\cdot\left(w_0^3u_0\nabla u_0\right)\in H^\sigma(\Omega)\subseteq D_{\mathcal{P}^*}(\alpha, \infty)\]
holds for $\alpha\in\left(0, \ \frac{\sigma}{2}\right)$ and $\tilde{u}_0=u_0-\theta_1\in H^{2+\sigma}(\Omega)\cap H_0^1(\Omega)$.

Then there exists $T_1>0$, such that the nonlinear problem \eqref{NPE} has a unique strict solution $\tilde{u}\in C^\alpha\left([0, T_1); H_*^2(\Omega)\right)\cap C^{\alpha+1}\left([0, T_1); L^2(\Omega)\right)$ and $\tilde{u}'(t)\in D_{\mathcal{P}^*}(\alpha, \infty)$, $\forall\ t\in[0, T_1)$. 
\end{thm}
\begin{proof}
We set $\sigma\in\left(0,\ \frac{1}{2}\right)$, $\alpha\in\left(0, \ \frac{\sigma}{2}\right)$ and divide the proof into three parts.\\
\noindent\textbf{H\"{o}lder Continuity.}
Let us first state some refinements of the results in Section \ref{SecSolnOp}. They concern the H\"{o}lder continuity of the solution operators
$\tilde{u}\longrightarrow v(\tilde{u}+\theta_1)$ and $\tilde{u}\longrightarrow w(\tilde{u}+\theta_1)$ needed later.
Take $T\in(0, T_0)$ to be specified below.
According to estimate \eqref{Holdercontinuityformular} of Corollary \ref{Holdercontinuity}, $v$ and $w$ are the solution operators satisfying
\[v:\ C^\alpha\left([0, T]; B_{H^2}(\tilde{u}_0, r)\right) \longrightarrow C^\alpha\left([0, T]; B_{L^2}(v_0, r)\right),\quad  \tilde{u}\longmapsto  v(\tilde{u}+\theta_1),\]
\[w:\ C^\alpha\left([0, T]; B_{H^2}(\tilde{u}_0, r)\right) \longrightarrow C^\alpha\left([0, T]; B_{H^2}(w_0, r)\right),\quad  \tilde{u}\longmapsto  w(\tilde{u}+\theta_1).\]
Thus, following the inequality \eqref{Max-I} in Lemma \ref{RHSMax},  $F(\tilde{u})$, defined by \eqref{F-def} satisfies
\[F(\tilde{u})\in C^\alpha\left([0, T]; L^2(\Omega)\right).\]
Theorem \ref{4th-cpl} with its following discussion about Fr\'{e}chet derivative and the estimate \eqref{Lip-Frechet-W-I} of Corollary \ref{Lip-Frechet-W} imply that the Fr\'{e}chet derivative $\left(v'(\tilde{u}+\theta_1)q, w'(\tilde{u}+\theta_1)q\right)$ of the function $\left(v(\tilde{u}+\theta_1), w(\tilde{u}+\theta_1)\right)$ on $\tilde{u}\in C\left([0, T]; H_*^2(\Omega)\right)$ exists in $C\left([0, T]; L^2(\Omega)\times H_*^2(\Omega)\right)$ and depends Lipschitz continuously on $\tilde{u}\in C\left([0, T]; H_*^2(\Omega)\right)$ for $q\in C\left([0, T]; H_*^2(\Omega)\right)$. If $\tilde{u}\in C^\alpha\left([0, T]; H_*^2(\Omega)\right)\cap C\left([0, T]; B_{H^2}\left(\tilde{u}_0, r\right)\right)$,  then by using inequality \eqref{Holder-Frechet-W-I} in Corollary \ref{Holder-Frechet-W-I-Cor}, $\left(v'(\tilde{u}+\theta_1)q, w'(\tilde{u}+\theta_1)q\right)\in C^\alpha\left([0, T]; L^2(\Omega)\times H_*^2(\Omega)\right)$, $\forall\ q\in C^\alpha\left([0, T]; H_*^2(\Omega)\right)$. Thus, following  \eqref{Max-II} in Lemma \ref{RHSMax},  the Fr\'{e}chet derivative $F'(\tilde{u})q$ of $F(\tilde{u})$ and $\mathcal{P}^*q$, defined by \eqref{Fre-F} and \eqref{linearopt} respectively, satisfy
\[F'(\tilde{u})q-\mathcal{P}^*q\in C^\alpha\left([0, T]; L^2(\Omega)\right),\quad \forall\ q\in C^\alpha\left([0, T]; H_*^2(\Omega)\right).\]
$v_0\in H_*^2(\Omega)$, $w_0\in H^4(\Omega)$ with $w_0(x)=\theta_1$, $\Delta w_0(x)=0$, $\forall$ $x\in \partial\Omega$ and the compatibility assumption of Theorem \ref{cp-sys} imply
\[[F(\tilde{u})](0)=\frac{1}{w_0}\nabla\cdot\left[w_0^3u_0\nabla u_0\right]-\frac{v_0}{w_0}u_0\in H^\sigma(\Omega)\subseteq D_{\mathcal{P}^*}(\alpha, \infty),\quad \tilde{u}_0\in D(\mathcal{P}^*).\]
By the definition \eqref{linearopt} of $\mathcal{P}^*$ and \eqref{graphnorm}, we know
\[D(\mathcal{P}^*)= H_*^2(\Omega),\ \overline{D(\mathcal{P}^*)}=L^2(\Omega),\ C^\alpha\left([0, T]; H_*^2(\Omega)\right)=C^\alpha\left([0, T]; D(\mathcal{P}^*)\right),\ \forall\ T\in (0, T_0).\]
\noindent\textbf{Equivalence.}
We now study the nonlinear problem
\be\label{4th-NLP}
\tilde{u}'(t)=\mathcal{P}^*\tilde{u}(t)+[F(\tilde{u})](t)-\mathcal{P}^*\tilde{u}(t), \quad t\in[0, T],\quad \tilde{u}(0)=\tilde{u}_{0}
\ee
whose integrated form is given by
\be\label{integralform}
\tilde{u}(t)=e^{t\mathcal{P}^*}\tilde{u}_0+\displaystyle\int_0^te^{(t-s)\mathcal{P}^*}\left\{[F(\tilde{u})](s)-\mathcal{P}^*\tilde{u}(s)\right\}ds,\quad t\in [0, T].
\ee
We will prove that if $\tilde{u}\in C^\alpha\left([0, T]; H_*^2(\Omega)\right)$ satisfies \eqref{integralform} and $\tilde{u}(t)\in B_{H^2}(\tilde{u}_0, r)$, $\forall\ t\in [0, T]$, $r\in\left(0, \frac{\kappa}{2C}\right)$, $\kappa=\displaystyle\inf_{x\in\overline{\Omega}}w_0(x)>0$ and $C=C(\Omega)>0$ is a constant, then $\tilde{u}\in C^{\alpha+1}\left([0, T]; L^2(\Omega)\right)$, $\tilde{u}'(t)\in D_{\mathcal{P}^*}(\alpha, \infty)$, $\forall\ t\in[0, T]$, and $\tilde{u}$ satisfies the equation \eqref{4th-NLP}. To prove this assertion, for each $\tilde{u}\in C^\alpha\left([0, T]; B_{H^2}(\tilde{u}_0, r)\right)$, we set
\be\label{pertubation}
[\mathcal{F}(\tilde{u})](t)=[F(\tilde{u})](t)-\mathcal{P}^*\tilde{u}(t),\quad\forall\ t\in[0, T],
\ee
and prove that
\be\label{HolderRHS}
\mathcal{F}(\tilde{u})\in C^\alpha\left([0, T]; L^2(\Omega)\right).
\ee
In fact, for $0\leq t<t+h\leq T$, by  using \eqref{Max-I} in Lemma \ref{RHSMax}, we have
\begin{align}
\left\|[\mathcal{F}(\tilde{u})](t+h)-[\mathcal{F}(\tilde{u})](t)\right\|_{L^2(\Omega)}\leq&\left\|[F(\tilde{u})](t+h)-[F(\tilde{u})](t)\right\|_{L^2(\Omega)}\notag\\
+&\left\|\mathcal{P}^*\tilde{u}(t+h)-\mathcal{P}^*\tilde{u}(t)\right\|_{L^2(\Omega)}\notag\\
\leq&\left(L_A+\left\|\mathcal{P}^*\right\|_{\mathcal{B}\left(H^2(\Omega), L^2(\Omega)\right)}\right)\left\|\tilde{u}(t+h)-\tilde{u}(t)\right\|_{H^2(\Omega)},\notag
\end{align}
because $\tilde{u}\in C^\alpha\left([0, T]; H_*^2(\Omega)\right)$, we get \eqref{HolderRHS}.

In addition, if $\tilde{u}(0)=\tilde{u}_0\in H_*^2(\Omega)= D(\mathcal{P}^*)$ and $[F(\tilde{u})](0)\in D_{\mathcal{P}^*}(\alpha, \infty)$, then we obtain
\[\mathcal{P}^*\tilde{u}_0+[\mathcal{F}(\tilde{u})](0)=[F(\tilde{u})](0)\in D_{\mathcal{P}^*}(\alpha, \infty).\]
Hence, by Theorem 1.2 of \cite{LS1} and  Theorem \ref{linear-parabolic-equation} we  conclude that if $\tilde{u}\in C^\alpha\left([0, T]; B_{H^2}(\tilde{u}_0,r )\right)$ is a solution of \eqref{integralform}, then there exist $\tilde{u}'\in C^{\alpha}\left([0, T]; L^2(\Omega)\right)$, $\tilde{u}'(t)\in D_{\mathcal{P}^*}(\alpha, \infty)$, $\forall\ t\in[0, T]$, and $\tilde{u}$ satisfies \eqref{4th-NLP}.

Conversely, let $\tilde{u}\in C^\alpha\left([0, T]; B_{H^2}(\tilde{u}_0,r)\right)\cap C^{\alpha+1}\left([0, T]; L^2(\Omega)\right)$ satisfy \eqref{4th-NLP}, i.e.
\[\tilde{u}'(t)=\mathcal{P}^*\tilde{u}(t)+[\mathcal{F}(\tilde{u})](t), \quad t\in[0, T],\quad \tilde{u}(0)=\tilde{u}_{0}.\]
As we have proved that $\mathcal{F}(\tilde{u})\in C^\alpha\left([0, T]; L^2(\Omega)\right)$, we can apply again Theorem \ref{linear-parabolic-equation} and deduce that $\tilde{u}$ is a solution of the integrated form \eqref{integralform}.

In conclusion, it is sufficient to solve \eqref{integralform} in the space $C^\alpha\left([0, T]; B_{H^2}(\tilde{u}_0, r)\right)$. To this end, we take $T\in(0, T_0)$ to be fixed later and find a fixed point for the mapping $\Gamma$ defined by
\be\label{integralform2}
\Gamma: Y\longrightarrow Y,\quad [\Gamma\tilde{u}](t)=e^{t\mathcal{P}^*}\tilde{u}_0+\displaystyle\int_0^te^{(t-s)\mathcal{P}^*}\left\{[F(\tilde{u})](s)-\mathcal{P}^*\tilde{u}(s)\right\}ds,\quad t\in [0, T],
\ee
\[Y=\left\{\tilde{u}\in C^\alpha\left([0, T]; H_*^2(\Omega)\right):\ \tilde{u}(0)=\tilde{u}_0,\ \left\|\tilde{u}(\cdot)-\tilde{u}_0\right\|_{C^\alpha([0, T]; H^2(\Omega))}\leq r\right\},\ \forall\ r\in\left(0, \frac{\kappa}{2C}\right).\]
\noindent\textbf{Contraction Mapping.}
If $Y$ is endowed with the metric induced by the norm of the space $C^\alpha\left([0, T]; H_*^2(\Omega)\right)$ we will show that $\Gamma$ is a contractive mapping of $Y$ into itself provided $T$ is sufficiently small.

From the preceding results and following the proof of Theorem 4.3.1 in \cite{LA2}, we know that $\Gamma u\in C^\alpha\left([0, T]; H_*^2(\Omega)\right)$ if $\tilde{u}\in Y$, because  $Y\subset C^\alpha\left([0, T]; B_{H^2}(\tilde{u}_0, r)\right)$. Now we will show that when $T$ is sufficiently small we have
\be\label{contractivemap}
\left\|\Gamma\tilde{u}_1-\Gamma\tilde{u}_2\right\|_{C^\alpha\left([0, T]; H^2(\Omega)\right)}\leq\frac{1}{2}\left\|\tilde{u}_1-\tilde{u}_2\right\|_{C^\alpha\left([0, T]; H^2(\Omega)\right)},\quad\forall\ \tilde{u}_1,\ \tilde{u}_2\in Y.
\ee
From \eqref{pertubation} and \eqref{integralform2}, we get
\[[\Gamma\tilde{u}_1](t)-[\Gamma\tilde{u}_2](t)=\displaystyle\int_0^te^{(t-s)\mathcal{P}^*}\left\{[\mathcal{F}(\tilde{u}_1)](s)-[\mathcal{F}(\tilde{u}_2)](s)\right\}ds,\quad t\in[0, T],\]
hence, by using \eqref{graphnorm} and applying \eqref{linear-solu-est} of Theorem \ref{linear-parabolic-equation}, we obtain
\begin{align}
\left\|\Gamma\tilde{u}_1-\Gamma\tilde{u}_2\right\|_{C^\alpha\left([0, T]; H^2(\Omega)\right)}\leq& \gamma_0I(T)\left\|\mathcal{F}(\tilde{u}_1)-\mathcal{F}(\tilde{u}_2)\right\|_{C^\alpha\left([0, T]; L^2(\Omega)\right)}\notag\\
\leq&\gamma_0I(T_0)\left\|\mathcal{F}(\tilde{u}_1)-\mathcal{F}(\tilde{u}_2)\right\|_{C^\alpha\left([0, T]; L^2(\Omega)\right)}.\label{Gammacontraction}
\end{align}
Here $I(\cdot)$ is a continuous and increasing function given by Theorem \ref{linear-parabolic-equation} when applied to $\mathcal{P}^*$ which is defined by \eqref{linearopt} and satisfies Lemma \ref{linearization-elliptic} and Corollary \ref{generatoroflinearization}. As $\tilde{u}_1(t)$ and $\tilde{u}_2(t)$ belong to $B_{H^2}(\tilde{u}_0, r)$ for $t\in[0, T]$, we can use inequality \eqref{nonlinearity-Lip-0} in Lemma \ref{Lip-nonlinearity} to estimate the right hand side, obtaining for all $t\in [0, T]$,
\begin{align}
\left\|[\mathcal{F}(\tilde{u}_1)](t)-[\mathcal{F}(\tilde{u}_2)](t)\right\|_{L^2(\Omega)}\leq&\left\|[F(\tilde{u}_1)](t)-[F(\tilde{u}_2)](t)\right\|_{L^2(\Omega)}+\left\|\mathcal{P}^*\tilde{u}_1(t)-\mathcal{P}^*\tilde{u}_2(t)\right\|_{L^2(\Omega)}\notag\\
\leq&\left(L_e+\left\|\mathcal{P}^*\right\|_{\mathcal{B}\left(H^2(\Omega), L^2(\Omega)\right)}\right)\left\|\tilde{u}_1(t)-\tilde{u}_2(t)\right\|_{H^2(\Omega)}.
\end{align}
As $\tilde{u}_1(0)=\tilde{u}_2(0)=\tilde{u}_0\in H_*^2(\Omega)= D(\mathcal{P}^*)$, then we have
\be\label{u1}
\sup_{t\in[0, T]}\left\|\tilde{u}_1(t)-\tilde{u}_2(t)\right\|_{H^2(\Omega)}\leq T^\alpha\left\|\tilde{u}_1-\tilde{u}_2\right\|_{C^\alpha\left([0, T]; H^2(\Omega)\right)},
\ee
and so
\be\label{F1}
\left\|\mathcal{F}(\tilde{u}_1)-\mathcal{F}(\tilde{u}_2)\right\|_{C\left([0, T]; L^2(\Omega)\right)}\leq \left[L_e+\left\|\mathcal{P}^*\right\|_{\mathcal{B}\left(H^2(\Omega), L^2(\Omega)\right)}\right]T^\alpha\left\|\tilde{u}_1-\tilde{u}_2\right\|_{C^\alpha\left([0, T]; H^2(\Omega)\right)}.
\ee
On the other hand for $0\leq t<t+h\leq T$, by using \eqref{Max-II} in Lemma \ref{RHSMax} and \eqref{u1}, we get
\begin{align}
&\left\|[\mathcal{F}(\tilde{u}_1)](t+h)-[\mathcal{F}(\tilde{u}_2)](t+h)-[\mathcal{F}(\tilde{u}_1)](t)+[\mathcal{F}(\tilde{u}_2)](t)\right\|_{L^2(\Omega)}\notag\\
=&\bigg\|\int_0^1[F'(\gamma\tilde{u}_1+(1-\gamma)\tilde{u}_2)(\tilde{u}_1-\tilde{u}_2)](t+h)-[F'(\gamma\tilde{u}_1+(1-\gamma)\tilde{u}_2)(\tilde{u}_1-\tilde{u}_2)](t)\notag\\
-&\mathcal{P}^*\left[\tilde{u}_1(t+h)-\tilde{u}_2(t+h)-\tilde{u}_1(t)+\tilde{u}_2(t)\right]d\gamma\bigg\|_{L^2(\Omega)}\notag\\
\leq&2h^\alpha T^\alpha L_{B}\left\|\tilde{u}_1-\tilde{u}_2\right\|_{C^\alpha([0, T]; H^2(\Omega))}\notag\\
+&2h^\alpha T^\alpha L_{B}\left\|\gamma\tilde{u}_1+(1-\gamma)\tilde{u}_2+\theta_1\right\|_{C^\alpha([0, T]; H^2(\Omega))}\left\|\tilde{u}_1-\tilde{u}_2\right\|_{C^\alpha([0, T]; H^2(\Omega))}\notag\\
\leq&2h^\alpha T^\alpha L_{B}\left[\left\|\tilde{u}_1-\tilde{u}_2\right\|_{C^\alpha([0, T]; H^2(\Omega))}+ \left(\|u_0\|_{H^2(\Omega)}+\kappa(2C)^{-1}\right)\left\|\tilde{u}_1-\tilde{u}_2\right\|_{C^\alpha([0, T]; H^2(\Omega))}\right]\label{contraction-2},
\end{align}
as well as
\begin{align}
&[\mathcal{F}(\tilde{u}_1)-\mathcal{F}(\tilde{u}_2)]_{C^\alpha([0, T]; L^2(\Omega))}\notag\\
=&\sup_{0\leq t<t+h\leq T}\frac{1}{h^\alpha}\left\{{\|[\mathcal{F}(\tilde{u}_1)](t+h)-[\mathcal{F}(\tilde{u}_2)](t+h)-[\mathcal{F}(\tilde{u}_1)](t)+[\mathcal{F}(\tilde{u}_2)](t)\|_{L^2(\Omega)}}\right\}\notag\\
\leq&2T^\alpha L_B\left[1+\left(\|u_0\|_{H^2(\Omega)}+\kappa(2C)^{-1}\right)\right]\left\|\tilde{u}_1-\tilde{u}_2\right\|_{C^\alpha([0, T]; H^2(\Omega))}.\label{contraction-5}
\end{align}
Hence we can deduce from \eqref{Gammacontraction}, \eqref{F1} and \eqref{contraction-5}:
\begin{align}
&\left\|\Gamma\tilde{u}_1-\Gamma\tilde{u}_2\right\|_{C^\alpha\left([0, T]; H^2(\Omega)\right)}\notag\\
\leq&\gamma_0I(T_0)\left\|\mathcal{F}(\tilde{u}_1)-\mathcal{F}(\tilde{u}_2)\right\|_{C^\alpha\left([0, T]; L^2(\Omega)\right)}\notag\\
\leq&\gamma_0I(T_0)\left[\left\|\mathcal{F}(\tilde{u}_1)-\mathcal{F}(\tilde{u}_2)\right\|_{C\left([0, T]; L^2(\Omega)\right)}+[\mathcal{F}(\tilde{u}_1)-\mathcal{F}(\tilde{u}_2)]_{C^\alpha([0, T]; L^2(\Omega))}\right]\notag\\
\leq&\gamma_0I(T_0)\left[L_e+\left\|\mathcal{P}^*\right\|_{\mathcal{B}\left(H^2(\Omega), L^2(\Omega)\right)}+2L_B\left[1+\left(\|u_0\|_{H^2(\Omega)}+\kappa(2C)^{-1}\right)\right]\right]T^\alpha\notag\\
\cdot&\left\|\tilde{u}_1-\tilde{u}_2\right\|_{C^\alpha([0, T]; H^2(\Omega))}.\label{finalcontraction}
\end{align}
Set
\be\label{timedef-1}
T_0^*:=\displaystyle{\left[2\gamma_0I(T_0)\left(L_e+\left\|\mathcal{P}^*\right\|_{\mathcal{B}\left(H^2(\Omega), L^2(\Omega)\right)}+2L_B\left[1+\left(\|u_0\|_{H^2(\Omega)}+\kappa(2C)^{-1}\right)\right]\right)\right]^{-\frac{1}{\alpha}}}.
\ee
If $0<T\leq \min\{T_0,\ T_0^*\}$, then $\Gamma$ satisfies the contraction property \eqref{contractivemap} by using \eqref{finalcontraction}.

To prove that $\Gamma(Y)\subseteq Y$, it remains to check that
\be\label{mapinto}
\left\|\Gamma\tilde{u}-\tilde{u}_0\right\|_{C^\alpha\left([0, T]; H^2(\Omega)\right)}\leq r, \quad \forall\ \tilde{u}\in Y.
\ee
Let us observe that if $0<T<\min\{T_0,\ T_0^*\}$, then from contraction property \eqref{contractivemap}, we get
\begin{align}
\left\|\Gamma\tilde{u}-\tilde{u}_0\right\|_{C^\alpha([0, T]; H^2(\Omega))}\leq&\left\|\Gamma\tilde{u}-\Gamma\tilde{u}_0\right\|_{C^\alpha([0, T]; H^2(\Omega))}+\left\|\Gamma\tilde{u}_0-\tilde{u}_0\right\|_{C^\alpha([0, T]; H^2(\Omega))}\notag\\
\leq&\frac{1}{2}\left\|\tilde{u}-\tilde{u}_0\right\|_{C^\alpha([0, T]; H^2(\Omega))}+\left\|\Gamma\tilde{u}_0-\tilde{u}_0\right\|_{C^\alpha([0, T]; H^2(\Omega))}\notag\\
\leq&\frac{r}{2}+\left\|\Gamma\tilde{u}_0-\tilde{u}_0\right\|_{C^\alpha([0, T]; H^2(\Omega))}\notag.
\end{align}
Now $\Gamma\tilde{u}_0-\tilde{u}_0\in C^\alpha\left([0, T]; H_*^2(\Omega)\right)$ and it vanishes at $t=0$, so there exists a $\delta^*=\delta^*(r)>0$ such that, if $0<T\leq\delta^*$, then
\[\left\|\Gamma\tilde{u}_0-\tilde{u}_0\right\|_{C^\alpha\left([0, T]; H^2(\Omega)\right)}\leq \frac{r}{2},\]
and consequently \eqref{mapinto} is true by choosing $0<T\leq\min\{T_0,\ T_0^*,\ \delta^*\}$. Because $r$ is controlled by $C=C(\Omega)$ and $\kappa$, we also note $\delta^*=\delta^*(\kappa,\Omega)$.

Summing up, set
\be\label{timedef}
T_1:=\min\left\{ T_0,\ T_0^*,\ \delta^*\right\},
\ee
$\Gamma$, defined by \eqref{integralform2},  is a contractive mapping of $Y$ into itself provided
\[0<T< T_1.\]
Hereby, $\Gamma$ has a unique fixed point $\tilde{u}$ in $Y$, $\tilde{u}\in C^\alpha\left([0, T_1); H_*^2(\Omega)\right)$ is a unique solution of the integral form \eqref{integralform}, and $\tilde{u}\in C^\alpha\left([0, T_1); H_*^2(\Omega)\right)\cap C^{\alpha+1}\left([0, T_1); L^2(\Omega)\right)$ is a unique strict solution of the nonlinear problem \eqref{4th-NLP}, by  preceding results in Equivalence, Theorem \ref{linear-parabolic-equation},  Theorem 1.2 of \cite{LS1} and Theorem 4.5 of \cite{SE}.

Recall that $[F(\tilde{u})](0)\in D_{\mathcal{P}^*}(\alpha, \infty)$,  $\tilde{u}_0\in H_*^2(\Omega)= D(\mathcal{P}^*)$, $F(\tilde{u})\in C^{\alpha}\left([0, T_1); L^2(\Omega)\right)$ with $L^2(\Omega)=\overline{D(\mathcal{P}^*)}$, as $\mathcal{P}^*\tilde{u}_0+[\mathcal{F}(\tilde{u})](0)=[F(\tilde{u})](0)=\tilde{u}'(0)$, Theorem 1.2 of \cite{LS1} and Theorem \ref{linear-parabolic-equation}  state in particular if $\tilde{u}'(t)\in D_{\mathcal{P}^*}(\alpha, \infty)$ for $t=0$, then the same is true for $t>0$, i.e. $\tilde{u}'(t)\in D_{\mathcal{P}^*}(\alpha, \infty)$, $\forall\ t\in[0, T_1)$, provided \[\mathcal{F}(\tilde{u})=F(\tilde{u})-\mathcal{P}^*\tilde{u}\in C^\alpha\left([0, T_1); L^2(\Omega)\right).\]
Meanwhile, $[F(\tilde{u})](t)\in D_{\mathcal{P}^*}(\alpha, \infty)$ as $\tilde{u}'(t)=[F(\tilde{u})](t)$ and $\tilde{u}'(t)\in D_{\mathcal{P}^*}(\alpha, \infty)$ for all $t\in[0, T_1)$.

Because of Theorem \ref{4th-mild-solution-cor} and $u=\tilde{u}+\theta_1$, the initial-boundary value problem  \eqref{cp2} has a unique strict solution $(u, w)$,
\[u\in C^\alpha\left([0, T_1); H^2(\Omega)\right)\cap C^{\alpha+1}\left([0, T_1); L^2(\Omega)\right),\]
\[{w}\in C\left([0, T_1); H^4(\Omega)\right)\cap C^{1}\left([0, T_1); H^2(\Omega)\right)\cap C^{2}\left([0, T_1); L^2(\Omega)\right).\]
This concludes the proof of Theorem \ref{cp-sys}.
\end{proof}
{Theorem \ref{cp-sys} directly implies Theorem \ref{coupled system}.}

\subsection*{{Maximal Time of Existence}}
\begin{cor}\label{Continuuation}
	{Assume that there are positive constants $C_\infty$ and $\delta_{\infty}$ such that} the solution $(u,w)$ from Theorem \ref{cp-sys} satisfies
	\begin{equation}\label{Final}
		u'(t)\in D_{\mathcal{P}^*}(\alpha,\ \infty),\quad \|u(t)\|_{H^{2+\alpha}(\Omega)}<C_{\infty},\quad \inf_\Omega w(t) > \delta_{\infty},
	\end{equation}
	for all $t \in [0,T_1]$. Then there exists $T_2>T_1$ such that $(u,w)$ uniquely extends to a solution of \eqref{cp2} on the time interval $[0,T_2)$.
\end{cor}
\begin{proof}
As \eqref{Final}, $u'(T_1)\in D_{\mathcal{P}^*}(\alpha,\ \infty)$, $\|u(T_1)\|_{H^{2+\alpha}(\Omega)}<C_{\infty}$, $\displaystyle\inf_\Omega w(T_1) > \delta_{\infty}$ and Theorem \ref{cp-sys} implies there exists $T_1^{'}$, such that the nonlinear coupled system \eqref{cp2}, with the admissible initial values $u(x,T_1)$,  $w(x,T_1)$, $\frac{\partial w}{\partial t}(x,T_1)$, has a solution $(u_1, w_1)$ on $[0, T_1^{'})$,
\[u_1\in C^\alpha\left([0, T_1^{'}); H^2(\Omega)\right)\cap C^{\alpha+1}\left([0, T_1^{'}); L^2(\Omega)\right),\]
\[{w}_1\in C\left([0, T_1^{'}); H^4(\Omega)\right)\cap C^{1}\left([0, T_1^{'}); H^2(\Omega)\right)\cap C^{2}\left([0, T_1^{'}); L^2(\Omega)\right).\]
Define the functions by
\[u_2(t)=\displaystyle\left\{\begin{array}{ll}
	u(t),\ &0\leq t\leq T_1,\\
	u_{1}(t-T_1),\ &T_1\leq t< T_1+T_1^{'},
\end{array}\right.
\quad w_2(t)=\displaystyle\left\{\begin{array}{ll}
	w(t),\ &0\leq t\leq T_1,\\
	w_{1}(t-T_1),\ &T_1\leq t< T_1+T_1^{'},
\end{array}\right.\]
$(u_2, w_2)$ is continuous and it is a strict solution of \eqref{cp2} for $t\in[0, T_1]$. For $t\in(T_1, T_1+T_1^{'})$,  set
\[\tilde{u}_2=u_2-\theta_1, \quad \begin{pmatrix}
	\tilde{v}_2\\ \tilde{w}_2
\end{pmatrix}=\begin{pmatrix}
\frac{\partial w_2}{\partial t}\\ w_2-\theta_2
\end{pmatrix}.\]
Based on Lemma 8.5 from \cite{schnaubelt}, we use the integral formulation \eqref{integralform} for $\tilde{u}(T_1)$ and calculate \eqref{Contin1}. Then we similarly get \eqref{Contin2} from the calculation for \eqref{Contin1} and the integral form \eqref{mild-solu-form}:
\begin{align}
	\tilde{u}_2(t)=\tilde{u}_1(t-T_1)&=e^{(t-T_1)\mathcal{P}^*}\tilde{u}(T_1)+\int_0^{t-T_1}e^{(t-T_1-s)\mathcal{P}^*}\left\{[F(\tilde{u_1})](s)-\mathcal{P}^*\tilde{u}_1(s)\right\}ds\notag\\
	&=e^{(t-T_1)\mathcal{P}^*}e^{T_1\mathcal{P}^*}\tilde{u}_0+e^{(t-T_1)\mathcal{P}^*}\int_0^{T_1}e^{(T_1-s)\mathcal{P}^*}\left\{[F(\tilde{u})](s)-\mathcal{P}^*\tilde{u}(s)\right\}ds\notag\\
	&+\int_{T_1}^{t}e^{(t-s)\mathcal{P}^*}\left\{[F(\tilde{u}_1)](s-T_1)-\mathcal{P}^*\tilde{u}_1(s-T_1)\right\}ds\notag\\
	&=e^{t\mathcal{P}^*}\tilde{u}_0+\int_{0}^{t}e^{(t-s)\mathcal{P}^*}\left\{[F(\tilde{u}_2)](s)-\mathcal{P}^*\tilde{u}_2(s)\right\}ds\label{Contin1},\\
	\begin{pmatrix}\tilde{v}_2(t)\\ \tilde{w}_2(t)\end{pmatrix}&=T(t)\begin{pmatrix}\tilde{v}_0\\ \tilde{w}_0\end{pmatrix}+
	\displaystyle\int_0^t\left\{T(t-s)\begin{pmatrix}[G(\tilde{w}_2)](s)+\beta_p\tilde{u}_2(s)\\ 0\end{pmatrix}\right\}ds.\label{Contin2}
\end{align}
Setting $T_2=T_1+T_1^{'}$, we conclude Corollary \ref{Continuuation}.
\end{proof}
Define the maximal existence time of the coupled system \eqref{cp2} by
\[T_{\max}=\sup\left\{T>0: \eqref{cp2}\ \text{has a solution}\ (u, w)\ \text{on}\ [0, T]\right\}.\]
Theorem \ref{cp-sys}, Corollary \ref{Continuuation} and Theorem 8.6 from \cite{schnaubelt} (page 80) {imply 
that if $T_{\max}<\infty$, 
\[\text{either}\ \lim_{t\to T_{\max}^-} \inf_{x\in\Omega}w(x,t)=0\ \text{or}\ \limsup_{t\to T_{\max}^{-}}\left\|u(t)\right\|_{L^\infty(\Omega)}=\infty.\]}

\appendix

\renewcommand{\theequation}{\thesection.\arabic{equation}}
\textcolor{black}{\section{Proofs of Lemmas in Section \ref{Sec:prerequisite}} \label{AppA}
	Before the proofs, we recall some well-known properties of the Sobolev spaces  $H^k(\Omega)$, where $k>0$.
\subsection{Sobolev Spaces and Algebraic Properties}
The algebra property of Sobolev spaces will be crucial in this work, see \cite{TT} for a proof.
\begin{lem}\label{alg}
	$H^k(\Omega)$ is an algebra whence $k>\frac{n}{2}$. In particular, $H^1(\Omega)$ is an algebra if $\Omega\subset\mathds{R}$ and $H^2(\Omega)$ is an algebra if $\Omega\subset \mathds{R}^n$, $n=1,\ 2$.
\end{lem}
We deduce some immediate consequences.
\begin{cor}\label{alg-1}
	If $f_1\in H^2(\Omega)$ and $f_2\in L^2(\Omega)$, then
	\be\label{alg-1-1}
	\left\|f_1f_2\right\|_{L^2(\Omega)}\leq C\left\|f_1\right\|_{H^2(\Omega)}\left\|f_2\right\|_{L^2(\Omega)}.
	\ee
	If $g_1$, $g_2$, $g_3\in H^2(\Omega)$, then
	\be\label{alg-1-2}
	\left\|g_1\nabla\cdot\left(g_2\nabla g_3\right)\right\|_{L^2(\Omega)}\leq2C\left\|g_1\right\|_{H^2(\Omega)}\left\|g_2\right\|_{H^2(\Omega)}\left\|g_3\right\|_{H^2(\Omega)}.
	\ee
\end{cor}
\begin{proof}
	The Sobolev embedding theorem implies
	\begin{align}
		\left\|f_1f_2\right\|_{L^2(\Omega)}&\leq\left\|f_1\right\|_{L^\infty(\Omega)}\left\|f_2\right\|_{L^2(\Omega)}\leq C\left\|f_1\right\|_{H^2(\Omega)}\left\|f_2\right\|_{L^2(\Omega)}\notag.
	\end{align}
	\begin{align}
		\left\|g_1\nabla\cdot\left(g_2\nabla g_3\right)\right\|_{L^2(\Omega)}&\leq\left\|g_1\right\|_{L^\infty(\Omega)}\left(\left\|\nabla g_2\cdot\nabla g_3\right\|_{L^2(\Omega)}+\left\|g_2\Delta g_3\right\|_{L^2(\Omega)}\right)\notag\\
		&\leq C\left\|g_1\right\|_{H^2(\Omega)}\left(\left|\int_{\Omega}\left|\nabla g_2\cdot\nabla g_3\right|^2dx\right|^{\frac{1}{2}}+\left\|g_2\right\|_{L^\infty(\Omega)}\left\|\Delta g_3\right\|_{L^2(\Omega)}\right)\notag\\
		&\leq C\left\|g_1\right\|_{H^2(\Omega)}\left[\left\|\nabla g_2\right\|_{L^4(\Omega)}\left\|\nabla g_3\right\|_{L^4(\Omega)}+\left\|g_2\right\|_{H^2(\Omega)}\left\|\Delta g_3\right\|_{L^2(\Omega)}\right]\notag\\
		&\leq C\left\|g_1\right\|_{H^2(\Omega)}\left[\left\|\nabla g_2\right\|_{H^1(\Omega)}\left\|\nabla g_3\right\|_{H^1(\Omega)}+\left\|g_2\right\|_{H^2(\Omega)}\left\|g_3\right\|_{H^2(\Omega)}\right]\notag\\
		&\leq2C\left\|g_1\right\|_{H^2(\Omega)}\left\|g_2\right\|_{H^2(\Omega)}\left\|g_3\right\|_{H^2(\Omega)}\notag.
	\end{align}
\end{proof}}
\subsection{Proof of Lemma \ref{estimates}}
\begin{proof}
	We first prove assertion \eqref{w-lower-bound} of Lemma \ref{estimates}.
	
	Since $ w\in C\left([0, T]; B_{H^2}\left({w}_0, r\right)\right)$,  then $\|w(t)-w_0\|_{H^{2}(\Omega)}\leq r$ holds for all $t\in [0, T]$.
	
	According to the triangle inequality and the Sobolev embedding theorem, there exists a constant $C=C(\Omega)$, such that for all $r\in\left(0, \frac{\kappa}{2C}\right)$, it follows that
	\bse\label{w-lower-bound1}
	\be\label{w-lower-bound11}
	w(t)=w_0+w(t)-w_0\geq\displaystyle\kappa-\left\|w(t)-w_0\right\|_{L^{\infty}(\Omega)}\geq\displaystyle\kappa-C\left\|w(t)-w_0\right\|_{H^{2}(\Omega)}\geq\kappa-Cr\geq\frac{\kappa}{2},
	\ee
	\be\label{w-lower-bound12}
	\left\|w(t)\right\|_{H^2(\Omega)}\leq\tilde{C},
	\ee
	\ese
	hold for all $t\in [0, T]$. Here, $\tilde{C}=\frac{\kappa}{2C}+\left\|w_0\right\|_{H^2(\Omega)}$.  Hereby, we prove assertion \eqref{w-lower-bound}.
	
	According to \eqref{w-lower-bound1} and $r\in\left(0, \frac{\kappa}{2C}\right)$, we have
	\begin{align}
		\displaystyle\sup_{t\in[0, T]}\left[\int_\Omega\left|\frac{1}{w(t)}\right|^2dx\right]\leq \frac{4C}{\kappa^2},\qquad \displaystyle\sup_{t\in[0, T]}\left[\int_\Omega\left|\displaystyle\nabla\left[\frac{1}{w(t)}\right]\right|^2dx\right]&=\sup_{t\in[0, T]}\left[\int_\Omega\frac{\left|\displaystyle\nabla w(t) \right|^2}{\left|w(t)\right|^4}dx\right]\notag\\
		&\leq\frac{16}{\kappa^4}\sup_{t\in[0, T]}\left[\int_\Omega\left|\displaystyle\nabla w(t)\right|^2dx\right]\notag\\
		&\leq\frac{16}{\kappa^4}\sup_{t\in[0, T]}\left\|w(t)\right\|_{H^2(\Omega)}^2\notag\\
		&\leq\frac{16}{\kappa^4}\tilde{C}^2\notag.
	\end{align}
	\begin{align}
		\displaystyle\sup_{t\in[0, T]}\left\{\int_\Omega \displaystyle\left|\Delta\left[\frac{1}{w(t)}\right]\right|^2dx\right\}
		\leq&\displaystyle\sup_{t\in[0, T]}\left[\left\|\frac{\Delta w(t)}{[w(t)]^2}\right\|_{L^2(\Omega)}
		+\left\|\frac{2\left(\nabla w(t)\right)^2}{[w(t)]^3}\right\|_{L^2(\Omega)}\right]^2\notag\\
		\leq&\displaystyle\sup_{t\in[0, T]}\left[\frac{4}{\kappa^2}\left\|\Delta w(t)\right\|_{L^2(\Omega)}+\frac{16}{\kappa^3}\left\|\left(\nabla w(t)\right)^2\right\|_{L^2(\Omega)}\right]^2\notag\\
		\leq&\displaystyle\sup_{t\in[0, T]}\left[\frac{4}{\kappa^2}\left\|w(t)\right\|_{H^2(\Omega)}+\frac{16}{\kappa^3}\left\|\nabla w(t)\right\|^2_{L^{4}(\Omega)}\right]^2\notag\\
		\leq&\displaystyle\sup_{t\in[0, T]}\left[\frac{4}{\kappa^2}\left\|w(t)\right\|_{H^2(\Omega)}+\frac{16C}{\kappa^3}\left\|\nabla w(t)\right\|^2_{H^{1}(\Omega)}\right]^2\notag\\
		\leq&\displaystyle\sup_{t\in[0, T]}\left[\frac{4}{\kappa^2}+\frac{16C}{\kappa^3}\left\|w(t)\right\|_{H^2(\Omega)}\right]^2\left\|w(t)\right\|^2_{H^2(\Omega)}\notag\\
		\leq&\displaystyle\left[\frac{4}{\kappa^2}+\frac{16C}{\kappa^3}\tilde{C}\right]^2\tilde{C}^2\notag.
	\end{align}
	\begin{align}
		\sup_{t\in[0, T]}\left\|\frac{1}{w(t)}\right\|^2_{H^2(\Omega)}=&\displaystyle\sup_{t\in[0, T]}\left\{\int_\Omega\left|\frac{1}{w(t)}\right|^2+\left|\nabla\left[\frac{1}{w(t)}\right]\right|^2+\left|\Delta\left[\frac{1}{w(t)}\right]\right|^2dx\right\}\notag\\
		\leq&\frac{4C}{\kappa^2}+\frac{16}{\kappa^4}\tilde{C}^2+\left[\frac{4}{\kappa^2}+\frac{16C}{\kappa^3}\tilde{C}\right]^2\tilde{C}^2.\label{w-1-1}
	\end{align}
	We set $C_1^2=\frac{4C}{\kappa^2}+\frac{16}{\kappa^4}\tilde{C}^2+\left[\frac{4}{\kappa^2}+\frac{16C}{\kappa^3}\tilde{C}\right]^2\tilde{C}^2$,
	$C_1$ is a positive constant depending on $\Omega$, $\kappa$,  and $\left\|w_0\right\|_{H^2(\Omega)}$. Because $H^2(\Omega)$ is an algebra, i.e. Lemma \ref{alg}, the assertion \eqref{C-a} of Lemma \ref{estimates} holds for $t\in [0, T]$.
	With these facts, we continue on to show the assertions \eqref{C-d} of Lemma \ref{estimates}. For all $w_1$ and $w_2\in C\left([0, T]; B_{H^2}(w_0,  r)\right)$, the algebraic property of $H^2(\Omega)$ from Lemma \ref{alg} and the triangle inequality imply
	\[\sup_{t\in[0, T]}\left\|\displaystyle \frac{\left[w_1(t)+w_2(t)\right]}{[w_1(t)]^{2}[w_2(t)]^{2}}\right\|_{H^2(\Omega)}\leq2 C^3_1,\
	\sup_{t\in[0, T]}\left\|\frac{[w_1(t)]^2+[w_2(t)]^2+w_1(t) w_2(t)}{[w_1(t)]^3 [w_2(t)]^3}\right\|_{H^2(\Omega)}\leq3C_1^4.\]
	Setting $C_2=2\displaystyle C^3_1$ and $C_3=3C^4_1$, hence we deduce \eqref{C-d} of Lemma \ref{estimates}.
	
	This concludes the proof of Lemma \ref{estimates}.
\end{proof}

\subsection{Proof of Lemma \ref{Lip-G-Lem}}
\begin{proof}
	Recall that $r\in\left(0, \frac{\kappa}{2C}\right)$, $w_0=\tilde{w}_0+\theta_2$, $\kappa=\displaystyle\inf_{\overline{\Omega}}w_0>0$, $w=\tilde{w}+\theta_2$, $\tilde{w}\in C\left([0, T]; B_{H^2}(\tilde{w}_0, r)\right)$. For small $h\in(0, T)$ such that $t+h\in(0, T]$, $\left\|\tilde{w}(t+h)-\tilde{w}_0\right\|_{H^2(\Omega)}\leq r$, \eqref{C-a} and \eqref{C-d} of Lemma \ref{estimates} imply  \eqref{Holdercontinuous} and \eqref{Lip-G} of Lemma \ref{Lip-G-Lem} are valid with $L_G=\beta_FC_2$.
	
	In particular, for $\tilde{w}_1\in C\left([0, T]; B_{H^2}\left(\tilde{w}_0,  r\right)\right)$, setting  $\tilde{w}_2(t)=\tilde{w}_0$, $\forall\ t\in[0, T]$, then one can obtain $[G(\tilde{w}_2)](t)=G(\tilde{w}_0)$. Hence \eqref{Lip-G-1} of Lemma \ref{Lip-G-Lem} is valid since the assertion \eqref{Lip-G} of Lemma \ref{Lip-G-Lem}.
	
	Set $\tilde{w}_2=\tilde{w}\in C\left([0, T]; B_{H^2}\left(\tilde{w}_0, r\right)\right)$, for $q\in C\left([0, T]; H_*^2(\Omega)\right)$, choose small $\lambda\in\mathds{R}$, such that
	\[\tilde{w}_1=\tilde{w}+\lambda q\in C\left([0, T]; B_{H^2}\left(\tilde{w}_0, r\right)\right).\]
	Then the Fr\'{e}chet derivative $G'\left(\tilde{w}\right)q$ of $G$ with respect to $\tilde{w}\in C\left([0, T]; B_{H^2}\left(\tilde{w}_0, r\right)\right)$ exists as a linear operator  $G'(\tilde{w}):\ C\left([0, T]; H_*^2(\Omega)\right)\longrightarrow C\left([0, T]; H_*^2(\Omega)\right)$ given by
	\[G'\left(\tilde{w}\right)q=\displaystyle\lim_{\lambda\rightarrow0}\frac{1}{\lambda}\left[G\left(\tilde{w}+\lambda q\right)-G\left(\tilde{w}\right)\right]=\frac{2\beta_F}{\left(\tilde{w}+\theta_2\right)^3}q,\] \[\left[G'\left(\tilde{w}\right)q\right](t)=\left[G'\left(\tilde{w}(t)\right)\right]q(t)=\frac{2\beta_F}{\left(\tilde{w}(t)+\theta_2\right)^3}q(t).\]
	According to the assertion \eqref{Lip-G}, the inequality \eqref{Lip-G-2} holds by the following computation:
	\begin{align}
		\sup_{t\in[0, T]}\left\|\left[G'\left(\tilde{w}\right)q\right](t)\right\|_{H^2(\Omega)}=&\sup_{t\in[0, T]}\left\|\displaystyle\lim_{\lambda\rightarrow0}\frac{[G\left(\tilde{w}+\lambda q\right)](t)-[G\left(\tilde{w}\right)](t)}{\lambda}\right\|_{H^2(\Omega)}\notag\\
		=&\displaystyle\lim_{\lambda\rightarrow0}\frac{1}{\lambda}\sup_{t\in[0, T]}\left\|[G\left(\tilde{w}_1\right)](t)-[G\left(\tilde{w}_2\right)](t)\right\|_{H^2(\Omega)}\notag\\
		\leq&\displaystyle\lim_{\lambda\rightarrow0}\frac{1}{\lambda}L_G\sup_{t\in[0, T]}\left\| \tilde{w}_1(t)-\tilde{w}_2(t)\right\|_{H^{2}(\Omega)}\notag\\
		=&L_G\sup_{t\in[0, T]}\left\|q(t)\right\|_{H^{2}(\Omega)}\notag.
	\end{align}
	For all $t\in[0, T]$, choose small $h\in\left(0, T\right)$ and $\tau\in[0, 1]$ such that $t+h\in(0, T]$,
	\[\left\|\tilde{w}(t)+\tau[\tilde{w}(t+h)-\tilde{w}(t)]-\tilde{w}_0\right\|_{H^2(\Omega)}\leq r,\]
	then for $\psi\in H_*^2(\Omega)$ with $\|\psi\|_{H^2(\Omega)}\leq1$, $G'(\tilde{w}(t)): \psi\in H_*^2(\Omega)\longrightarrow \left[G'(\tilde{w}(t))\right]\psi\in H_*^2(\Omega)$. By the algebraic properties of $H^2(\Omega)$, i.e. \eqref{C-a} and \eqref{C-d} from Lemma \ref{estimates}, we have
	\begin{align}
		&\left\|G'(\tilde{w}(t)+\tau[\tilde{w}(t+h)-\tilde{w}(t)])-G'(\tilde{w}(t))\right\|_{\mathcal{B}\left(H^2(\Omega)\right)}\notag\\
		=&\left\|\left[G'(\tilde{w}(t)+\tau[\tilde{w}(t+h)-\tilde{w}(t)])\right]\psi-\left[G'(\tilde{w}(t))\right]\psi\right\|_{H^2(\Omega)}\notag\\
		\leq&2\beta_F\sup_{\begin{smallmatrix}0\leq t< t+h\leq T\\ 0\leq \tau\leq 1\end{smallmatrix}}\left\|\frac{1}{\left({w}(t)+\tau[{w}(t+h)-{w}(t)]\right)^3}-\frac{1}{[{w}(t)]^3}\right\|_{H^2(\Omega)}\|\psi\|_{H^2(\Omega)}\notag\\
		\leq&2\beta_FC_3\sup_{0\leq t<t+h\leq T}\|\tilde{w}(t+h)-\tilde{w}(t)\|_{H^2(\Omega)}\notag.
	\end{align}
	Since  $\tilde{w}\in C\left([0, T]; H^2(\Omega)\right)$,  $\tilde{w}$ are uniformly continuous with respect to $t\in[0, T]$, hence the assertion \eqref{uniformly-continuous-Fre-G} is proved by
	\[\lim_{h\rightarrow 0^+}\sup_{0\leq t< t+ h\leq T}\|\tilde{w}(t+\tau h)-\tilde{w}(t)\|_{H^2(\Omega)}=0.\]
	This concludes the proof of Lemma \ref{Lip-G-Lem}.
	
\end{proof}

\subsection{Proof of Lemma \ref{Lip-nonlinearity}}
\begin{proof}
	Let $u_1$, $u_2\in C\left([0, T]; B_{H^2}(u_0, r)\right)$, according to the definitions of the operators $u\longmapsto v(u)$ and $u\longmapsto w(u)$,  it follows that 
	$(v_1, w_1)=(v(u_1), w(u_1))$ and $(v_2, w_2)=(v(u_2), w(u_2))$ belong to $ C\left([0, T]; B_{H^2}(v_0, r)\times B_{H^2}(w_0, r)\right)$.
	Setting that 
	$\tilde{C}=\left\|w_0\right\|_{H^2(\Omega)}+\frac{\kappa}{2C}$, $\tilde{C_1}=\left\|u_0\right\|_{H^2(\Omega)}+\frac{\kappa}{2C}$, $\tilde{C_2}=\left\|v_0\right\|_{L^2(\Omega)}+\frac{\kappa}{2C}$. Thus, $\forall\ t\in[0, T]$, $\left\|u_1(t)\right\|_{H^2(\Omega)}\leq\tilde{C_1}$, $\left\|u_2(t)\right\|_{H^2(\Omega)}\leq\tilde{C_1}$, $\left\|v_1(t)\right\|_{L^2(\Omega)}\leq\tilde{C_2}$, $\left\|v_2(t)\right\|_{L^2(\Omega)}\leq\tilde{C_2}$, $\left\|w_1(t)\right\|_{H^2(\Omega)}\leq\tilde{C}$, $\left\|w_2(t)\right\|_{H^2(\Omega)}\leq\tilde{C}$.
	
	Because the estimate $\left\|w_1(t)-w_2(t)\right\|_{H^2(\Omega)}\leq L_W\left\|u_1(t)-u_2(t)\right\|_{H^2(\Omega)}$ for all $t\in[0, T]$ and $H^2(\Omega)$ is an algebra for $\Omega\subset\mathds{R}^n$, $n=1,\ 2$, i.e. \eqref{C-a} and \eqref{C-d} in Lemma \ref{estimates}, we obtain  similar bounds for $[w_1(t)]^{-1}-[w_2(t)]^{-1}$ and $[w_1(t)]^3-[w_2(t)]^3$:
	\be\label{AppB01}
	\left\|[w_1(t)]^{-1}-[w_2(t)]^{-1}\right\|_{H^2(\Omega)}\leq C_1^2L_W\left\|u_1(t)-u_2(t)\right\|_{H^2(\Omega)},
	\ee
	\be\label{AppB02}
	\left\|[w_1(t)]^3-[w_2(t)]^3\right\|_{H^2(\Omega)}\leq3\tilde{C}^2L_W\left\|u_1(t)-u_2(t)\right\|_{H^2(\Omega)}.
	\ee
	Similarly, the algebraic property of $H^2(\Omega)$ from Lemma \ref{alg} implies
	\begin{align}
		\left\|[u_1(t)]^2-[u_2(t)]^2\right\|_{H^2(\Omega)}\leq2\tilde{C}_1\left\|u_1(t)-u_2(t)\right\|_{H^2(\Omega)}.\label{AppB1}
	\end{align}
	The algebraic properties of $H^2(\Omega)$, i.e. Lemma \ref{alg} and \eqref{alg-1-2} of Corollary \ref{alg-1}, imply
	\begin{align}
		&\frac{1}{2}\left\|\left\{[w_1(t)]^{-1}-[w_2(t)]^{-1}\right\}\nabla\cdot\left\{[w_1(t)]^3\nabla [u_1(t)]^2\right\}\right\|_{L^2(\Omega)}\notag\\
		\leq&C\displaystyle\left\|[w_1(t)]^{-1}-[w_2(t)]^{-1}\right\|_{H^2(\Omega)}\left\|[w_1(t)]^3\right\|_{H^2(\Omega)}\left\|[u_1(t)]^2\right\|_{H^2(\Omega)}\notag\\
		\leq&C\displaystyle C_1^2\tilde{C}^3\tilde{C}_1^2L_W\left\|u_1(t)-u_2(t)\right\|_{H^2(\Omega)}\label{AppB3}.
	\end{align}
	Similarly, set $\widehat{C}_1=3CC_1L_W\tilde{C}^2\tilde{C}_1^2$ and $\widehat{C}_2=2CC_1\tilde{C}^3\tilde{C}_1$, then
	\begin{align}
		\frac{1}{2}\left\|[w_2(t)]^{-1}\nabla\cdot\left\{\left([w_1(t)]^3-[w_2(t)]^3\right)\nabla[u_1(t)]^2\right\}\right\|_{L^2(\Omega)}\leq\widehat{C}_1\left\|u_1(t)-u_2(t)\right\|_{H^2(\Omega)}\notag,
	\end{align}
	\begin{align}
		\frac{1}{2}\left\|[w_2(t)]^{-1}\nabla\cdot\left\{[w_2(t)]^3\nabla\left([u_1(t)]^2-[u_2(t)]^2\right)\right\}\right\|_{L^2(\Omega)}\leq\widehat{C}_2\left\|u_1(t)-u_2(t)\right\|_{H^2(\Omega)}\notag.
	\end{align}
	Because of the estimates \eqref{AppB01} and $\left\|v_1(t)-v_2(t)\right\|_{L^2(\Omega)}\leq L_W\left\|u_1(t)-u_2(t)\right\|_{H^2(\Omega)}$, and $H^2(\Omega)$ is an algebra, i.e. \eqref{alg-1-1} in Corollary \ref{alg-1}, we obtain
	\begin{align}
		\left\|\frac{v_1(t)}{w_1(t)}u_1(t)-\frac{v_2(t)}{w_2(t)}u_2(t)\right\|_{L^2(\Omega)}\leq&C\left\|v_1(t)\right\|_{L^2(\Omega)}\left\|[w_1(t)]^{-1}\right\|_{H^2(\Omega)}\left\|u_1(t)-u_2(t)\right\|_{H^2(\Omega)}\notag\\
		+&C\left\|u_2(t)\right\|_{H^2(\Omega)}\left\|[w_1(t)]^{-1}\right\|_{H^2(\Omega)}\left\|v_1(t)-v_2(t)\right\|_{L^2(\Omega)}\notag\\
		+&C\left\|u_2(t)\right\|_{H^2(\Omega)}\left\|v_2(t)\right\|_{L^2(\Omega)}\left\|[w_1(t)]^{-1}-[w_2(t)]^{-1}\right\|_{H^2(\Omega)}\notag\\
		\leq&\widehat{C}_3\left\|u_1(t)-u_2(t)\right\|_{H^2(\Omega)}\label{AppB7-1}.
	\end{align}
	Here $\widehat{C}_3=C\left(\tilde{C}_2C_1+\tilde{C}_1C_1L_W+\tilde{C}_1\tilde{C}_2C_1^2L_W\right)$.
	
	Consequently, \eqref{nonlinearity-Lip-0} holds by setting $L_e=CC_1^2\tilde{C}^3\tilde{C}_1^2L_W+\widehat{C}_1+\widehat{C}_2+\widehat{C}_3$ and computing
	\begin{align}
		\|[f(u_1)](t)-[f(u_2)](t)\|_{L^2(\Omega)}\leq&\frac{1}{2}\left\|\left\{[w_1(t)]^{-1}-[w_2(t)]^{-1}\right]\nabla\cdot\left\{[w_1(t)]^3\nabla [u_1(t)]^2\right\}\right\|_{L^2(\Omega)}\notag\\
		+&\frac{1}{2}\left\|[w_2(t)]^{-1}\nabla\cdot\left\{\left([w_1(t)]^3-[w_2(t)]^3\right)\nabla[u_1(t)]^2\right\}\right\|_{L^2(\Omega)}\notag\\
		+&\frac{1}{2}\left\|[w_2(t)]^{-1}\nabla\cdot\left\{[w_2(t)]^3\nabla\left([u_1(t)]^2-[u_2(t)]^2\right)\right\}\right\|_{L^2(\Omega)}\notag\\
		+&\left\|\frac{v_1(t)}{w_1(t)}u_1(t)-\frac{v_2(t)}{w_2(t)}u_2(t)\right\|_{L^2(\Omega)}\notag\\
		\leq&L_e\left\|u_1(t)-u_2(t)\right\|_{H^2(\Omega)}.\label{AppB2}
	\end{align}
	This concludes the proof of Lemma \ref{Lip-nonlinearity}.
\end{proof}

\section{\textcolor{black}{Proofs of Results in Section \ref{4th-order problem} }}\label{AppA1}
\subsection{Proof of Lemma \ref{equivalence-1}}
\begin{proof}
	Let $\tilde{w}\in C^2\left([0, T]; L^2(\Omega)\right)\cap C^1\left([0, T]; H_*^2(\Omega)\right)\cap C\left([0, T]; H_*^4(\Omega)\right)$ solve the semilinear fourth-order equation \eqref{4th-LWE-1}. Then $\Phi:=\left(\tilde{w}', \tilde{w}\right)\in C^1([0, T]; \mathfrak{X})$, $\Phi(t)=\left(\tilde{w}'(t), \tilde{w}(t)\right)\in D(\mathcal{A})$ for all $t\in [0, T]$, $\Phi\in C([0, T]; D(\mathcal{A}))$ and
	\begin{align}
		\Phi'(t)=\begin{pmatrix}\tilde{w}^{''}(t)\\ \tilde{w}'(t)\end{pmatrix}
		=\begin{pmatrix}\mathbb{A}\tilde{w}(t)-\frac{\beta_F}{(\tilde{w}(t)+\theta_2)^2}+\beta_p(\tilde{u}(t)+\theta_1-1)\\ \tilde{w}'(t)\end{pmatrix}
		=\mathcal{A}\Phi(t)+[\mathcal{G}(\Phi)](t),\notag
	\end{align}
	holds for all $t\in [0, T]$. Moreover, $\Phi(0)=\left(\tilde{w}'(0), \tilde{w}(0)\right)=\left(\tilde{v}_0, \tilde{w}_0\right)=\Phi_0.$
	
	Therefore, $\Phi\in C([0, T]; D(\mathcal{A}))\cap C^1([0, T]; \mathfrak{X})$ solves the equation \eqref{4th-IEE}.
	
	Conversely, let $\Phi=\left(\varphi_1, \varphi_2\right)\in C^1([0, T]; \mathfrak{X})\cap C([0, T]; D(\mathcal{A}))$ solves the semilinear evolution equation \eqref{4th-IEE}. We set $\tilde{w}:=\varphi_2$, obtaining $\tilde{w}\in C^1\left([0, T]; H_*^2(\Omega)\right),\ \tilde{w}(t)\in H_*^4(\Omega)$, $\forall\ t\in [0, T]$, and $\tilde{w}\in C\left([0, T]; H_*^4(\Omega)\right)$. It further follows, $\forall\ t\in[0, T]$,
	\begin{align}
		\begin{pmatrix}\varphi'_1(t)\\ \tilde{w}'(t)\end{pmatrix}=\mathcal{A}\begin{pmatrix}\varphi_1(t)\\ \tilde{w}(t)\end{pmatrix}+[\mathcal{G}(\Phi)](t)
		=\begin{pmatrix}\mathbb{A}\tilde{w}(t)-\frac{\beta_F}{(\tilde{w}(t)+\theta_2)^2}+\beta_p(\tilde{u}(t)+\theta_1-1)\\ \varphi_1(t)\end{pmatrix}.\notag
	\end{align}
	Thus, $\tilde{w}'=\varphi_1\in C^1\left([0, T]; L^2(\Omega)\right)$ and $\Phi=\left(\tilde{w}', \tilde{w}\right)$, so that
	$\tilde{w}\in C^2\left([0, T]; L^2(\Omega)\right)$, $\left(\tilde{w}'(0), \tilde{w}(0)\right)=\left(\tilde{v}_0, \tilde{w}_0\right)$.
	So $\tilde{w}\in C^2\left([0, T]; L^2(\Omega)\right)\cap C^1\left([0, T]; H_*^2(\Omega)\right)\cap C\left([0, T]; H_*^4(\Omega)\right)$
	solves the semilinear fourth-order equation \eqref{4th-LWE-1}.
	
	This equivalence also yields that the solutions to the semilinear evolution equation \eqref{4th-IEE} are unique if and only if the solutions to the semilinear fourth-order equation \eqref{4th-LWE-1} are unique.
\end{proof}

\subsection{Proof of Lemma \ref{generator}}
\begin{proof}
	We aim to show that $\mathcal{A}$, defined by \eqref{A-1}, is skew adjoint on the Hilbert space $\mathfrak{X}$ defined by \eqref{state-space}, and thus generates a strongly continuous semigroup ($C_0$-semigroup) on $\mathfrak{X}$ by using Stone's Lemma (see 3.24 Theorem, Section 3, Chapter II, \cite{EK}).
	
	From the definition \eqref{A-1} of $\mathcal{A}$, $\mathcal{A}$ is densely defined in $\mathfrak{X}$, i.e. $\overline{D(\mathcal{A})}=\mathfrak{X}$,  then $\mathcal{A}$ is skew symmetric (i.e. $\mathbf{i}\mathcal{A}$ is symmetric) for any two $(\phi_1, \phi_2)\in D(\mathcal{A})$ and $(\psi_1, \psi_2 )\in D(\mathcal{A})$ by the following computations:
	\begin{align}
		\left\langle\mathcal{A}\begin{pmatrix}\phi_1\\ \phi_2\end{pmatrix}, \begin{pmatrix}\psi_1\\ \psi_2 \end{pmatrix}\right\rangle_{\mathfrak{X}}&=\left\langle\begin{pmatrix}\mathbb{A}\phi_2\\ \phi_1 \end{pmatrix}, \begin{pmatrix}\psi_1\\ \psi_2 \end{pmatrix}\right\rangle_{\mathfrak{X}}\notag\\
		&=\displaystyle\int_\Omega\mathbb{A}\phi_2\cdot\overline{\psi_1}+\nabla\phi_1\cdot\nabla\overline{\psi_2}+\Delta\phi_1\cdot\Delta\overline{\psi_2}dx\notag\\
		&=\displaystyle\int_\Omega -\nabla\phi_2\cdot\nabla\overline{\psi_1}-\Delta\phi_2\cdot\Delta\overline{\psi_1}+\nabla\phi_1\cdot\nabla\overline{\psi_2}+\Delta\phi_1\cdot\Delta\overline{\psi_2}dx\notag\\
		&=-\displaystyle\left[\int_\Omega \nabla\phi_2\cdot\nabla\overline{\psi_1}+\Delta\phi_2\cdot\Delta\overline{\psi_1}-\nabla\phi_1\cdot\nabla\overline{\psi_2}-\Delta\phi_1\cdot\Delta\overline{\psi_2}dx\right]\notag\\
		&=-\displaystyle\left[\int_\Omega \nabla\phi_2\cdot\nabla\overline{\psi_1}+\Delta\phi_2\cdot\Delta\overline{\psi_1}+\phi_1\cdot\overline{\mathbb{A}\psi_2}dx\right]\notag\\
		&=-\left\langle\begin{pmatrix}\phi_1\\ \phi_2 \end{pmatrix}, \begin{pmatrix}\mathbb{A}\psi_2\\ \psi_1 \end{pmatrix}\right\rangle_{\mathfrak{X}}\notag\\
		&=-\left\langle\begin{pmatrix}\phi_1\\ \phi_2 \end{pmatrix}, \mathcal{A}\begin{pmatrix}\psi_1\\ \psi_2 \end{pmatrix}\right\rangle_{\mathfrak{X}}.\notag
	\end{align}
	Furthermore, $\text{Re}\left\langle\mathcal{A}\begin{pmatrix}\psi\\ \phi\end{pmatrix}, \begin{pmatrix}\psi\\ \phi \end{pmatrix}\right\rangle_{\mathfrak{X}}=0$ for all $(\psi, \phi )\in D(\mathcal{A})$, so $\mathcal{A}$ is dissipative. By using the Lax-Milgram Theorem (Theorem 1, Section 6.2, \cite{EL}), we have the inverse $\mathbb{A}^{-1}$ of $\mathbb{A}$ exists, thus we define an operator
	\[\mathcal{R}=\begin{pmatrix}0 &1\\ \mathbb{A}^{-1} &0\end{pmatrix}.\]
	Then
	\[\mathcal{R}\mathfrak{X}\subset D(\mathcal{A}),\ \mathcal{A}\mathcal{R}=I, \ \mathcal{R}\mathcal{A}\begin{pmatrix}\psi\\ \phi \end{pmatrix}=\begin{pmatrix}\psi\\ \phi \end{pmatrix},\ \forall\ \left(\psi, \phi \right)\in D(\mathcal{A}).\]
	Therefore, $\mathbf{i}\mathcal{A}$ is invertible and the resolvent set $\rho(\mathbf{i}\mathcal{A})$ of $\mathbf{i}\mathcal{A}$ satisfies $\rho(\mathbf{i}\mathcal{A})\cap \mathds{R}\neq\emptyset$, so the spectrum $\sigma(\mathbf{i}\mathcal{A})\subseteq\mathds{R}$, consequently, $\mathbf{i}\mathcal{A}$ is selfadjoint, as a result, $\mathcal{A}$ is skew adjoint. According to Stone's Lemma, we have the linear operator $\mathcal{A}$ generates a $C_0$-semigroup
	\[\left\{T(t)\in \mathcal{B}(\mathfrak{X}):\ t\in[0, \infty)\right\}.\]
\end{proof}

\subsection{Proof of Theorem \ref{4th-solu-thm}}
\begin{proof}
	We let $T\in(0, \infty)$ be taken to be specified below. Because $\mathcal{A}$, defined by \eqref{A-1}, generates a strongly continuous semigroup ($C_0$-semigroup) $\left\{T(t)\in\mathcal{B}\left(L^2(\Omega)\times H_*^2(\Omega)\right): t\in[0, \infty)\right\}$, and $[G\left(\tilde{w}\right)](t)=-\beta_F[\tilde{w}(t)+\theta_2]^{-2}+\beta_p\left(\theta_1-1\right)$, we introduce a nonlinear operator $\Phi$ on $\mathcal{Z}(T)$ by
	\[\left[\Phi(\tilde{v},\tilde{w})\right](t):=T(t)\begin{pmatrix}\tilde{v}_0\\ \tilde{w}_0\end{pmatrix}+\displaystyle\int_0^t\left\{T(t-s)\begin{pmatrix}[G(\tilde{w})](s)+\beta_p\tilde{u}(s)\\ 0\end{pmatrix}\right\}ds,\quad \forall\ t\in [0,T].\]
	We notice that
	\[\left[\Phi(\tilde{v},\tilde{w})\right](0)=T(0)\begin{pmatrix}\tilde{v}_0\\ \tilde{w}_0\end{pmatrix}=\begin{pmatrix}\tilde{v}_0\\ \tilde{w}_0\end{pmatrix}\in D(\mathcal{A}).\]
	According to Lemma 1.3 of Chapter II in \cite{EK},
	\[T(t)\begin{pmatrix}\tilde{v}_0\\ \tilde{w}_0\end{pmatrix}\in D(\mathcal{A}),\ \forall\ t\in[0, T].\]
	Since $(\tilde{v},\tilde{w})\in \mathcal{Z}(T)$, $\tilde{u}\in C\left([0, T]; B_{H^2}\left(\tilde{u}_0, r\right)\right)$ such that
	$G(\tilde{w})+\beta_p\tilde{u}\in C([0, T]; H^2(\Omega))$, hence
	\[\begin{pmatrix}[G(\tilde{w})](t)+\beta_p\tilde{u}(t)\\ 0\end{pmatrix},\quad \int_0^t\left\{T(t-s)\begin{pmatrix}[G(\tilde{w})](s)+\beta_p\tilde{u}(s)\\ 0\end{pmatrix}\right\}ds\in L^2(\Omega)\times H_*^2(\Omega).\]
	Therefore, $\Phi$ is a nonlinear operator which maps $\mathcal{Z}(T)$ into $C\left([0,T]; L^2(\Omega)\times H_*^2(\Omega)\right)$.
	
	We next show that there exists a unique mild solution $(\tilde{v},\tilde{w})\in \mathcal{Z}(T)$ of the semilinear evolution equation \eqref{4th-SWE-1} which is a fixed point of $\Phi$ on $\mathcal{Z}(T)$.
	
	We denote by $M_0=\displaystyle\sup_{t\in [0, \infty )}\|T(t)\|_{\mathcal{B}\left(L^2(\Omega)\times H^2(\Omega)\right)}$ an operator norm of $\left\{T(t)\right\}_{0\leq t<\infty}$ on the space $L^2(\Omega)\times H_*^2(\Omega)$. For given $\tilde{u}\in C\left([0, T]; B_{H^2}\left(\tilde{u}_0, r\right)\right)$, if $(\tilde{v}_1,\tilde{w}_1),\ (\tilde{v}_2,\tilde{w}_2)\in \mathcal{Z}(T)$, then\\
	$[G(\tilde{w}_1)](t)-[G(\tilde{w}_2)](t)\in H^2(\Omega)$, $\forall\ t\in[0, T]$.
	By using the estimate \eqref{Lip-G} of Lemma \ref{estimates}, we obtain
	\begin{align}
		&\sup_{t\in[0, T]}\left\|[\Phi(\tilde{v}_1,\tilde{w}_1)](t)-[\Phi(\tilde{v}_2,\tilde{w}_2)](t)\right\|_{L^2(\Omega)\times H^{2}(\Omega)}\notag\\
		=&\sup_{t\in[0, T]}\left\|\int_0^tT(t-s)\begin{pmatrix}[G(\tilde{w}_1)](s)+\beta_p\tilde{u}(s)-[G(\tilde{w}_2)](s)-\beta_p\tilde{u}(s)\\ 0\end{pmatrix}ds\right\|_{L^{2}(\Omega)\times H^{2}(\Omega)}\notag\\
		\leq&TM_0\sup_{t\in[0, T]}\left\|[G(\tilde{w}_1)](t)-[G(\tilde{w}_2)](t)\right\|_{L^{2}(\Omega)}\notag\\
		\leq&TM_0\sup_{t\in[0, T]}\left\|[G(\tilde{w}_1)](t)-[G(\tilde{w}_2)](t)\right\|_{H^{2}(\Omega)}\notag\\
		\leq&TM_0L_G\sup_{t\in[0, T]}\left\| \tilde{w}_1(t)-\tilde{w}_2(t)\right\|_{H^{2}(\Omega)}\notag\\
		\leq&TM_0L_G\sup_{t\in[0, T]}\left\|\begin{pmatrix}\tilde{v}_1(t)-\tilde{v}_2(t)\\ \tilde{w}_1(t)-\tilde{w}_2(t)\end{pmatrix}\right\|_{L^{2}(\Omega)\times H^{2}(\Omega)}.\label{difference-1}
	\end{align}
	Because $\{T(t)\in\mathcal{B}\left(L^2(\Omega)\times H_*^2(\Omega)\right):\ t\geq 0\}$ is a strongly continuous semigroup,  according to the definition of strong continuity, for $(\tilde{v}_0, \tilde{w}_0)\in D(\mathcal{A})$ and given constant $r\in\left(0, \frac{\kappa}{2C}\right)$, there exists $\delta_o=\delta_o(r)>0$, such that if $0<t\leq\delta_o$, then
	\be\label{semigp-continuity}
	0<\left\|T(t)\begin{pmatrix}\tilde{v}_0\\ \tilde{w}_0\end{pmatrix}-\begin{pmatrix}\tilde{v}_0\\ \tilde{w}_0\end{pmatrix}\right\|_{L^2(\Omega)\times H^{2}(\Omega)}\leq\frac{r}{2}.
	\ee
	Since $r\in\left(0, \frac{\kappa}{2C}\right)$ and $C=C(\Omega)$ is a constant, $\delta_o$ depends on $\kappa$ and $\Omega$, i.e. $\delta_o=\delta_o(\kappa, \Omega)$.
	
	As $\tilde{u}\in C\left([0, T]; B_{H^2}\left(\tilde{u}_0, r\right)\right)$ and $(\tilde{v}_1,\tilde{w}_1)\in \mathcal{Z}(T)$, then $\tilde{v}_1(t)\in L^2(\Omega)$, $\tilde{w}_1(t),\ \tilde{u}(t)\in H_*^2(\Omega)$, $\forall\ t\in[0, T]$, and $\displaystyle\sup_{t\in[0, T]}\left\|\tilde{u}(t)-\tilde{u}_0\right\|_{H^2(\Omega)}\leq r$, thus $[G(\tilde{w}_1)](t)+\beta_p\tilde{u}(t)\in H^2(\Omega)$, $\forall\ t\in[0, T]$. Because $G_0=G(\tilde{w}_0)+\beta_p\tilde{u}_0\in H^2(\Omega)$, the inequality \eqref{Lip-G-1} of Lemma \ref{estimates} implies
	\begin{align}
		&\sup_{t\in[0, T]}\left\|\left[\Phi(\tilde{v}_1,\tilde{w}_1)\right](t)-\begin{pmatrix}\tilde{v}_0\\ \tilde{w}_0\end{pmatrix}\right\|_{L^2(\Omega)\times H^{2}(\Omega)}\notag\\
		=&\sup_{t\in[0, T]}\left\|T(t)\begin{pmatrix}\tilde{v}_0\\ \tilde{w}_0\end{pmatrix}-\begin{pmatrix}\tilde{v}_0\\ \tilde{w}_0\end{pmatrix}+\displaystyle\int_0^t\left\{T(t-s)\begin{pmatrix}[G(\tilde{w}_1)](s)+\beta_p\tilde{u}(s)\\ 0\end{pmatrix}\right\}ds\right\|_{L^2(\Omega)\times H^{2}(\Omega)}\notag\\
		\leq&\sup_{t\in[0, T]}\left\|T(t)\begin{pmatrix}\tilde{v}_0\\ \tilde{w}_0\end{pmatrix}-\begin{pmatrix}\tilde{v}_0\\ \tilde{w}_0\end{pmatrix}\right\|_{L^2(\Omega)\times H^{2}(\Omega)}+\displaystyle TM_0\left\|G_0\right\|_{ L^{2}(\Omega)}\notag\\
		+&\displaystyle TM_0\sup_{t\in[0, T]}\left\|[G(\tilde{w}_1)](t)-G(\tilde{w}_0)\right\|_{L^2(\Omega)}+TM_0\sup_{t\in[0, T]}\left\|\tilde{u}(t)-\tilde{u}_0\right\|_{L^2(\Omega)}\notag\\
		\leq&\sup_{t\in[0, T]}\left\|T(t)\begin{pmatrix}\tilde{v}_0\\ \tilde{w}_0\end{pmatrix}-\begin{pmatrix}\tilde{v}_0\\ \tilde{w}_0\end{pmatrix}\right\|_{L^2(\Omega)\times H^{2}(\Omega)}+\displaystyle TM_0\left\|G_0\right\|_{ H^{2}(\Omega)}\notag\\
		+&\displaystyle TM_0\sup_{t\in[0, T]}\left\|[G(\tilde{w}_1)](t)-G(\tilde{w}_0)\right\|_{H^2(\Omega)}+TM_0\sup_{t\in[0, T]}\left\|\tilde{u}(t)-\tilde{u}_0\right\|_{H^2(\Omega)}\notag\\
		\leq&\sup_{t\in[0, T]}\displaystyle\left\|T(t)\begin{pmatrix}\tilde{v}_0\\ \tilde{w}_0\end{pmatrix}-\begin{pmatrix}\tilde{v}_0\\ \tilde{w}_0\end{pmatrix}\right\|_{L^2(\Omega)\times H^{2}(\Omega)}+TM_0\left(\left\|G_0\right\|_{ H^{2}(\Omega)}+ \left(L_G+1\right) r\right)\label{estimate of phi-1}.
	\end{align}
	For fixed small $r\in\left(0, \frac{\kappa}{2C}\right)$, then there exists a number $T_0>0$,
	\be\label{T-0}
	T_0=\inf\left\{\delta_o,\ \frac{1}{2M_0L_G},\
	\frac{\kappa}{2M_0}\left[\left(L_G+1\right)\kappa+2C\left\|G_0\right\|_{ H^{2}(\Omega)}\right]^{-1}\right\},
	\ee
	such that for every $T\in (0,T_0)$, it follows that
	\[\sup_{t\in[0, T]}\left\|\left[\Phi(\tilde{v}_1,\tilde{w}_1)\right](t)-\begin{pmatrix}\tilde{v}_0\\ \tilde{w}_0\end{pmatrix}\right\|_{L^2(\Omega)\times H^{2}(\Omega)}\leq r,\]
	\[\sup_{t\in[0, T]}\|[\Phi(\tilde{v}_1,\tilde{w}_1)](t)-[\Phi(\tilde{v}_2,\tilde{w}_2)](t)\|_{L^2(\Omega)\times H^{2}(\Omega)}
	\leq\frac{1}{2}\sup_{t\in[0, T]}\left\|\begin{pmatrix}\tilde{v}_1(t)-\tilde{v}_2(t)\\ \tilde{w}_1(t)-\tilde{w}_2(t)\end{pmatrix}\right\|_{L^{2}(\Omega)\times H^{2}(\Omega)}.\]
	Hereby $\Phi(\tilde{v}_1,\tilde{w}_1)\in \mathcal{Z}(T)$ for $(\tilde{v}_1,\tilde{w}_1)\in\mathcal{Z}(T)$,  $\Phi(\tilde{v}, \tilde{w})$ is Lipschitz continuous on the bounded set $\mathcal{Z}(T)$ with Lipschitz constant smaller than or equal to $\frac{1}{2}$, and $\Phi(\tilde{v}, \tilde{w})$ is a contractive mapping of $\mathcal{Z}(T)$ into itself.
	
	According to the Banach fixed point theorem, for each $T\in(0, T_0)$, there exists a unique fixed point $(\tilde{v}_T, \tilde{w}_T)\in \mathcal{Z}(T)$, such that $(\tilde{v}_T,\tilde{w}_T)=\Phi(\tilde{v}_T,\tilde{w}_T)$ for given $\tilde{u}\in C\left([0, T]; B_{H^2}\left(\tilde{u}_0, r\right)\right)$.
	
	Hence, $(\tilde{v}_T, \tilde{w}_T)\in \mathcal{Z}(T)$ is the unique mild solution of the semilinear evolution equation \eqref{4th-SWE-1} on $[0, T]$, and $(\tilde{v}_T,\tilde{w}_T)$ satisfies the integral formulation \eqref{mild-solu-form}. Due to the uniqueness of the fixed point, we set $(\tilde{v}, \tilde{w})=(\tilde{v}_{T},\tilde{w}_{T})$ and note that $(\tilde{v}_T, \tilde{w}_T)$ is the restriction $(\tilde{v}|_{[0, T]}, \tilde{w}|_{[0, T]})\in\mathcal{Z}(T)$ of $(\tilde{v}, \tilde{w})$. As a result, the assertion is proved.
\end{proof}

\subsection{Proof of Corollary \ref{Lip-mild-solu}}
\begin{proof}
	Take $0\leq t<t+h\leq T$. Equation \eqref{mild-solu-form} leads to
	\begin{align}
		\begin{pmatrix}\tilde{v}(t+h)-\tilde{v}(t)\\ \tilde{w}(t+h)-\tilde{w}(t)\end{pmatrix}=&T(t)\left[T(h)\begin{pmatrix}\tilde{v}_0\\ \tilde{w}_0\end{pmatrix}-\begin{pmatrix}\tilde{v}_0\\ \tilde{w}_0\end{pmatrix}\right]+\displaystyle\int_0^hT(t+h-s)\begin{pmatrix}[G(\tilde{w})](s)+\beta_p\tilde{u}(s)\\ 0\end{pmatrix}ds\notag\\
		+&\displaystyle\int_0^tT(t-s)\begin{pmatrix}[G(\tilde{w})](s+h)-[G(\tilde{w})](s)+\beta_p[\tilde{u}(s+h)-\tilde{u}(s)]\\ 0\end{pmatrix}ds\notag\\
		=&\displaystyle \int_0^hT(t+s)\mathcal{A}\begin{pmatrix}\tilde{v}_0\\ \tilde{w}_0\end{pmatrix}ds+\displaystyle\int_0^hT(t+h-s)\begin{pmatrix}[G(\tilde{w})](s)+\beta_p\tilde{u}(s)\\ 0\end{pmatrix}ds\notag\\
		+&\displaystyle\int_0^tT(t-s)\begin{pmatrix}[G(\tilde{w})](s+h)-[G(\tilde{w})](s)+\beta_p[\tilde{u}(s+h)-\tilde{u}(s)]\\ 0\end{pmatrix}ds.\label{diff-mild-solution}
	\end{align}
	Notice that
	\be\label{Lip-est-1}
	M_0=\sup_{t\in[0, \infty)}\left\|T(t)\right\|_{\mathcal{B}\left(L^2(\Omega)\times H^2(\Omega)\right)}, \quad \left\|\mathcal{A}\begin{pmatrix}\tilde{v}_0\\ \tilde{w}_0\end{pmatrix}\right\|_{L^2(\Omega)\times H^2(\Omega)}\leq \left\|\left(\tilde{v}_0, \tilde{w}_0\right)\right\|_{D(\mathcal{A})}.
	\ee
	For $T\in (0, T_0)$ and $\tilde{u}\in C\left([0, T]; B_{H^2}\left(\tilde{u}_0, r\right)\right)$, the semilinear evolution equation \eqref{4th-SWE-1} has a unique mild solution $\left(\tilde{v}, \tilde{w}\right)\in \mathcal{Z}(T)$, by using the estimate \eqref{Lip-G-1} in Lemma \ref{Lip-G-Lem}, we have
	\begin{align}
		\sup_{t\in[0, T]}\left\|\begin{pmatrix}[G(\tilde{w})](t)+\beta_p\tilde{u}(t)\\ 0\end{pmatrix}\right\|_{L^2(\Omega)\times H^2(\Omega)}
		&\leq \sup_{t\in[0, T]}\left\|[G(\tilde{w})](t)+\beta_p\tilde{u}(t)\right\|_{H^2(\Omega)}\notag\\
		&\leq \sup_{t\in[0, T]}\left\|[G(\tilde{w})](t)-G(\tilde{w}_0)\right\|_{H^2(\Omega)}\notag\\
		&+\beta_p\sup_{t\in[0, T]}\left\|\tilde{u}(t)-\tilde{u}_0\right\|_{H^2(\Omega)}+\left\|G_0\right\|_{H^{2}(\Omega)}\notag\\
		&\leq \left(L_G+1\right) r+\left\|G_0\right\|_{H^{2}(\Omega)}\notag\\
		&\leq \frac{\kappa \left(L_G+1\right)}{2C}+\left\|G_0\right\|_{H^{2}(\Omega)}\label{Lip-est-2}.
	\end{align}
	Here $L_G$ is given by Lemma \ref{Lip-G-Lem} and $G_0=G(\tilde{w}_0)+\beta_p\tilde{u}_0\in H^2(\Omega)$.
	
	Moreover, $\tilde{u}(s+h)$, $\tilde{u}(s)$, $\tilde{w}(s+h)$, $\tilde{w}(s)\in H_*^2(\Omega)$, $\forall\ 0\leq s< s+h \leq t\leq T$, therefore,
	$[G(\tilde{w})](s+h)+\beta_p\tilde{u}(s+h)-[G(\tilde{w})](s)-\beta_p\tilde{u}(s)\in H^2(\Omega).$
	
	Since $\tilde{u}\in C^1\left([0, T_0); L^2(\Omega)\right)$, then, $\forall\ T\in(0, T_0)$,
	\begin{align}
		\sup_{t\in[0, T]}\displaystyle\int_0^t\left\|\tilde{u}(s+h)-\tilde{u}(s)\right\|_{L^2(\Omega)}ds\leq& T_0\sup_{\begin{smallmatrix}0\leq s\leq T,\\ 0\leq s+\sigma h\leq T\end{smallmatrix}}h\int_0^1\left\|\tilde{u}'(s+\sigma h)\right\|_{L^2(\Omega)}d\sigma\notag\\
		\leq&T_0h\left\|\tilde{u}\right\|_{C^1\left([0, T_0); L^2(\Omega)\right)}.\notag
	\end{align}
	According to  inequality \eqref{Holdercontinuous} of Lemma \ref{Lip-G-Lem}, we have
	\begin{align}
		&\left\|\displaystyle\int_0^tT(t-s)\begin{pmatrix}[G(\tilde{w})](s+h)-[G(\tilde{w})](s)+\beta_p[\tilde{u}(s+h)-\tilde{u}(s)]\\ 0\end{pmatrix}ds\right\|_{L^2(\Omega)\times H^2(\Omega)}\notag\\
		=& M_0\displaystyle\int_0^t\left\|[G(\tilde{w})](s+h)-[G(\tilde{w})](s)+\beta_p[\tilde{u}(s+h)-\tilde{u}(s)]\right\|_{L^2(\Omega)}ds\notag\\
		\leq& M_0\displaystyle\int_0^t\beta_p\left\|\tilde{u}(s+h)-\tilde{u}(s)\right\|_{L^2(\Omega)}
		+\left\|[G(\tilde{w})](s+h)-[G(\tilde{w})](s)\right\|_{H^2(\Omega)}ds\notag\\
		\leq& hM_0\beta_pT_0\left\|\tilde{u}\right\|_{C^1\left([0, T_0); L^2(\Omega)\right)}+M_0L_G\displaystyle\int_0^t\left\|\begin{pmatrix}\tilde{v}(s+h)-\tilde{v}(s)\\  \tilde{w}(s+h)-\tilde{w}(s)\end{pmatrix}\right\|_{L^2(\Omega)\times H^2(\Omega)}ds\label{diff-mild-solution-1}.
	\end{align}
	Combing \eqref{diff-mild-solution}, \eqref{Lip-est-1}, \eqref{Lip-est-2} and \eqref{diff-mild-solution-1} gives
	\begin{align}
		\left\|\begin{pmatrix}\tilde{v}(t+h)-\tilde{v}(t)\\ \tilde{w}(t+h)-\tilde{w}(t)\end{pmatrix}\right\|_{L^2(\Omega)\times H^2(\Omega)}&\leq hM_0\left\|\left(\tilde{v}_0, \tilde{w}_0\right)\right\|_{D(\mathcal{A})}+hM_0\left(\frac{\kappa\left(L_G+1\right) }{2C}+\left\|G_0\right\|_{H^{2}(\Omega)}\right)\notag\\
		&+hM_0\beta_pT_0\left\|\tilde{u}\right\|_{C^1\left([0, T_0); L^2(\Omega)\right)}\notag\\
		&+M_0L_G\displaystyle\int_0^t\left\|\begin{pmatrix}\tilde{v}(s+h)-\tilde{v}(s)\\  \tilde{w}(s+h)-\tilde{w}(s)\end{pmatrix}\right\|_{L^2(\Omega)\times H^2(\Omega)}ds\notag.
	\end{align}
	Set $V_o=\left\|\left(\tilde{v}_0, \tilde{w}_0\right)\right\|_{D(\mathcal{A})}+\left(\frac{\kappa \left(L_G+1\right)}{2C}+\left\|G_0\right\|_{H^{2}(\Omega)}\right)+\beta_pT_0\left\|\tilde{u}\right\|_{C^1\left([0, T_0); L^2(\Omega)\right)}$.
	
	Gronwall's inequality then implies that
	\[\left\|\begin{pmatrix}\tilde{v}(t+h)-\tilde{v}(t)\\ \tilde{w}(t+h)-\tilde{w}(t)\end{pmatrix}\right\|_{L^2(\Omega)\times H^2(\Omega)}\leq M_0V_o\left(e^{M_0L_GT_0}\right)h.\]
	Therefore, \eqref{Lip-mild-solu-inquality} holds for all $h\in(0, T]$ by setting $L_V=M_0V_o\left(e^{M_0L_GT_0}\right).$
\end{proof}

\subsection{Proof of Corollary \ref{Holdercontinuity}}
\begin{proof}
	{Note that}
	\begin{align}
		\sup_{t\in[0, T]}\displaystyle\int_0^t\left\|\tilde{u}(s+h)-\tilde{u}(s)\right\|_{L^2(\Omega)}ds\leq& T_0\sup_{0\leq s<s+h\leq T}\left\|\tilde{u}(s+h)-\tilde{u}(s)\right\|_{L^2(\Omega)}\notag\\
		\leq&T_0\sup_{0\leq s<s+h\leq T}\left\|\tilde{u}(s+h)-\tilde{u}(s)\right\|_{H^2(\Omega)}\notag\\
		\leq&T_0h^\alpha\left[\tilde{u}\right]_{C^\alpha\left([0, T_0); H^2(\Omega)\right)}.\label{Lip-est-3}
	\end{align}
	According to \eqref{diff-mild-solution}, \eqref{Lip-est-1}, \eqref{Lip-est-2}, \eqref{diff-mild-solution-1} and \eqref{Lip-est-3}, {and for}  $0\leq t<t+h\leq T$, $h\in(0, T]$,
	\begin{align}
		\left\|\begin{pmatrix}\tilde{v}(t+h)-\tilde{v}(t)\\ \tilde{w}(t+h)-\tilde{w}(t)\end{pmatrix}\right\|_{L^2(\Omega)\times H^2(\Omega)}&\leq hM_0\left\|\left(\tilde{v}_0, \tilde{w}_0\right)\right\|_{D(\mathcal{A})}+h^\alpha M_0\beta_pT_0\left[\tilde{u}\right]_{C^\alpha\left([0, T_0); H^2(\Omega)\right)}\notag\\
		&+hM_0\left(\frac{\kappa\left(L_G+1\right)}{2C}+\left\|G(\tilde{w}_0)+\beta_p\tilde{u}_0\right\|_{H^{2}(\Omega)}\right)\notag\\
		&+M_0L_G\displaystyle\int_0^t\left\|\begin{pmatrix}\tilde{v}(s+h)-\tilde{v}(s)\\  \tilde{w}(s+h)-\tilde{w}(s)\end{pmatrix}\right\|_{L^2(\Omega)\times H^2(\Omega)}ds\notag.
	\end{align}
	Set $P_0=\frac{\kappa\left(L_G+1\right)}{2C}+\left\|G(\tilde{w}_0)+\beta_p\tilde{u}_0\right\|_{H^{2}(\Omega)}$, and Gronwall's inequality then implies that
	\begin{align}
		\left\|\begin{pmatrix}\tilde{v}(t+h)-\tilde{v}(t)\\ \tilde{w}(t+h)-\tilde{w}(t)\end{pmatrix}\right\|_{L^2(\Omega)\times H^2(\Omega)}
		\leq& \left(e^{M_0L_GT_0}\right)M_0\left\{\left\|\left(\tilde{v}_0, \tilde{w}_0\right)\right\|_{D(\mathcal{A})}+P_0\right\}h\notag\\
		+&\left(e^{M_0L_GT_0}\right)M_0\beta_pT_0\left[\tilde{u}\right]_{C^\alpha\left([0, T_0); H^2(\Omega)\right)}h^\alpha \notag\\
		:=&I_1h+I_2h^\alpha\notag.
	\end{align}
	{Note that}
	\[I_1h\leq \left(e^{M_0L_GT_0}\right)M_0\left\{\left\|\left(\tilde{v}_0, \tilde{w}_0\right)\right\|_{D(\mathcal{A})}+P_0\right\}T_0^{1-\alpha}h^\alpha\notag,\]
{and}
	\begin{align}
		I_2\leq& \left(e^{M_0L_GT_0}\right)M_0\beta_pT_0\left\|\tilde{u}\right\|_{C^\alpha\left([0, T_0); H^2(\Omega)\right)}\notag\\
		\leq&\left(e^{M_0L_GT_0}\right)M_0\beta_pT_0\left(\left\|\tilde{u}-\tilde{u}_0\right\|_{C^\alpha\left([0, T_0); H^2(\Omega)\right)}+\left\|\tilde{u}_0\right\|_{H^2(\Omega)}\right)\notag\\
		\leq& \left(e^{M_0L_GT_0}\right)M_0\beta_pT_0\left(r+\left\|\tilde{u}_0\right\|_{H^2(\Omega)}\right)\notag\\
		\leq&\left(e^{M_0L_GT_0}\right)M_0\beta_pT_0\left(\frac{\kappa}{2C}+\left\|\tilde{u}_0\right\|_{H^2(\Omega)}\right)\notag
	\end{align}
	Set $L_U=\left(e^{M_0L_GT_0}\right)M_0\left\{\left\|\left(\tilde{v}_0, \tilde{w}_0\right)\right\|_{D(\mathcal{A})}+P_0\right\}T_0^{1-\alpha}+\left(e^{M_0L_GT_0}\right)M_0\beta_pT_0\left(\frac{\kappa}{2C}+\left\|\tilde{u}_0\right\|_{H^2(\Omega)}\right)$
	and $L_U$ is a Lipschitz constant depending on $\alpha$,  $T_0$, $\kappa$, $\Omega$, $\beta_p$, $\beta_F$, $\left\|\tilde{u}_0\right\|_{H^2(\Omega)}$, $\left\|\tilde{w}_0\right\|_{H^{2}(\Omega)}$, $\left\|\left(\tilde{v}_0, \tilde{w}_0\right)\right\|_{D(\mathcal{A})}$ and  $M_0=\displaystyle\sup_{t\in[0, \infty)}\left\|T(t)\right\|_{\mathcal{B}\left(L^2(\Omega)\times H^2(\Omega)\right)}$.
	
	Therefore, \eqref{Holdercontinuityformular} holds for all $h\in(0, T]$.
\end{proof}

\subsection{Proof of Theorem \ref{4th-mild-solution-cor}}

\begin{proof}
	Let $T\in(0, T_0)$, $G_0=G\left(\tilde{w}_0\right)+\beta_p\tilde{u}_0$ and $\left(\tilde{v}, \tilde{w}\right)$ be the mild solution of the semilinear evolution equation \eqref{4th-SWE-1} defined by \eqref{mild-solu-form}. Take $\tilde{u}\in C\left([0, T]; B_{H^2}\left(\tilde{u}_0,  r\right)\right)\cap C^1([0, T]; L^2(\Omega))$ to be given such that  $\tilde{u}'(t)\in L^2(\Omega)$ is uniformly continuous for all $t\in [0, T]$.
	
	We first prove the linear non-autonomous problem
	\be\label{LNP}
	\begin{pmatrix}\tilde{p}(t)\\ \tilde{q}(t)\end{pmatrix}=T(t)\left(\begin{pmatrix}G_0\\ 0\end{pmatrix}+\mathcal{A}\begin{pmatrix}\tilde{v}_0\\ \tilde{w}_0\end{pmatrix}\right)+\displaystyle\int_0^tT(t-s)\begin{pmatrix}[\mathcal{H}(\tilde{q})](s)+\beta_p\tilde{u}'(s)\\ 0\end{pmatrix}ds,
	\ee
	can be solved for $t\in[0, T]$. Here
	\be\label{H-nonlinearity}
	[\mathcal{H}(\tilde{q})](s)=\frac{2\beta_F\tilde{q}(s)}{\left[\tilde{w}(s)+\theta_2\right]^{3}}=[G'(\tilde{w}(s))]q(s)=[G'(\tilde{w})q](s),\quad s\in [0, t].
	\ee
	We define a nonlinear operator $\Psi$ by
	\[\Psi:\ C\left([0, T]; L^2(\Omega)\times H_*^2(\Omega)\right)\longrightarrow C\left([0, T]; L^2(\Omega)\times H_*^2(\Omega)\right),\]
	\[\left[\Psi\left(\tilde{p}, \tilde{q}\right)\right](t)=T(t)\left(\begin{pmatrix}G_0\\ 0\end{pmatrix}+\mathcal{A}\begin{pmatrix}\tilde{v}_0\\ \tilde{w}_0\end{pmatrix}\right)+\displaystyle\int_0^tT(t-s)\begin{pmatrix}[\mathcal{H}(\tilde{q})](s)+\beta_p\tilde{u}'(s)\\ 0\end{pmatrix}ds.\]
	For any $\left(\tilde{p}_1, \tilde{q}_1\right),\ \left(\tilde{p}_2, \tilde{q}_2\right)\in C\left([0, T]; L^2(\Omega)\times H_*^2(\Omega)\right),\ \tilde{w}\in C\left([0, T]; H_*^2(\Omega)\right),$
	then
	\begin{align}
		\mathcal{H}(\tilde{q}_1)-\mathcal{H}(\tilde{q}_2)=\frac{2\beta_F}{\left[\tilde{w}+\theta_2\right]^{3}}\left[\tilde{q}_1-\tilde{q}_2\right]=G'(\tilde{w})(q_1-q_2)\in C\left([0, T]; H_*^2(\Omega)\right)\notag.
	\end{align}
	Hence, according to the estimate \eqref{Lip-G-2} of Fr\'{e}chet derivative $G'\left(\tilde{w}\right)q$ from Lemma \ref{Lip-G-Lem} and the definition of $T_0$ in Theorem \ref{4th-solu-thm}, $\Psi$ is a contractive mapping on $C\left([0, T]; L^2(\Omega)\times H_*^2(\Omega)\right)$ because
	\begin{align}
		&\displaystyle\sup_{t\in[0, T]}\left\|\left[\Psi\left(\tilde{p}_1, \tilde{q}_1\right)\right](t)-\left[\Psi\left(\tilde{p}_2, \tilde{q}_2\right)\right](t)\right\|_{L^2(\Omega)\times H^2(\Omega)}\notag\\
		=&\displaystyle\sup_{t\in[0, T]}\left\|\displaystyle\int_0^tT(t-s)\begin{pmatrix}[\mathcal{H}(\tilde{q}_1)](s)-[\mathcal{H}(\tilde{q}_2)](s)\\ 0\end{pmatrix}ds
		\right\|_{L^2(\Omega)\times H^2(\Omega)}\notag\\
		\leq& T\sup_{0\leq s\leq t\leq T}\|T(t-s)\|_{\mathcal{B}\left(L^2(\Omega)\times H^2(\Omega)\right)}\sup_{t\in[0, T]}\left\|[\mathcal{H}(\tilde{q}_1)](t)-[\mathcal{H}(\tilde{q}_2)](t)\right\|_{L^2(\Omega)}\notag\\
		\leq&TM_0L_G\displaystyle\sup_{t\in[0, T]}\left\|\tilde{q}_1(t)-\tilde{q}_2(t)\right\|_{H^2(\Omega)}
		\leq\frac{1}{2}\sup_{t\in[0, T]}\left\|\begin{pmatrix}\tilde{p}_1(t)-\tilde{p}_2(t)\\ \tilde{q}_1(t)-\tilde{q}_2(t)\end{pmatrix}\right\|_{L^2(\Omega)\times H^2(\Omega)}\notag.
	\end{align}
	According to the Banach fixed point theorem, for any $\tilde{u}\in C\left([0, T]; B_{H^2}\left(\tilde{u}_0, r\right)\right)\cap C^1\left([0, T]; L^2(\Omega)\right),$ there exists a unique fixed point $(\tilde{p}, \tilde{q})\in C\left([0, T]; L^2(\Omega)\times H_*^2(\Omega)\right)$, such that $(\tilde{p}, \tilde{q})=\Psi(\tilde{p}, \tilde{q})$. Hereby, the $\mathds{R}$-linear non-autonomous problem \eqref{LNP} can be solved for $t\in[0, T]$.
	
	We next prove that $(\tilde{p}, \tilde{q})$ is the time derivative of the mild solution $(\tilde{v}, \tilde{w})$.
	
	Let $0\leq t<t+h\leq T$ for some $h\in(0, T]$, equations \eqref{mild-solu-form} and \eqref{LNP} imply that
	\begin{align}
		E(h,t)&:=\frac{1}{h}\begin{pmatrix}\tilde{v}(t+h)-\tilde{v}(t)\\ \tilde{w}(t+h)-\tilde{w}(t)\end{pmatrix}-\begin{pmatrix}\tilde{p}(t)\\ \tilde{q}(t)\end{pmatrix}\notag\\
		&=T(t)\frac{1}{h}\left(T(h)-I\right)\begin{pmatrix}\tilde{v}_0\\ \tilde{w}_0\end{pmatrix}-T(t)\mathcal{A}\begin{pmatrix}\tilde{v}_0\\ \tilde{w}_0\end{pmatrix}\notag\\
		&+\frac{1}{h}\displaystyle\int_0^h\left\{T(t+h-s)\begin{pmatrix}[G(\tilde{w})](s)+\beta_p\tilde{u}(s)\\ 0\end{pmatrix}\right\}ds-T(t)\begin{pmatrix}G_0\\ 0\end{pmatrix}\notag\\
		&+\displaystyle\int_0^tT(t-s)\begin{pmatrix}\frac{1}{h}\left\{[G(\tilde{w})](s+h)-[G(\tilde{w})](s)\right\}-[\mathcal{H}(\tilde{q})](s)\\ 0\end{pmatrix}ds\notag\\
		&+\int_0^tT(t-s)\begin{pmatrix}\beta_p\left[\frac{1}{h}\left\{\tilde{u}(s+h)-\tilde{u}(s)\right\}-\tilde{u}'(s)\right]\\ 0\end{pmatrix}ds\notag.
	\end{align}
	Let
	\[E^{(1)}(h,t):=T(t)\frac{1}{h}\left(T(h)-I\right)\begin{pmatrix}\tilde{v}_0\\ \tilde{w}_0\end{pmatrix}-T(t)\mathcal{A}\begin{pmatrix}\tilde{v}_0\\ \tilde{w}_0\end{pmatrix},\]
	\[E^{(2)}(h,t):=\frac{1}{h}\int_0^h\left\{T(t+h-s)\begin{pmatrix}[G(\tilde{w})](s)+\beta_p\tilde{u}(s)\\ 0\end{pmatrix}\right\}ds-T(t)\begin{pmatrix}G_0\\ 0\end{pmatrix},\]
	\begin{align}
		E^{(3)}(h,t)&:=\displaystyle\int_0^tT(t-s)\begin{pmatrix}\frac{1}{h}\left\{[G(\tilde{w})](s+h)-[G(\tilde{w})](s)\right\}-[\mathcal{H}(\tilde{q})](s)\\ 0\end{pmatrix}ds\notag\\
		&+\int_0^tT(t-s)\begin{pmatrix}\beta_p\left[\frac{1}{h}\left\{\tilde{u}(s+h)-\tilde{u}(s)\right\}-\tilde{u}'(s)\right]\\ 0\end{pmatrix}ds\notag.
	\end{align}
	We initially notice that
	\begin{align}
		\lim_{h\rightarrow 0^+}\left\|E^{(1)}(h,t)\right\|_{L^2(\Omega)\times H^2(\Omega)}&\leq \lim_{h\rightarrow 0^+} M_0\left\|\frac{1}{h}\left(T(h)-I\right)\begin{pmatrix}\tilde{v}_0\\ \tilde{w}_0\end{pmatrix}-\mathcal{A}\begin{pmatrix}\tilde{v}_0\\ \tilde{w}_0\end{pmatrix}\right\|_{L^2(\Omega)\times H^2(\Omega)}\notag\\
		&:=\lim_{h\rightarrow 0^+}\Lambda_1(h)=0,\notag
	\end{align}
{and}
	\[\lim_{h\rightarrow 0^+}\frac{1}{h}\displaystyle\int_0^h\left\{T(h-s)\begin{pmatrix}G_0 \\ 0\end{pmatrix}\right\}ds=\begin{pmatrix}G_0 \\ 0\end{pmatrix}.\]
	Because $ G(\tilde{w})\in C([0, T]; H^2(\Omega))$, $\tilde{u}\in C\left([0, T]; H_*^2(\Omega)\right)$, then
	\[\lim_{h\rightarrow 0^+}\sup_{0\leq s\leq h}\left\|[G(\tilde{w})](s)-[G(\tilde{w})](0)+\beta_p(\tilde{u}(s)-\tilde{u}(0))\right\|_{H^2(\Omega)}=0,\]
	hence
	\begin{align}
		&\lim_{h\rightarrow 0^+}\left\|E^{(2)}(h,t)\right\|_{L^2(\Omega)\times H^2(\Omega)}\notag\\
		=&\lim_{h\rightarrow 0^+}\left\|\frac{1}{h}\displaystyle\int_0^h\left\{T(t+h-s)\begin{pmatrix}[G(\tilde{w})](s)+\beta_p\tilde{u}(s)\\ 0\end{pmatrix}\right\}ds-T(t)\begin{pmatrix}G_0\\ 0\end{pmatrix}\right\|_{L^2(\Omega)\times H^2(\Omega)}\notag\\
		=&\lim_{h\rightarrow 0^+}\left\|T(t)\frac{1}{h}\displaystyle\int_0^h\left\{T(h-s)\begin{pmatrix}[G(\tilde{w})](s)-G(\tilde{w}_0)+\beta_p(\tilde{u}(s)-\tilde{u}_0) \\ 0\end{pmatrix}\right\}ds\right\|_{L^2(\Omega)\times H^2(\Omega)}\notag\\
		\leq&\lim_{h\rightarrow 0^+}M_0\left\|\frac{1}{h}\displaystyle\int_0^h\left\{T(h-s)\begin{pmatrix}[G(\tilde{w})](s)-G(\tilde{w}_0)+\beta_p(\tilde{u}(s)-\tilde{u}_0)\\ 0\end{pmatrix}\right\}ds\right\|_{L^2(\Omega)\times H^2(\Omega)}\notag\\
		\leq&\lim_{h\rightarrow 0^+}M_0^2\frac{h}{h}\sup_{0\leq s\leq h}\left\|[G(\tilde{w})](s)-G(\tilde{w}_0)+\beta_p(\tilde{u}(s)-\tilde{u}_0)\right\|_{L^2(\Omega)}\notag\\
		=&\lim_{h\rightarrow 0^+}M_0^2\sup_{0\leq s\leq h}\left\|[G(\tilde{w})](s)-[G(\tilde{w})](0)+\beta_p(\tilde{u}(s)-\tilde{u}(0))\right\|_{H^2(\Omega)}\notag\\
		:=&\lim_{h\rightarrow 0^+}\Lambda_2(h)=0.\notag
	\end{align}
	Define
	\[G_D(t,h):=\left[G(\tilde{w})\right](t+h)-\left[G(\tilde{w})\right](t)-\left[G'(\tilde{w})\right](t)\cdot\left[\tilde{w}(t+h)-\tilde{w}(t)\right],\]
	\[E^{(3)}_1(h,t):=\displaystyle\int_0^tT(t-s)\begin{pmatrix}\frac{1}{h}G_D(s,h)\\ 0\end{pmatrix}ds,\]
	\[E^{(3)}_2(h,t):=\displaystyle\int_0^tT(t-s)\begin{pmatrix}[G'(\tilde{w})](s)\left\{\frac{1}{h}\{\tilde{w}(s+h)-\tilde{w}(s)\}-\tilde{q}(s)\right\}\\0\end{pmatrix}ds,\]
	\[E^{(3)}_3(h,t):=\displaystyle\int_0^tT(t-s)\begin{pmatrix}\beta_p\left\{\frac{1}{h}[\tilde{u}(s+h)-\tilde{u}(s)]-\tilde{u}'(s)\right\}\\ 0\end{pmatrix}ds.\]
	We then write
	\[E^{(3)}(h,t)=E^{(3)}_1(h,t)+E^{(3)}_2(h,t)+E^{(3)}_3(h,t).\]
	{Hence}
	\begin{align}
		\left\|E^{(3)}_3(h,t)\right\|_{L^2(\Omega)\times H^2(\Omega)}=&\left\|\displaystyle\int_0^tT(t-s)\begin{pmatrix}\beta_p\left\{\frac{1}{h}[\tilde{u}(s+h)-\tilde{u}(s)]-\tilde{u}'(s)\right\}\\ 0\end{pmatrix}ds\right\|_{L^2(\Omega)\times H^2(\Omega)}\notag\\
		\leq& T_0M_0\beta_p\sup_{\begin{smallmatrix}0\leq t<t+h\leq T\end{smallmatrix}}\left\|\frac{\tilde{u}(t+h)-\tilde{u}(t)}{h}-\tilde{u}'(t)\right\|_{L^2(\Omega)}\notag\\
		=&T_0M_0\beta_p\sup_{\begin{smallmatrix}0\leq t\leq t+\sigma h\leq T\\ 0\leq\sigma\leq 1\end{smallmatrix}}\left\|\frac{1}{h}\int_0^1\frac{d}{d\sigma}\left[\tilde{u}(t+\sigma h)\right]d\sigma-\tilde{u}'(t)\right\|_{L^2(\Omega)}\notag\\
		=&T_0M_0\beta_p\sup_{\begin{smallmatrix}0\leq t\leq t+\sigma h\leq T\\ 0\leq\sigma\leq 1\end{smallmatrix}}\left\|\tilde{u}'(t+\sigma h)-\tilde{u}'(t)\right\|_{L^2(\Omega)}\notag\\
		:=&\Lambda_3(h)\rightarrow 0,\ \text{as}\ h\rightarrow 0^+, \notag
	\end{align}
	since $\tilde{u}\in C\left([0, T]; B_{H^2}\left(\tilde{u}_0, r\right)\right)\cap C^1\left([0, T]; L^2(\Omega)\right)$ is given such that the time derivative $\tilde{u}'(t)\in L^2(\Omega)$ is uniformly continuous for all $t\in [0, T]$. Using  the bound estimate \eqref{Lip-G-2} of Fr\'{e}chet derivative $G'\left(\tilde{w}\right)$ from Lemma \ref{Lip-G-Lem} again gives
	\begin{align}
		\left\|E^{(3)}_2(h,t)\right\|_{L^2(\Omega)\times H^2(\Omega)}\leq M_0L_G\displaystyle\int_0^t\left\|E(h,s)\right\|_{L^2(\Omega)\times H^2(\Omega)}ds\notag.
	\end{align}
	Because $G_D(t,h)\in H^2(\Omega)$. The estimate \eqref{Lip-mild-solu-inquality} of Corollary \ref{Lip-mild-solu} implies the function $\tilde{w}$ is Lipschitz continuous with respect to $t\in[0, T]$.  Employing this fact and the limit \eqref{uniformly-continuous-Fre-G} of Lemma \ref{Lip-G-Lem} gives
	\begin{align}
		&\left\|E^{(3)}_1(h,t)\right\|_{L^2(\Omega)\times H^2(\Omega)}\notag\\
		=&\frac{T_0M_0}{h}\sup_{\begin{smallmatrix}t\in[0, T]\\ t+h\in[0, T]\end{smallmatrix}}\left\|\displaystyle\int_0^1\left[G'(\tilde{w}(t)+\tau[\tilde{w}(t+h)-\tilde{w}(t)])-G'(\tilde{w}(t))\right]\left[\tilde{w}(t+h)-\tilde{w}(t)\right]d\tau\right\|_{ H^2(\Omega)}\notag\\
		\leq& \frac{T_0M_0L_V h}{h}\sup_{0\leq t< t+ h\leq T}\left\|\displaystyle\int_0^1G'(\tilde{w}(t)+\tau[\tilde{w}(t+h)-\tilde{w}(t)])-G'(\tilde{w}(t))d\tau\right\|_{\mathcal{B}\left(H^2(\Omega)\right)}\notag\\
		=&T_0M_0L_V\sup_{\begin{smallmatrix}0\leq t\leq t+\tau h\leq T\\ 0\leq\tau\leq1\end{smallmatrix}}\left\|G'(\tilde{w}(t)+\tau[\tilde{w}(t+h)-\tilde{w}(t)])-G'(\tilde{w}(t))\right\|_{\mathcal{B}\left(H^2(\Omega)\right)}\notag\\
		:=&\Lambda_4(h)\rightarrow 0 ,\ \text{as}\ h\rightarrow 0^+. \notag
	\end{align}
	Summing up, we have shown
	\[\left\|E(h,t)\right\|_{L^2(\Omega)\times H^2(\Omega)}\leq\Lambda_1(h)+\Lambda_2(h)+\Lambda_3(h)+\Lambda_4(h)+M_0L_G\displaystyle\int_0^t\left\|E(h,s)\right\|_{L^2(\Omega)\times H^2(\Omega)}ds.\]
	Gronwall's inequality thus implies the inequality
	\[\left\|E(h,t)\right\|_{L^2(\Omega)\times H^2(\Omega)}\leq\left(\Lambda_1(h)+\Lambda_2(h)+\Lambda_3(h)+\Lambda_4(h)\right)e^{tM_0L_G}{,} \]
	{which} holds for $t\in [0, T]$. Letting $h\rightarrow 0^+$, we then deduce that the $\left(\tilde{v}, \tilde{w}\right)$ is differentiable from the right and the right  derivative of $\left(\tilde{v}, \tilde{w}\right)$ coincides with $\left(\tilde{p}, \tilde{q}\right)$. Because $\left(\tilde{p}, \tilde{q}\right)$ is continuous on $[0, T]$, by using Lemma \ref{time-derivative-continuity}, we conclude $\left(\tilde{v}, \tilde{w}\right)\in C^1\left([0, T]; L^2(\Omega)\times H_*^2(\Omega)\right)$. As $\tilde{u}\in C\left([0, T]; B_{H^2}\left(\tilde{u}_0, r\right)\right)\cap C^1\left([0, T]; L^2(\Omega)\right)$, then $\left(G\left(\tilde{w}\right)+\beta_p\tilde{u}, 0\right)\in C^1\left([0, T]; L^2(\Omega)\times H_*^2(\Omega)\right)$. By Lemma \ref{IEE-S}, the mild solution $\left(\tilde{v}, \tilde{w}\right)$ , defined by \eqref{mild-solu-form}, uniquely solves the semilinear evolution equation \eqref{4th-SWE-1} on $[0, T]$, $\left(\tilde{v}, \tilde{w}\right)$ is a unique strict solution of semilinear evolution equation \eqref{4th-SWE-1}, and
	\[\left(\tilde{v}, \tilde{w}\right)\in C\left([0, T]; H_*^2(\Omega)\times H_*^4(\Omega)\right)\cap C^1\left([0, T]; L^2(\Omega)\times H_*^2(\Omega)\right), \forall \ T\in(0, T_0).\]
\end{proof}

\section{Proof of Lemma \ref{RHSMax}}\label{AppA2}
\begin{proof}
	Let $T\in(0, T_0)$. Recall that $C=C(\Omega)$ is a constant which may vary from line to line but depends on $\Omega$ only. From the discussion of the graph norm of the linear operator $\mathcal{P}^*$ and $\tilde{u}\in C^\alpha([0, T]; B_{H^2}(\tilde{u}_0, r))$, then $\tilde{u}\in C\left([0, T]; B_{H^2}\left(\tilde{u}_0, r\right)\right),\ \tilde{u}(0)=\tilde{u}_0\in{D(\mathcal{P}^*)}.$
	
	According to Theorem \ref{4th-solu-thm}, $(\tilde{v}, \tilde{w})\in \mathcal{Z}(T)$ is a unique mild solution of the semilinear evolution equation \eqref{4th-SWE-1} for all $r\in\left(0, \frac{\kappa}{2C}\right)$. Here $\kappa=\displaystyle\inf_{x\in{\Omega}}\tilde{w}_0(x)+\theta_2$. Thus from Corollary \ref{Holdercontinuity}, if $\tilde{u}\in C^\alpha\left([0, T]; B_{H^2}(\tilde{u}_0,r)\right)$.
	\[(\tilde{v}, \tilde{w})=(\tilde{v}(\tilde{u}), \tilde{w}(\tilde{u}))\in \mathcal{Z}(T)\cap C^\alpha\left(\left[0, T\right]; L^2(\Omega)\times H_*^2(\Omega)\right).\]
	Recall
	$u_0=\tilde{u}_0+\theta_1$, $v_0=\tilde{v}_0$, $w_0=\tilde{w}_0+\theta_2$, $u=\tilde{u}+\theta_1$, $v=\tilde{v}$, $w=\tilde{w}+\theta_2$, and note that
	\[v(u)=v(\tilde{u}+\theta_1)=\tilde{v}(\tilde{u}),\quad w(u)=w(\tilde{u}+\theta_1)=\tilde{w}(\tilde{u})+\theta_2.\]
	{Thus}
	\bse\label{B-Holder}
	\be\label{B-Holder-1}
	\text{if}\quad u\in C^\alpha([0, T]; B_{H^2}(u_0, r)),
	\ee
	\be\label{B-Holder-2}
	\text{then}\quad (v, w)\in C^\alpha\left(\left[0, T\right]; L^2(\Omega)\times H_2^2(\Omega)\right)\cap \left\{C\left([0, T]; B_{L^2}(v_0, r)\times B_{H^2}(w_0,  r)\right)\right\},\ i.e.
	\ee
	\be\label{B-Holder-3}
	\left\|v(t+h)-v(t)\right\|_{L^2(\Omega)}\leq L_Uh^\alpha,\quad \left\|w(t+h)-w(t)\right\|_{H^2(\Omega)}\leq L_Uh^\alpha.
	\ee
	\ese
	Hence
	\[F\left(\tilde{u}\right)=\frac{1}{w}\nabla\cdot\left(w^3(\tilde{u}+\theta_1)\nabla\tilde{u}\right)-\frac{v}{w}(\tilde{u}+\theta_1)\in C^\alpha\left([0, T]; L^2(\Omega)\right). \]
	Following  these facts, we are going to show that the assertion \eqref{Max-I} of Lemma \ref{RHSMax} holds.
	
	Let $h\in(0, T]$ be such that $0\leq t<t+h\leq T$. Recall that $\left\|u(t)\right\|_{H^2(\Omega)}\leq\tilde{C_1}$, $\left\|u(t+h)\right\|_{H^2(\Omega)}\leq\tilde{C_1}$,  $\left\|v(t)\right\|_{L^2(\Omega)}\leq\tilde{C_2}$,  $\left\|v(t+h)\right\|_{L^2(\Omega)}\leq\tilde{C_2}$,  $\left\|w(t+h)\right\|_{H^2(\Omega)}\leq\tilde{C}$,\\
	$\left\|w(t)\right\|_{H^2(\Omega)}\leq\tilde{C}$.
	Here $\tilde{C}=\left\|w_0\right\|_{H^2(\Omega)}+\frac{\kappa}{2C}$, $\tilde{C_1}=\left\|u_0\right\|_{H^2(\Omega)}+\frac{\kappa}{2C}$, $\tilde{C_2}=\left\|v_0\right\|_{L^2(\Omega)}+\frac{\kappa}{2C}$.
	
	{Because $H^2(\Omega)$ is an algebra, that is estimate \eqref{C-a} of Lemma \ref{estimates}, and estimates \eqref{Holdercontinuityformular} of Corollary \ref{Holdercontinuity} and estimate \eqref{B-Holder-3} are satisfied,} we obtain
	\be\label{B8}
	\left\|[w(t+h)]^{-1}-[w(t)]^{-1}\right\|_{H^2(\Omega)}\leq C_1^2L_Uh^\alpha,\ \
	\left\|[w(t+h)]^3-[w(t)]^3\right\|_{H^2(\Omega)}\leq3\tilde{C}^2L_Uh^\alpha.
	\ee
	Similarly, for $u$ from Corollary \ref{Holdercontinuity} and \eqref{B-Holder-1}, we get
	\begin{align}
		\left\|[u(t+h)]^2-[u(t)]^2\right\|_{H^2(\Omega)}\leq2\tilde{C}_1\left[u\right]_{C^\alpha([0, T]; H^2(\Omega))}h^\alpha.\label{B8-1}
	\end{align}
	The arguments of the proof of Lemma \ref{Lip-nonlinearity} give that \eqref{Max-I} of Lemma \ref{RHSMax} holds:
	\begin{align}
		\left\|\left[F(\tilde{u})\right](t+h)-\left[F(\tilde{u})\right](t)\right\|_{L^2(\Omega)}
		\leq L_AL_Uh^\alpha+L_A[u]_{C^\alpha\left(\left[0, T\right]; H^2(\Omega)\right)}h^\alpha\label{B15}.
	\end{align}
	Here $L_A$ is a constant depending on $C$, $C_1$, $\tilde{C}$, $\tilde{C}_1$ and $\tilde{C}_2$.
	
	We next prove the assertion \eqref{Max-II} of Lemma \ref{RHSMax}.
	
	For $q\in C^\alpha([0, T]; B_{H^2}(\tilde{u}_0,r))$, $t\in [0, T]$, we note $w(t)=[w(u)](t)$, $v(t)=[v(u)](t)$. From the definition \eqref{Fre-F} of the Frech\'{e}t derivative of $F(\tilde{u})$ on $\tilde{u}$ at $t$ and the definition \eqref{linearopt} of $\mathcal{P}^*q(t)$, we have, for $h\in (0, T]$ such that $t+h\in (0, T]$,
	\begin{align}
		&\left[F'(\tilde{u})q\right](t+h)-\left[F'(\tilde{u})q\right](t)-\mathcal{P}^*\left(q(t+h)-q(t)\right)\notag\\
		=&\frac{1}{w(t+h)}\nabla\cdot\left\{[w(t+h)]^3u(t+h)\nabla q(t+h)+[w(t+h)]^3\left[\nabla u(t+h)\right]q(t+h)\right\}\notag\\
		-&\frac{1}{w(t)}\nabla\cdot\left\{[w(t)]^3u(t)\nabla q(t)+[w(t)]^3\left[\nabla u(t)\right]q(t)\right\}\notag\\
		+&\frac{3}{2w(t+h)}\nabla\cdot\left\{\left(\nabla[u(t+h)]^2\right)[w(t+h)]^2[w'(u)q](t+h)\right\}\notag\\
		-&\frac{3}{2w(t)}\nabla\cdot\left\{\left(\nabla[u(t)]^2\right)[w(t)]^2[w'(u)q](t)\right\}\notag\\
		-&\frac{[w'(u)q](t+h)}{2[w(t+h)]^2}\nabla\cdot\left([w(t+h)]^3\nabla[u(t+h)]^2\right)-\frac{v(t+h)}{w(t+h)}q(t+h)\notag\\
		+&\frac{[w'(u)q](t)}{2[w(t)]^2}\nabla\cdot\left([w(t)]^3\nabla[u(t)]^2\right)+\frac{v(t)}{w(t)}q(t)\notag\\
		-&\frac{w(t+h)[v'(u)q](t+h)-v(t+h)[w'(u)q](t+h)}{[w(t+h)]^2}u(t+h)\notag\\
		+&\frac{w(t)[v'(u)q](t)-v(t)[w'(u)q](t)}{[w(t)]^2}u(t)\notag\\
		-&\frac{1}{w_0}\nabla\cdot\left\{w_0^3u_0\nabla q(t+h)+w_0^3\left(\nabla u_0\right)q(t+h)\right\}-\frac{v_0}{w_0}q(t+h)\notag\\
		+&\frac{1}{w_0}\nabla\cdot\left\{w_0^3u_0\nabla q(t)+w_0^3\left(\nabla u_0\right)q(t)\right\}+\frac{v_0}{w_0}q(t). \label{Fre-F-4}
	\end{align}
	Because
	\begin{align}
		\left\|[w(t+h)]^3u(t+h)-[w(t)]^3u(t)\right\|_{H^2(\Omega)}
		\leq&\left\|[w(t+h)]^3-[w(t)]^3\right\|_{H^2(\Omega)}\left\|u(t+h)\right\|_{H^2(\Omega)}\notag\\
		+&\left\|[w(t)]^3\right\|_{H^2(\Omega)}\left\|u(t+h)-u(t)\right\|_{H^2(\Omega)}\notag\\
		\leq&h^\alpha3L_U\tilde{C}^2\tilde{C}_1+h^\alpha\tilde{C}^3\left\|u\right\|_{C^\alpha([0, T]; H^2(\Omega))}\label{w3u},
	\end{align}
	the algebraic properties of $H^2(\Omega)$, i.e. Lemma \ref{alg}, inequality \eqref{alg-1-2} of Corollary \ref{alg-1} and the assertion \eqref{C-a} of Lemma \ref{estimates} imply
	\begin{align}
		&\left\|\left([w(t+h)]^{-1}-[w(t)]^{-1}\right)\nabla\cdot\left\{[w(t+h)]^3u(t+h)\nabla q(t+h)\right\}\right\|_{L^2(\Omega)}\notag\\
		\leq&2C\left\|[w(t+h)]^{-1}-[w(t)]^{-1}\right\|_{H^2(\Omega)}\left\|[w(t+h)]^3u(t+h)\right\|_{H^2(\Omega)}\left\|q(t+h)\right\|_{H^2(\Omega)}\notag\\
		\leq&h^\alpha 2CC_1^2L_U\tilde{C}^3\tilde{C}_1\sup_{t\in[0, T]}\left\|q(t)\right\|_{H^2(\Omega)}\label{B20},
	\end{align}
{and}
	\begin{align}
		&\left\|[w(t)]^{-1}\nabla\cdot\left\{\left([w(t+h)]^3u(t+h)-[w(t)]^3u(t)\right)\nabla q(t+h)\right\}\right\|_{L^2(\Omega)}\notag\\
		\leq&2C\left\|[w(t)]^{-1}\right\|_{H^2(\Omega)}\left\|[w(t+h)]^3u(t+h)-[w(t)]^3u(t)\right\|_{H^2(\Omega)}\left\|q(t+h)\right\|_{H^2(\Omega)}\notag\\
		\leq&h^\alpha\left[6CC_1L_U\tilde{C}^2\tilde{C}_1+2CC_1\tilde{C}^3\left\|u\right\|_{C^\alpha([0, T]; H^2(\Omega))}\right]\sup_{t\in[0, T]}\left\|q(t)\right\|_{H^2(\Omega)}\label{Bb20}.
	\end{align}
	Hence, we deduce the estimate:
	\begin{align}
		&\left\|\frac{1}{w(t+h)}\nabla\cdot\left\{[w(t+h)]^3u(t+h)\nabla q(t+h)\right\}-\frac{1}{w(t)}\nabla\cdot\left\{[w(t)]^3u(t)\nabla q(t+h)\right\}\right\|_{L^2(\Omega)}\notag\\
		\leq&\left\|\left([w(t+h)]^{-1}-[w(t)]^{-1}\right)\nabla\cdot\left\{[w(t+h)]^3u(t+h)\nabla q(t+h)\right\}\right\|_{L^2(\Omega)}\notag\\
		+&\left\|[w(t)]^{-1}\nabla\cdot\left\{\left([w(t+h)]^3u(t+h)-[w(t)]^3u(t)\right)\nabla q(t+h)\right\}\right\|_{L^2(\Omega)}\notag\\
		\leq&h^\alpha C\left[2C_1^2L_U\tilde{C}^3\tilde{C}_1+6C_1L_U\tilde{C}^2\tilde{C}_1+2C_1\tilde{C}^3\left\|u\right\|_{C^\alpha([0, T]; H^2(\Omega))}\right]\sup_{t\in[0, T]}\left\|q(t)\right\|_{H^2(\Omega)}.\label{B22}
	\end{align}
	According to the H\"{o}lder inequality and algebraic property of $H^2(\Omega)$, i.e. Lemma \ref{alg} and estimate \eqref{w3u}, we obtain
	\begin{align}
		&\left\|\left\{\nabla\left([w(t)]^3u(t)\right)\right\}\cdot\left\{\nabla q(t+h)-\nabla q(t)\right\}\right\|_{L^2(\Omega)}\notag\\
		\leq&\left\|\nabla\left([w(t)]^3u(t)\right)\right\|_{L^4(\Omega)}\left\|\nabla q(t+h)-\nabla q(t)\right\|_{L^4(\Omega)}\notag\\
		\leq&C\left\|\nabla\left([w(t)]^3u(t)\right)\right\|_{H^1(\Omega)}\left\|\nabla q(t+h)-\nabla q(t)\right\|_{H^1(\Omega)}\notag\\
		\leq&C\left\|[w(t)]^3u(t)\right\|_{H^2(\Omega)}\left\|q(t+h)-q(t)\right\|_{H^2(\Omega)}\notag\\
		\leq&C\left\|w(t)\right\|_{H^2(\Omega)}^3\left\|u(t)\right\|_{H^2(\Omega)}\left\|q(t+h)-q(t)\right\|_{H^2(\Omega)}\notag\\
		\leq&h^\alpha C\tilde{C}^3\tilde{C}_1\left\|q\right\|_{C^\alpha([0, T];H^2(\Omega))}\label{B25},
	\end{align}
{and}
	\begin{align}
		&\left\|\left\{\nabla\left([w(t)]^3u(t)\right)-\nabla\left(w_0^3u_0\right)\right\}\cdot\left\{\nabla q(t+h)-\nabla q(t)\right\}\right\|_{L^2(\Omega)}\notag\\
		\leq&\left\|\nabla\left([w(t)]^3u(t)\right)-\nabla\left(w_0^3u_0\right)\right\|_{L^4(\Omega)}\left\|\nabla q(t+h)-\nabla q(t)\right\|_{L^4(\Omega)}\notag\\
		\leq&C\left\|\nabla\left([w(t)]^3u(t)\right)-\nabla\left(w_0^3u_0\right)\right\|_{H^1(\Omega)}\left\|\nabla q(t+h)-\nabla q(t)\right\|_{H^1(\Omega)}\notag\\
		\leq&C\left\|[w(t)]^3u(t)-w_0^3u_0\right\|_{H^2(\Omega)}\left\|q(t+h)-q(t)\right\|_{H^2(\Omega)}\notag\\
		\leq&T^\alpha C\left[3L_U\tilde{C}^2\tilde{C}_1+\tilde{C}^3\left\|u\right\|_{C^\alpha([0, T]; H^2(\Omega))}\right]\left\|q(t+h)-q(t)\right\|_{H^2(\Omega)}\notag\\
		\leq&h^\alpha T^\alpha C\left[3L_U\tilde{C}^2\tilde{C}_1+\tilde{C}^3\left\|u\right\|_{C^\alpha([0, T]; H^2(\Omega))}\right]\left\|q\right\|_{C^\alpha([0, T];H^2(\Omega))}.\label{B26}
	\end{align}
	Combining \eqref{B25}, \eqref{B26} with \eqref{B8} gives the estimate:
	\begin{align}
		&\left\|\left\{\frac{1}{w(t)}\nabla\left([w(t)]^3u(t)\right)-\frac{1}{w_0}\nabla\left(w_0^3u_0\right)\right\}\cdot\left[\nabla q(t+h)-\nabla q(t)\right]\right\|_{L^2(\Omega)}\notag\\
		\leq&C\left\|[w(t)]^{-1}-w_0^{-1}\right\|_{H^2(\Omega)}\left\|\nabla\left([w(t)]^3u(t)\right)\cdot\left\{\nabla q(t+h)-\nabla q(t)\right\}\right\|_{L^2(\Omega)}\notag\\
		+&C\left\|{w_0}^{-1}\right\|_{H^2(\Omega)}\left\|\left\{\nabla\left([w(t)]^3u(t)\right)-\nabla\left(w_0^3u_0\right)\right\}\cdot\left\{\nabla q(t+h)-\nabla q(t)\right\}\right\|_{L^2(\Omega)}\notag\\
		\leq&h^\alpha T^\alpha C\left(C_1^2L_U\tilde{C}^3\tilde{C}_1+3L_U\tilde{C}^2\tilde{C}_1\left\|{w_0}^{-1}\right\|_{H^2(\Omega)}\right)\left\|q\right\|_{C^\alpha([0, T];H^2(\Omega))}\notag\\
		+&h^\alpha T^\alpha C\tilde{C}^3\left\|{w_0}^{-1}\right\|_{H^2(\Omega)}\left\|u\right\|_{C^\alpha([0, T]; H^2(\Omega))}\left\|q\right\|_{C^\alpha([0, T];H^2(\Omega))}\label{B27}.
	\end{align}
	{Hence}
	\begin{align}
		&\left\|\left([w(t)]^2u(t)-w_0^2u_0\right)\left(\Delta q(t+h)-\Delta q(t)\right)\right\|_{L^2(\Omega)}\notag\\
		\leq&C\left\|[w(t)]^2u(t)-w_0^2u_0\right\|_{H^2(\Omega)}\left\|\Delta q(t+h)-\Delta q(t)\right\|_{L^2(\Omega)}\notag\\
		\leq&C\left(\left\|[w(t)]^2-w^2_0\right\|_{H^2(\Omega)}\left\|u(t)\right\|_{H^2(\Omega)}+\left\|u(t)-u_0\right\|_{H^2(\Omega)}\left\|w^2_0\right\|_{H^2(\Omega)}\right)\left\|q(t+h)-q(t)\right\|_{H^2(\Omega)}\notag\\
		\leq&C\left\|w(t)-w_0\right\|_{H^2(\Omega)}\left(\left\|w_0\right\|_{H^2(\Omega)}+\tilde{C}\right)\tilde{C}_1\left\|q(t+h)-q(t)\right\|_{H^2(\Omega)}\notag\\
		+&C\left\|u(t)-u_0\right\|_{H^2(\Omega)}\|w_0\|_{H^2(\Omega)}^2\left\|q(t+h)-q(t)\right\|_{H^2(\Omega)}\notag\\
		\leq&h^\alpha T^\alpha L_UC\left(\left\|w_0\right\|_{H^2(\Omega)}+\tilde{C}\right)\tilde{C}_1\left\|q\right\|_{C^\alpha([0,T];H^2(\Omega))}\notag\\
		+&h^\alpha T^\alpha C\|w_0\|_{H^2(\Omega)}^2\left\|u\right\|_{C^\alpha([0,T];H^2(\Omega))}\left\|q\right\|_{C^\alpha([0,T];H^2(\Omega))}\label{B28}.
	\end{align}
	Denote by $V_1$ a constant which is a combination of $C$, $C_1$, $\tilde{C}$, $\tilde{C}_1$,  $L_U$ and $\left\|{w_0}^{-1}\right\|_{H^2(\Omega)}$.
	Therefore, the triangle inequality, \eqref{B22}, \eqref{B27} and \eqref{B28} imply the estimate
	\begin{align}
		&\bigg\|\frac{1}{w(t+h)}\nabla\cdot\left\{[w(t+h)]^3u(t+h)\nabla q(t+h)\right\}-\frac{1}{w(t)}\nabla\cdot\left\{[w(t)]^3u(t)\nabla q(t)\right\}\notag\\
		&-\frac{1}{w_0}\nabla\cdot\left\{w_0^3u_0\nabla q(t+h)\right\}+\frac{1}{w_0}\nabla\cdot\left\{w_0^3u_0\nabla q(t)\right\}\bigg\|_{L^2(\Omega)}\notag\\
		\leq&\left\|\frac{1}{w(t+h)}\nabla\cdot\left\{[w(t+h)]^3u(t+h)\nabla q(t+h)\right\}-\frac{1}{w(t)}\nabla\cdot\left\{[w(t)]^3u(t)\nabla q(t+h)\right\}\right\|_{L^2(\Omega)}\notag\\
		+&\left\|\left(\frac{1}{w(t)}\nabla\left\{[w(t)]^3u(t)\right\}-\frac{1}{w_0}\nabla\left(w_0^3u_0\right)\right)\cdot\left(\nabla q(t+h)-\nabla q(t)\right)\right\|_{L^2(\Omega)}\notag\\
		+&\left\|\left([w(t)]^2u(t)-w_0^2u_0\right)\left(\Delta q(t+h)-\Delta q(t)\right)\right\|_{L^2(\Omega)}\notag\\
		\leq&h^\alpha V_1\sup_{t\in[0, T]}\left\|q(t)\right\|_{H^2(\Omega)}+h^\alpha V_1\left\|u\right\|_{C^\alpha\left([0, T];H^2(\Omega)\right)}\sup_{t\in[0, T]}\left\|q(t)\right\|_{H^2(\Omega)}\notag\\
		+&h^\alpha T^\alpha V_1\left\|u\right\|_{C^\alpha([0,T];H^2(\Omega))}\left\|q\right\|_{C^\alpha\left([0,T];H^2(\Omega)\right)}+h^\alpha T^\alpha V_1\left\|q\right\|_{C^\alpha\left([0, T];H^2(\Omega)\right)}\label{B0}.
	\end{align}
	Similarly,
	\begin{align}
		&\bigg\|\frac{1}{w(t+h)}\nabla\cdot\left\{[w(t+h)]^3q(t+h)\nabla u(t+h)\right\}-\frac{1}{w(t)}\nabla \cdot\left\{[w(t)]^3q(t)\nabla u(t)\right\}\notag\\
		&-\frac{1}{w_0}\nabla\cdot\left\{w_0^3q(t+h)\nabla u_0\right\}+\frac{1}{w_0}\nabla\cdot\left\{w_0^3q(t)\nabla u_0\right\}\bigg\|_{L^2(\Omega)}\notag\\
		\leq&h^\alpha V_1\sup_{t\in[0, T]}\left\|q(t)\right\|_{H^2(\Omega)}+h^\alpha V_1\left\|u\right\|_{C^\alpha\left([0, T];H^2(\Omega)\right)}\sup_{t\in[0, T]}\left\|q(t)\right\|_{H^2(\Omega)}\notag\\
		+&h^\alpha T^\alpha V_1\left\|q\right\|_{C^\alpha([0, T]; H^2(\Omega))}+h^\alpha T^\alpha V_1\left\|u\right\|_{C^\alpha([0, T]; H^2(\Omega))}\left\|q\right\|_{C^\alpha([0, T]; H^2(\Omega))}.\label{B30-1}
	\end{align}
	Set $\tilde{C}_3=C\left[C_1L_U+\tilde{C}_2C_1^2L_U\right]$, $\tilde{C}_4=CC_1\left[L_U+\left\|v_0\right\|_{L^2(\Omega)}\left\|w_0^{-1}\right\|_{H^2(\Omega)}\right]$. The triangle inequality, algebraic properties of Sobolev spaces, i.e. \eqref{alg-1-1} of Corollary \ref{alg-1}, \eqref{C-a} of Lemma \ref{estimates}, and the assertion \eqref{Holdercontinuityformular} of Corollary \ref{Holdercontinuity} imply the estimate
	\begin{align}
		\left\|-\frac{v(t+h)}{w(t+h)}q(t+h)+\frac{v(t)}{w(t)}q(t)-\frac{v_0}{w_0}q(t+h)+\frac{v_0}{w_0}q(t)\right\|_{L^2(\Omega)}
		\leq& h^\alpha \tilde{C}_3\sup_{t\in[0, T]}\left\|q(t)\right\|_{H^2(\Omega)}\notag\\
		+&h^\alpha T^\alpha \tilde{C}_4\left\|q\right\|_{C([0, T]; H^2(\Omega))}\notag.
	\end{align}
	Denote by $V_2$ a constant which is a combination of $C_2$ $C_1$, $C$, $\tilde{C}$, $\tilde{C}_1$, $\tilde{C}_2$, $L_U$, $L_W$, $L_M$ and $T_0^{1-\alpha}$. We combine \eqref{bound-Fre-W}, \eqref{Holder-Frechet-W-I} in Corollary \ref{Holder-Frechet-W-I-Cor}, \eqref{Holdercontinuityformular} in Corollary \ref{Holdercontinuity} with above arguments for estimate \eqref{B0} and similarly deduce
	\begin{align}
		&\frac{3}{2}\bigg\|\frac{1}{w(t+h)}\nabla\cdot\left\{[w(t+h)]^2[w'(u)q](t+h)\nabla[u(t+h)]^2\right\}\notag\\
		&-\frac{1}{w(t)}\nabla\cdot\left\{[w(t)]^2[w'(u)q](t)\nabla[u(t)]^2\right\}\bigg\|_{L^2(\Omega)}\notag\\
		\leq&h^\alpha T^\alpha V_2\left\|q\right\|_{C^\alpha([0, T]; H^2(\Omega))}
		+h^\alpha V_2\left[1+\left\|u\right\|_{C^\alpha([0, T]; H^2(\Omega))}\right]\sup_{t\in[0, T]}\left\|q(t)\right\|_{H^2(\Omega)},\notag
	\end{align}
	\begin{align}
		&\left\|\frac{[w'(u)q](t+h)}{2[w(t+h)]^2}\nabla\cdot\left\{[w(t+h)]^3\nabla[u(t+h)]^2\right\}-\frac{[w'(u)q](t)}{2[w(t)]^2}\nabla\cdot\left\{[w(t)]^3\nabla[u(t)]^2\right\}\right\|_{L^2(\Omega)}\notag\\
		\leq&\displaystyle h^\alpha V_2\left[1+\|u\|_{C^\alpha\left(\left[0, T\right]; H^2(\Omega)\right)}\right]\sup_{t\in[0, T]}\|q(t)\|_{H^2(\Omega)}
		+h^\alpha T^\alpha V_2\left\|q\right\|_{C^\alpha([0, T]; H^2(\Omega))},\notag
	\end{align}
{and}
	\begin{align}
		&\bigg\|-\frac{w(t+h)[v'(u)q](t+h)-v(t+h)[w'(u)q](t+h)}{[w(t+h)]^2}u(t+h)\notag\\
		&+\frac{w(t)[v'(u)q](t)-v(t)[w'(u)q](t)}{[w(t)]^2}u(t)\bigg\|_{L^2(\Omega)}\notag\\
		\leq&h^\alpha V_2\left[1+\left\|u\right\|_{C^\alpha([0, T]; H^2(\Omega))}\right]\sup_{t\in[0, T]}\left\|q(t)\right\|_{H^2(\Omega)}+T^\alpha h^\alpha V_2\left\|q\right\|_{C^\alpha([0, T]; H^2(\Omega))}\notag.
	\end{align}
	Consequently, by setting $L_B=V_1+V_2+\tilde{C}_3+\tilde{C}_4$, we obtain
	\begin{align}
		&\left\|\left[F'(\tilde{u})q\right](t+h)-\left[F'(\tilde{u})q\right](t)-\mathcal{P}^*\left[q(t+h)-q(t)\right]\right\|_{L^2(\Omega)}\notag\\
		\leq&h^\alpha L_B\sup_{t\in[0, T]}\left\|q(t)\right\|_{H^2(\Omega)}+h^\alpha L_B\left\|u\right\|_{C^\alpha\left([0, T]; H^2(\Omega)\right)}\sup_{t\in[0, T]}\left\|q(t)\right\|_{H^2(\Omega)}\notag\\
		+&h^\alpha T^\alpha L_B\left\|q\right\|_{C^\alpha([0, T];H^2(\Omega))}+h^\alpha T^\alpha L_B\left\|u\right\|_{C^\alpha([0, T]; H^2(\Omega))}\left\|q\right\|_{C^\alpha([0, T]; H^2(\Omega))}\notag.
	\end{align}
	Hereby, \eqref{Max-II} is proved and this concludes the proof of Lemma \ref{RHSMax}.
\end{proof}

\vspace*{-0.4cm}
\bibliographystyle{abbrv}
\bibliography{Bibliography4}

\end{document}